\documentclass[11pt]{article_ET_m}

\usepackage[dvipdf]{graphicx}
\usepackage{latexsym}
\usepackage{amsmath,amssymb,amsfonts,amsthm}
\usepackage{xcolor}
\usepackage{mathrsfs}
\usepackage{bbm}
\usepackage{bm}
\usepackage{enumitem}
\usepackage{mathtools}

\usepackage{graphicx}
\usepackage{amstext}
\usepackage[french,english]{babel}
\usepackage[utf8]{inputenc}
\usepackage[T1]{fontenc}
\usepackage[numbers,comma,sort]{natbib}

\usepackage{pdfsync}

\numberwithin{equation}{section}

\definecolor{db}{RGB}{0, 0, 130}
\usepackage[colorlinks=true,citecolor=red,linkcolor=db,urlcolor=blue,pdfstartview=FitH]{hyperref}

\newcommand{\R}{\mathbb R}
\newcommand{\Z}{\mathbb Z}
\newcommand{\N}{\mathbb N}

\newcommand{\E}{\mathbb E}
\newcommand{\PP}{\mathbb P}

\newcommand{\T}{\mathbb{T}}
\newcommand{\dd}{\, \text{d}}
\newcommand{\bchi}{\boldsymbol\chi}

\newcommand{\calD} {\ensuremath {\mathcal{D}}}

\newcommand{\rate}{\varrho_{N}}
\newcommand{\rateone}{\varrho_{1,N}}
\newcommand{\ratetwo}{\varrho_{2,N}}
\newcommand{\ratetilde}{\widetilde{\varrho}_{N}}

\DeclareMathOperator{\dive}{div}

\numberwithin{equation}{section}
\newtheorem{theorem}{Theorem}[section]
\newtheorem{assumption}{Assumption}
\newtheorem{proposition}[theorem]{Proposition}
 \newtheorem{remark}[theorem]{Remark}
\newtheorem{lemma}[theorem]{Lemma}

\newtheorem{definition}[theorem]{Definition}

\usepackage{todonotes}

\newcommand{\forxvt}{for all $x\in\T^d$, $v\in\R^d$ and $t\geq 0$}
\newcommand{\dv}{\textbf{v}}

\author{Ludovic Goudenège\footnote{Universit\'e Paris-Saclay \'Evry and CNRS UMR-8071, France. \texttt{ludovic.goudenege@math.cnrs.fr}.} \and
Christian Olivera\footnote{Departamento de Matem\'atica, Universidade Estadual de Campinas, Brazil. \texttt{colivera@ime.unicamp.br}.} \and
Gabriela Planas\footnote{Departamento de Matem\'atica, Universidade Estadual de Campinas, Brazil. \texttt{gplanas@unicamp.br}.}\and
Alexandre Richard\footnote{Universit\'e Paris-Saclay, CentraleSup\'elec and CNRS FRA-3487, France. \hfill ~\newline  \texttt{alexandre.richard@centralesupelec.fr}.}
}

\title{Quantitative approximation of the Vlasov(--Fokker--Planck)--Navier--Stokes system by stochastic particle systems}
\begin{document}

\date{\today}

\maketitle

\begin{abstract}
This paper is concerned with a fluid-particle system given by the incompressible Navier--Stokes equations coupled with the Vlasov(--Fokker--Planck) equation through a drag force. Such a model arises naturally in the study of aerosols, sprays, and more generally two-phase flows. In dimensions $d\in \{2,3\}$, we establish a rate of convergence for a system of $N$ interacting stochastic particles coupled with a fluid, towards the Vlasov(--Fokker--Planck)--Navier--Stokes system, as $N\to \infty$. The case of particles with a noise that vanishes as $N\to \infty$ is considered and leads specifically to the Vlasov--Navier--Stokes system.
More precisely, we prove that the empirical measure associated with the particle system converges to the Vlasov(--Fokker--Planck) component, while the fluid velocity converges to the Navier--Stokes component of the coupled system. The proofs combine stochastic calculus and PDE techniques to establish energy estimates and commutator estimates for both the discrete and continuous systems.
 \end{abstract}


\vspace{0.3cm} \noindent {\bf MSC2020 subject classification:} 60H15, %
 35R60, %
 35Q30, %
 35F10, %
 60H30. %

\setcounter{tocdepth}{2}
\renewcommand\contentsname{}
\vspace{-1cm}

\tableofcontents

\section {Introduction} \label{Intro}

The aim of this work is to study a fluid-particle system which, in the limit of infinitely many particles, converges to a Vlasov(--Fokker--Planck)--Navier--Stokes system. Such systems model the behaviour of particles suspended in a moving fluid and capture the complex interactions that arise in two-phase flows. A canonical example is the motion of a particle suspension in a dense fluid, where the dynamics of the fluid is governed by the Navier--Stokes or Euler equations, and the coupling between fluid and particles is realised through a drag force. These systems arise naturally in the study of thin sprays (as opposed to thick sprays, for which the interaction
also occurs through the volume fraction occupied by the fluid) and aerosols; see~\cite{BernardEtAl,Baranger, BGGMMN, Gemci}.

In the incompressible setting, when the fluid-particle interaction is modelled by the Stokes drag force, the Vlasov(--Fokker--Planck)--Navier--Stokes system reads
\begin{equation}\label{eq:PDE}
\begin{cases}
&\partial_t u_{t}(x) - \Delta u_{t}(x) + (u_{t} \cdot \nabla) [u_{t}](x) + \nabla p_t(x)
+\displaystyle\int_{\R^d} (u_{t}(x)-v') F_{t}(x,v') \dd v'=0,
\quad t>0,~x\in \T^{d},\\
& \dive u_{t}(x) =0 , \quad t>0,~ x\in \T^{d}, \\
& u_{0}(x) = u^\circ(x), \quad x\in \T^{d},\\
&\partial_{t} F_{t}(x,v) + v \cdot \nabla_{x} F_{t}(x,v)
+ \dive_v\!\left((u_{t}(x)-v) F_{t}(x,v)\right)
= \frac{\sigma^{2}}{2} \Delta_v F_{t}(x,v) ,
\quad t>0, \ (x,v)\in \T^d\times \R^{d},\\
& F_{0}(x,v) = F^\circ(x,v), \quad (x,v)\in \T^{d}\times \R^d,
\end{cases}
\end{equation}
where $u$ and $p$ denote the fluid velocity and pressure, $F$ is the particle distribution function, $u^\circ \colon \T^d \to \R^d$ (with $d \in \{2,3\}$) is the initial fluid velocity, $F^\circ$ is the initial particle density in phase space $\T^d \times \R^d$, and $\sigma \geq 0$ is a diffusion parameter. We refer the reader to~\cite{Boudin} for a discussion of several physical extensions and applications of this model. A distinctive feature of this system is the coupling between the Vlasov(--Fokker--Planck) equation and the Navier--Stokes component through the term $(u_{t}-v) F_{t}$, which accounts for the action of particles (represented by their density $F$) on the fluid, and reciprocally. We will see that this term is derived from Stokes' law, which at the level of microscopic particles is responsible for a drag force due to the viscosity of the fluid.

\bigskip

The analysis of Vlasov(--Fokker--Planck)--Navier--Stokes systems has attracted considerable attention over the past three decades, giving rise to an extensive literature. The Cauchy theory for system~\eqref{eq:PDE} -- addressing global existence of weak solutions in dimensions $d=2$ and $d=3$ -- has been developed in \emph{e.g.}~\cite{Boudin2009, ChaeKangLee, Flandoli2, Yu}.
Global existence of smooth solutions with small initial data for the Vlasov--Fokker--Planck--Navier--Stokes system in dimension $3$ was first established in~\cite{GoudonHeMoussaZhang}, while a uniqueness result is given in~\cite{Choi_2015}. 
More recent results on existence and uniqueness of weak solutions in dimension $2$ were obtained in~\cite{HanKwan} for $\sigma = 0$ and in~\cite{Flandoli2} for $\sigma > 0$. 

\bigskip

The main objective of this paper is to derive the system~\eqref{eq:PDE}, for $d \in \{2,3\}$ and $\sigma \geq 0$, as the mean-field limit of a system of kinetic stochastic particles. The microscopic model we consider consists of the following system of stochastic differential equations driven by independent $(\mathcal{F}_{t})$-Brownian motions $(B^i)_{i\in \N}$ and coupled with a PDE for the fluid $u^N$:
\begin{equation}\label{Part}
\begin{cases}
& \displaystyle \partial_t u_{t}^{N}(x) - \Delta u_{t}^{N}(x)
+ (u_{t}^{N} \cdot \nabla) \left[\bchi_A(u_{t}^{N})\right](x) + \nabla p^N_{t}(x) \\
&\hspace{1.5cm}
+\displaystyle \frac{1}{N} \sum_{i=1}^{N}
\big(\bchi_A\big(u_{t}^{N}(X_{t}^{i,N})\big)-V_{t}^{i,N}\big) \delta_{X^{i,N}_t}^{N}(x)=0,
\quad t>0,~ x\in \T^{d},\\
& \dive u_t^N(x) =0 , \quad t>0,~ x\in \T^{d}, \\
& u^{N}_{0}(x) = u^{\circ,N}(x), \quad x\in\T^{d},\\
& \dd X^{i,N}_t= V^{i,N}_t \,\dd t, \quad t>0,~ i\in \{1,\dots, N\},\\
& \dd V^{i,N}_t=
\big(\bchi_A\big(u_{t}^N(X_{t}^{i,N})\big)-V_{t}^{i,N}\big) \dd t + \sigma_{N} \,\dd B_{t}^{i},
\quad t>0,~ i\in \{1,\dots, N\},
\end{cases}
\end{equation}
where $N$ is the total number of particles, $(X^{i,N}_t, V^{i,N}_t)$ denote the position and velocity of the $i$th particle, for each $N$ we assume $\sigma_{N}>0$, and $\bchi_A$ is a smooth cut-off function controlled by a positive parameter $A$ to be fixed later. The initial positions and velocities are denoted by $(X^{i,N}_{0}, V^{i,N}_{0})$, $i \in \{1,\dots,N\}$. The equations for the fluid velocity and pressure $(u^N, p^N)$ take the form of the incompressible Navier--Stokes equations with a discrete particle-fluid interaction term; as mentioned previously, this term models the Stokes drag force. In the fluid equation, the mollified Dirac delta function $\delta^N$ is used both for mathematical convenience and to account for the idea that particles are not mere pointwise objects but have a volume of interaction with the fluid.
For a more detailed phenomenological derivation of this model, we refer to Flandoli, Leocata and Ricci~\cite{Flandoli2}.

We study the convergence of the mollified empirical measure
\begin{equation}\label{eq:defF}
F_{t}^{N}(x,v) = (S_{t}^{N} \ast \vartheta^N)(x,v)
\end{equation}
associated to the empirical measure of the particle system
\begin{equation}
\label{eq:depempmeas}
S_{t}^{N} = \frac{1}{N} \sum_{i=1}^N \delta_{(X^{i,N}_t, V^{i,N}_t)},
\end{equation}
where $\delta_a$ denotes the Dirac mass at $a$ and $\vartheta^N$ is a mollifier defined in the next section.
Our goal is to show that $F_t^N$ and $u_t^N$ converge as $N \to \infty$ to $F$ and $u$, which are the solution of system~\eqref{eq:PDE}.

\bigskip

The derivation of hydrodynamic equations from microscopic particle systems is a classical and central problem in mathematical physics, which has been approached by a variety of methods over the past century; see for instance~\cite{Gallagher, JabinWang, Lanford, Pulvirenti,Oelschlger1991, Saint, Serfaty}. The specific problem of deriving the Vlasov(--Fokker--Planck)--Navier--Stokes system~\eqref{eq:PDE} from a microscopic description based on stochastic particle systems remains,
however, largely open, and existing results are still fragmentary. The works~\cite{Gou, Gou2} address PDE-to-PDE convergence, deriving the incompressible Navier--Stokes equations from the Vlasov--Navier--Stokes system; for further developments in this direction, see~\cite{HanKwanMichel}. A related line of work uses homogenization techniques to derive systems of the type~\eqref{eq:PDE}; see \emph{e.g.}~\cite{Alla,Golse,CarrapatosoHillairet} for the derivation of the Navier--Stokes equations with Brinkman's correction term. The quantitative derivation of nonlinear Fokker--Planck equations from systems of moderately interacting stochastic particles has been initiated by Oelschl\"ager~\cite{Oelschlager87} and a long list of works followed, see \emph{e.g.}~\cite{ORT} where convergence rates are established in the case of direct particle interactions through a singular convolution kernel. A key distinction of the present work is that the particles do not interact directly with one another, but only indirectly through their coupling with the fluid. In the context of stochastic particle systems closest to ours, Flandoli et al.~\cite{Flandoli2} studied model~\eqref{Part} in the case $d=2$, $\sigma > 0$, and proved convergence of the empirical measure toward~\eqref{eq:PDE} without quantitative rates.

The present paper provides the first derivation of system~\eqref{eq:PDE} via stochastic interacting particle systems that covers both dimensions $d \in \{2,3\}$ and the full range $\sigma \geq 0$, and yields explicit convergence rates. Our first main result establishes convergence in Bessel norm  of regularity $\gamma$ and integrability $p$, denoted $ \lVert\cdot \rVert_{\gamma,p}$, for $u^N \to u$ and in weighted $L^2$ norm for $F^N \to F$, and can be summarised informally as follows.

\medskip

\noindent\textbf{Main result.}
Under assumptions on the initial conditions of the system~\eqref{eq:PDE} and the boundedness of $u$ (see Assumption~\ref{assump-PDE}), and on the parameters of the particle system (see Assumption~\ref{assump-particles}),
 for $p \in (d, +\infty]$, $\gamma \in (\tfrac{d}{p}, 1)$, an integer $k \geq 3$ and for any $q \in [1,+\infty)$, then for any~$N \in \N^*$,
\begin{align*}
\Bigl(\E &\sup_{t\in[0,T]} \big\| u^{N}_t- u_t\big\|_{ \gamma,p}^{2q}\Bigr)^{\frac{1}{q}}
+  \Bigl(\E \sup_{t\in[0,T]} \big\| \langle v \rangle^{k}(F_{t}^{N}-F_{t})
\big\|_{L^{2}_{x,v}}^{2q} \Bigr)^{\frac{1}{q}} \\
& \leq C_{N} \Bigl( \bigl(\E \big\| u^{\circ,N}-u^\circ \big\|_{\gamma, p}^{2q}\bigr)^{\frac{1}{q}}
+   \bigl(\E \big\| \langle v \rangle^{k} (F_{0}^{N}-F^\circ)
\big\|_{L^{2}_{x,v}}^{2q}\bigr)^{\frac{1}{q}} + \rho_{N} \Bigr),
\end{align*}
where $C_{N} \equiv C(N,q,d,\gamma,p,A,\sigma)$ depends on $N$ (explicitly) only when $\sigma=0$; and $\rho_{N} > 0$ depends explicitly on the parameters of the problem, as specified in Theorem~\ref{th:convBessel}. In practice, the rate of convergence is algebraic when $\sigma>0$ and is logarithmic when $\sigma=0$.
The detailed statement is given in Theorem~\ref{th:convBessel} and a discussion on the rate of convergence is carried out in Remarks~\ref{rk:rhoN} and ~\ref{rk:rate-init}, see Section~\ref{subsec:Main}. Our second main result, Theorem~\ref{th:convenergy},  establishes convergence of $u^N$ toward $u$ in the following energy norm, under essentially the same assumptions and with the same rate:
\begin{align*}
\Bigl(\E & \sup_{t\in[0,T]} \big\|u^{N}_t-u_t\big\|_{L^2_{x}}^{2q}\Bigr)^{\frac{1}{q}}
+ \frac{1}{2} \Bigl(\E \Bigl|\int_{0}^{T}
\big\| \nabla (u_{s}^{N}- u_{s})\big\|_{L^2_{x}}^{2} \,\dd s \Bigr|^{q}\Bigr)^{\frac{1}{q}}\\
& \leq C_{N} \Bigl( \bigl(\E \big\| u^{\circ,N}-u^\circ \big\|_{L^2_{x}}^{2q}\bigr)^{\frac{1}{q}}
+ \bigl(\E \big\| \langle v \rangle^{k} (F_{0}^{N}- F^\circ)\big\|_{L^{2}_{x,v}}^{2q}\bigr)^{\frac{1}{q}} + \rho_{N} \Bigr).
\end{align*}
An interesting consequence of the first result is the quantitative propagation of chaos of the particle system. While for fixed $N$, the particles are correlated due to their interaction through the fluid, it appears that at the limit $N\to \infty$, they behave like independent kinetic particles in the fluid $u$ given by \eqref{eq:PDE}. The trajectorial rate of convergence between a particle in the finite particle system and a particle with the same noise in the infinite particle system is bounded by
\begin{equation*}
|\sigma-\sigma_{N}| + \E \sup_{t\in[0,T]} \big\| u^{N}_t- u_t \big\|_{ \gamma,p} ,
\end{equation*}
see Theorem~\ref{th:strongPoC}.

\paragraph{Organisation of the paper.}
In Section~\ref{sec:main}, we introduce first the notations that will be used throughout the paper; then we state the assumptions of the paper and the three main results about the convergence in Bessel norm, in energy norm and the propagation of chaos.
The remaining sections are devoted to the proofs of these results. Let us explain here how they are organised. In Section~\ref{sec:prelim}, we recall the definition of weak solution to \eqref{eq:PDE} and gather from the literature some useful properties of such solutions that hold under our set of assumptions. We also establish some results at the level of the PDE that seem to be new: under the assumption that $u^\circ$ is in the Bessel space $H^\gamma_{p}$ of positive regularity, then $u$ satisfies a mild formulation of \eqref{eq:PDE} and we deduce some regularity on $u$.
Then considering the system~\eqref{Part}, we claim in Section~\ref{sec:fluid-particle} that a weak solution $u^N$ exists for almost every $\omega$, and that the solution $(u^N,X^{1,N},V^{1,N},\dots,X^{N,N},V^{N,N})$ is adapted to the Brownian filtration. The proof is then given in Appendix~\ref{app:proof-prop-IPS}. In the rest of Section~\ref{sec:fluid-particle}, we establish \emph{a priori} bounds on the weighted $L^2$ norm of $F^N$ and on the Bessel norm of $u^N$. Although the cut-off $\bchi_{A}$ in \eqref{Part} is convenient to obtain bounds on $F^N$ independently of those on $u^N$ (and \emph{vice versa}), we remark that the main reason we had to impose their presence is that we were not able to prove the existence of solution to ~\eqref{Part} without cut-off -- we thus leave this problem open.
In Section~\ref{sec:FN-F}, we introduce for each $N$ an auxiliary PDE which corresponds to~\eqref{eq:PDE} with $\sigma$ replaced by $\sigma_{N}$ and with a cut-off function applied on $u$: the solution $(u^{(N)},F^{(N)})$ is compared to $(u^N,F^N)$, and we first obtain a bound on $F^N-F^{(N)}$ that depends on $u^N-u^{(N)}$, as expected due to the coupling between the fluid and the particles. The (deterministic) distance between $(u^{(N)},F^{(N)})$ and $(u,F)$ is evaluated in Section~\ref{subsec:discuss-rateregul} when $\sigma>0$ and in Section~\ref{subsec:discuss-rateregul-0} when $\sigma=0$. 
Finally in Section~\ref{sec:proofs}, we express $u^N-u^{(N)}$ and bound it by several terms: some commutators give a rate of convergence while other terms involve time integrals of $u^N-u^{(N)}$ and $F^N-F^{(N)}$. Thanks to the bound on $F^N-F^{(N)}$ from Section~\ref{sec:FN-F}, we can proceed with a Gr\"onwall argument leading successively to the proofs of Theorem~\ref{th:convBessel}, Theorem~\ref{th:convenergy} and Theorem~\ref{th:strongPoC}.

\paragraph{Acknowledgments.}
This work is supported by the SDAIM project jointly funded by the S\~ao Paulo Research Foundation (FAPESP) and the French National Research Agency (ANR) through the grants  ANR-22-CE40-0015 and 22/03379-0. G. Planas was partially supported by CNPq-Brazil grant 310274/2021-4.  C. Olivera was partially supported  by  CNPq by the grant $422145/2023-8$.

\section{Notations and main results}
\label{sec:main}

\subsection{Notations and definitions}\label{subsec:notations}

\paragraph{Mollifiers.}
We construct a mollifying sequence $(\vartheta^N)_{N\in \N^*}$ on $\T^d\times \R^d$ as follows. Let $\alpha,\beta\in (0,1]$.
\begin{itemize}
\item Let $\vartheta^1:\R^d\to \R_{+}$ be a smooth density function with support in $(-\frac{1}{2},\frac{1}{2})^d$ such that
\begin{equation}\label{hyp:theta1}
|\nabla_x \vartheta^{1}|\lesssim \vartheta^{1} .
\end{equation}
For any $x\in [-\frac{1}{2},\frac{1}{2})^d$, define
\begin{equation}\label{eq:defTheta0}
\vartheta^{1,N}_{0}(x) \coloneqq N^{d\beta} \vartheta^{1}(N^\beta x),
\end{equation}
and let $\vartheta^{1,N}:\T^d \to \R_{+}$  be the periodisation of $\vartheta^{1,N}_{0}$.

\item Let $\vartheta^2:\R^d\to \R_{+}$ be a smooth density function such that
\begin{equation}\label{hyp:theta2}
 \text{supp}(\vartheta^2) \subset B(0,1) \quad \text{and} \quad \int_{\R^d} \vartheta^2(v) v \dd v = 0,
\end{equation}
and define
\begin{equation}\label{eq:defTheta2N}
 \vartheta^{2,N}(v) \coloneqq N^{d\alpha} \vartheta^{2}(N^{\alpha} v).
\end{equation}
\end{itemize}
Now $(\vartheta^N)_{N\in \N^*}$ is defined by
\begin{equation}\label{eq:defThetaN}
\vartheta^N(x,v) \coloneqq \vartheta^{1,N}(x) \vartheta^{2,N}(v) .
\end{equation}
We introduce the mollified delta Dirac function
\begin{equation}
\label{eq:mollifieddelta}
\delta_{X^{i,N}_t}^{N}(x)= \vartheta^{1,N}(x - X^{i,N}_t).
\end{equation}

\paragraph{Cut-Off function.}
For  $A>0$, let
 $\chi_A: \R \to \R$ be a $\mathcal{C}_b^2(\R)$ function such that
\begin{enumerate}
\item $\chi_A(x)=x$, for $x\in [-A,A]$,
\item $\chi_A(x)= A$, for $x> 1+A$ and $\chi_A(x)= -A$, for $x< -(1+A)$,
\item $\|\chi_{A}'\|_\infty \leq 1$ and $\|\chi_{A}''\|_\infty <\infty$.
\end{enumerate}
As a consequence, $\|\chi_A\|_\infty\leq 1+A$.
Now define the smooth cut-off function $\bchi_{A}:\R^d\to \R^d$ by
\begin{equation}\label{def:cutoff}
\bchi_{A}  : x \mapsto \left(\chi_{A}(x_{1}) ,\dots, \chi_{A}(x_{d})\right) .
\end{equation}
To make notations easier, we will simply write $\bchi$ in the sequel unless we need to emphasise on the dependence on $A$.

\paragraph{Heat Kernel.}

 In this paper, $(e^{t\Delta })_{t\geq 0 }$ denotes the semigroup of the heat operator on $\T^d$.
 That is, for $f \in {L}^p(\T^d)$,
\begin{equation*}
e^{t \Delta}f (x)
 = g_{2t} \ast_{} f(x),
\end{equation*}
where $\ast$ denotes the convolution on $\T^d$ and for any $t>0$, $g_{t}$ is the heat kernel given, for all $x \in \T^d$, by
\begin{equation*}
g_{t}^{}(x) =\frac{1}{(2\pi t)^{d/2}} \sum_{k\in \mathbb{Z}^{d}} e^{-\frac{| x-k |^{2}}{2t}}.%
\end{equation*}

\paragraph{Bessel spaces.}

 Let $\calD(\mathbb{T}^{d})$ be the collection of all infinitely differentiable functions on
$\mathbb{T}^{d}$. Then $\calD^{\prime}(\mathbb{T}^{d})$ stands for the topological dual of
$\calD(\mathbb{T}^{d})$. We denote the Fourier coefficients of $u\in \calD^{\prime}(\mathbb{T}^{d})$ by $\widehat{u}(k)\coloneqq \frac{1}{(2\pi)^{d/2}} (u,e^{2i\pi \langle k,\cdot\rangle})$. For any $\gamma\in \mathbb R$, we define the Bessel potential operator $(I-\Delta)^{\gamma/2}$ applied to $u\in \calD^{\prime}(\T^d)$ by
\begin{equation*}
(I-\Delta)^{\frac{\gamma}{2}} u(x)  =  \frac{1}{(2\pi)^{d/2}} \sum_{k\in \Z^d} (1+|k|^2)^{\frac{\gamma}{2}} \, \widehat{u}(k)\, e^{-2i\pi \langle k,x\rangle},
\end{equation*}
and denote
\begin{align*}%
\|u\|_{\gamma,p} =  \big\| (I-\Delta)^{\frac{\gamma}{2}} u\big\|_{L^p(\T^d)}.
\end{align*}
As in \cite[p.168]{SchmeisserTriebel}, $H_{p}^{\gamma}(\mathbb{T}^{d})$ is
defined for $p\in\left(1,\infty\right)$ and $\gamma\in \mathbb R$ by
\begin{align*}
H_{p}^{\gamma}(\mathbb{T}^{d}) \coloneqq \Big\{ u \in \calD^{\prime}(\mathbb T^d) ;\, \|u\|_{\gamma,p}<\infty \Big\}.
\end{align*}
When $ p=2 $ we use the standard notation $ H^\gamma_2(\mathbb{T}^{d}) = H^\gamma(\mathbb{T}^{d})$.

\paragraph{Heat estimates.}
For any $n\in \N$ and $p\in [1,+\infty]$, we recall that
\begin{align*}%
\lVert \nabla^n g_{t} \rVert_{L^p(\T^d)} \lesssim  t^{-\frac{1}{2}(n+d(1-\frac{1}{p}))} , \quad \text{for any } t>0.
\end{align*}
It follows by interpolation \cite[Theorem 1~(ii), p.173]{SchmeisserTriebel}
and a convolution inequality that for any $n\in \N$, $p\in (1,+\infty)$, $\gamma\geq 0$, for any $1\leq r\leq p$ and any measurable $f$,
\begin{align}\label{eq:Bessel-heat-estimate}
\lVert (I-\Delta)^\frac{\gamma}{2} \nabla^n e^{t\Delta} f \rVert_{L^p(\T^d)} \lesssim  (1\wedge t)^{-\frac{1}{2}\big(\gamma+n + d(\frac{1}{r}-\frac{1}{p})\big)} \lVert f \rVert_{L^r(\T^d)} , \quad \text{for any } t>0.
\end{align}

\paragraph{The Leray projector.}
We denote by $P$ the orthogonal projection from $L^2(\T^d;\R^d)$ onto the space of divergence-free vector fields $\{f\in L^2(\T^d;\R^d):~ \langle k, \widehat f(k) \rangle=0, \ \forall k\in \Z^d\}$. Based on the Helmholtz decomposition \cite[Theorem 2.6]{RobinsonEtAl}, one can check that
\begin{equation*}
Pf(x) = (2\pi)^{d/2} \hat{f}(0) +  (2\pi)^{d/2} \sum_{k\in \Z^d, k\neq 0} \Big(\widehat f(k) - \frac{\langle k, \widehat f(k) \rangle}{|k|^2} k \Big) e^{-2i\pi \langle k,x\rangle}.
\end{equation*}
Based on the previous expression, one obtains the following properties:
\begin{itemize}
\item For any $p\in (1,+\infty)$, $\lVert Pf \rVert_{L^p(\T^d)} \lesssim \lVert f \rVert_{L^p(\T^d)}$, see \cite[Theorem 2.28]{RobinsonEtAl};

\item The Leray projector commutes with derivatives;

\item The Leray projector commutes with the heat semigroup;

\item From the Helmholtz orthogonal decomposition, we have for any $f\in H^1(\T^d)$, $P\nabla f=0$.

\end{itemize}

\paragraph{Moments of the density of particles.}
Let $f$ be a nonnegative real function defined on $\T^d\times\R^d$.
Let $k\in \mathbb{N}$ and define the $k$-th moments in velocity of $f$, for all $x\in\T^d$, by
\begin{align}\label{eq:defMoments}
 m_{k}(f)(x) = \int_{\R^d} |v|^{k} f(x,v) \dd v,
\end{align}
when it exists, and define the $k$-th moments of $f$ by
\begin{align}\label{eq:defMomentsFull}
M_{k}(f) = \int_{\T^d} m_k(f)(x) \dd x = \int_{\T^d} \int_{\R^d}  |v|^{k} f(x,v) \dd v \dd x,
\end{align}
when it exists.

\paragraph{Various notations.}
\begin{itemize}
\item The quantity $ \langle v \rangle= (1+|v|^2)^{1/2}$ will be used frequently when computing moments in velocity.

\item To make notations lighter, we will use the following notations for Lebesgue norms: $ \lVert\cdot \rVert_{L^p_{x}}$ for $ \lVert\cdot \rVert_{L^p(\T^d)}$, $ \lVert\cdot \rVert_{L^p_{v}}$ for $ \lVert\cdot \rVert_{L^p(\R^d)}$, $ \lVert\cdot \rVert_{L^p_{t}}$ for $ \lVert\cdot \rVert_{L^p([0,T])}$,
$ \lVert\cdot \rVert_{L^p_{x,v}}$ for $ \lVert\cdot \rVert_{L^p(\T^d\times \R^d)}$, etc.

\item On a Polish space $\mathcal{X}$, $\mathcal{P}(\mathcal{X})$ denotes the set of Borel probability measures.

\item $C$ will denote a constant whose precise value may change from line to line.

\end{itemize}

\subsection{Main results}
\label{subsec:Main}

\paragraph{Hypotheses.}

Recall that we work in dimension $d=2$ or $d=3$. Throughout, we fix the time horizon $T>0$ and we will make the following assumptions on the PDE system~\eqref{eq:PDE}:
\begin{assumption}\label{assump-PDE}
Let $p>d$ and $\gamma\in (\frac{d}{p},1)$. It will be assumed that
\begin{enumerate}[label=(\alph*)]
\item\label{hyp:F0Linfty} $F^\circ \in L^\infty(\T^d\times\R^d)\cap L^1(\T^d\times\R^d)$;

\item\label{momenthigh} $M_{p+d(p-1)}(F^\circ)< \infty$;

\item\label{hyp:F0k}
There exists an integer $k\geq 3$ such that
$\displaystyle \int_{\R^{d}} \langle v \rangle^{2k} | F^\circ |^2 \dd v \in L^2(\T^{d})$;

\item\label{inivel} $u^\circ\in H^\gamma_{p}(\T^d)^d$;

\item\label{regul}
The velocity of the fluid is bounded, \emph{i.e.}
\[
u\in  L^\infty([0,T]; L^\infty(\T^d)^d).
\]

\end{enumerate}
\end{assumption}

As for the particle system, we will make the following assumptions.
\begin{assumption}\label{assump-particles}
Assume that in \eqref{eq:defThetaN}, the parameters $\alpha, \beta \in (0,1]$ are such that
\begin{enumerate}[label=(\alph*)]
\item\label{hyp:alphabeta}  $d\beta+(d+1)\alpha<\frac{1}{2}$ and  $\alpha> \beta$ \emph{;}

\end{enumerate}
Concerning the diffusion coefficients $\sigma$ in \eqref{eq:PDE} and $\sigma_{N}$ in \eqref{Part}, we assume the following.
\begin{enumerate}[label=(\alph*)]
\setcounter{enumi}{1}
\item\label{hyp:sigma} Assume that $\sigma\in [0,+\infty)$ and that for any $N\in \N^*$, $\sigma_{N}>0$. In addition, the following limits hold:
\begin{equation}\label{eq:assump-sigma}
 \lim_{N\to+\infty} \sigma_{N} = \sigma \ \text{and } \lim_{N\to+\infty}\sqrt{\log N}\, \sigma_{N} = +\infty .
\end{equation}
\end{enumerate}

Finally, concerning the initial conditions of the particle system, we shall assume that:
\begin{enumerate}[label=(\alph*)]
\setcounter{enumi}{2}
\item\label{hyp:ICu0}
There exist $p>d$ and $\gamma\in (\frac{d}{p},1)$ such that
 for any $q\geq 1$,
\begin{equation*}
\sup_{N \in \N} \E \big\| u^{\circ,N} \big\|_{L_{x}^{2}}^{q} <\infty  ~\text{ and }~ \sup_{N \in \N} \E \big\| u^{\circ,N} \big\|_{\gamma,p}^{q} <\infty  \emph{;}
\end{equation*}

\item\label{hyp:F0Nk} There exists an integer $k\geq 3$ such that
for any $q\geq 1$,
\begin{equation*}
\sup_{N \in \N} \E\left[  \big\| \langle v \rangle^{k}  ( \vartheta^N \ast  S_{0}^{N}) \big\|_{L^{2}_{x,v}}^{q}  \right] \
<\infty . %
\end{equation*}

\end{enumerate}

\end{assumption}

\paragraph{Main results.}

We now state our first main result on the convergence rate of
$(u^N,F^N)$, solution of \eqref{Part} with the smoothed empirical measure $F^N$ defined in \eqref{eq:defF}, towards a bounded weak solution
$(u,F)$ of \eqref{eq:PDE}. The precise definitions of such solutions, together with the existence and regularity to both systems under Assumptions~\ref{assump-PDE}~and~\ref{assump-particles} will be given in the next sections (see Proposition~\ref{prop:CoBessel-u} for the limit PDE and Proposition~\ref{prop:existence-IPS} for the fluid-particle system). In the following theorems, the rates of convergence are partially expressed in terms of an error that involves an auxiliary PDE system $(u^{(N)}, F^{(N)})$, which is a bounded weak solution of \eqref{eq:PDE} with $\sigma$ replaced by $\sigma_{N}$; see~\eqref{eq:intermediatePDE} for the precise definition.

\begin{theorem}\label{th:convBessel}
Let Assumptions~\ref{assump-PDE}~and~\ref{assump-particles} hold, and in particular let $p\in (d,+\infty)$, $\gamma \in (\tfrac{d}{p},1)$ and an integer $k \geq 3$ for which these assumptions hold. Assume further that the cut-off parameter satisfies
${A\geq \| u\|_{L^{\infty}_{t,x}}}$, and if $d=3$, that $\frac{\gamma}{2} + \frac{3}{2}(\frac{1}{2}-\frac{1}{p}) <1$. Then the following holds:
\begin{itemize}
\item If $\sigma>0$, then for any $q\in [1,+\infty)$, there exists $C = C(q,d,\gamma,p,A,\sigma)>0$ such that for any $N\in \N^*$,
\begin{align*}
\Big(\E &\sup_{t\in[0,T]} \big\| u^{N}_t- u_t \big\|_{ \gamma,p}^{2q}\Big)^{\frac{1}{q}}
+  \Big(\E \sup_{t\in[0,T]} \big\| \langle v \rangle^{k}(F_{t}^{N}-F_{t}) \big\|_{L^{2}_{x,v}}^{2q} \Big)^{\frac{1}{q}} \\
& \leq C \bigg( \big(\E \big\| u^{\circ,N}-u^\circ \big\|_{\gamma, p}^{2q}\big)^{\frac{1}{q}} +   \big(\E \big\| \langle v \rangle^{k} (F_{0}^{N}-\vartheta^N\ast F^\circ) \big\|_{L^{2}_{x,v}}^{2q}\big)^{\frac{1}{q}}\\
&\quad + N^{-2(\alpha-\beta)} +N^{-2\beta(\gamma-\frac{d}{p})} + N^{-(\frac{1}{2} - d\beta -(d+1)\alpha)} + \ratetilde^2 \bigg) + 2 \rate^2 ,
\end{align*}
where
\begin{equation}\label{eq:def-rate}
\rate^2 = \sup_{t\in[0,T]} \big\| u_t^{(N)} - u_t \big\|_{ \gamma,p}^2 +  \sup_{t\in[0,T]} \big\| \langle v \rangle^{k} (F^{(N)}_{t}\ast \vartheta^{N} - F_{t}) \big\|_{L^{2}_{x,v}}^2
\end{equation}
and
\begin{equation}\label{eq:def-ratetilde}
\ratetilde^2 = \sup_{t\in[0,T]} \big\| \langle v \rangle^{k} (F^{(N)}_{t}\ast \vartheta^{N} - F^{(N)}_{t}) \big\|_{L^{2}_{x,v}}^2
\end{equation}

\item If $\sigma=0$, then for any $q\in [1,+\infty)$ and any $\varepsilon>0$, there exists $C = C(q,d,\gamma,p,A,\varepsilon)>0$ such that for any $N\in \N^*$,
\begin{align*}
\Big(\E &\sup_{t\in[0,T]} \big\| u^{N}_t- u_t \big\|_{ \gamma,p}^{2q}\Big)^{\frac{1}{q}}
+  \Big(\E \sup_{t\in[0,T]} \big\| \langle v \rangle^{k}(F_{t}^{N}-F_{t}) \big\|_{L^{2}_{x,v}}^{2q} \Big)^{\frac{1}{q}} \\
& \leq C \bigg( \big(\E \big\| u^{\circ,N}-u^\circ \big\|_{\gamma, p}^{2q}\big)^{\frac{1}{q}} +   \big(\E \big\| \langle v \rangle^{k} (F_{0}^{N}-\vartheta^N\ast F^\circ) \big\|_{L^{2}_{x,v}}^{2q}\big)^{\frac{1}{q}} +\ratetilde^2 \\
&\quad + N^{-2(\alpha-\beta)} + N^{-2\beta(\gamma-\frac{d}{p})+\varepsilon} + N^{-(\frac{1}{2} - d\beta -(d+1)\alpha)}
\bigg) \, \frac{1}{1-\tilde\gamma} e^{CT\big(1+\sigma_{N}^{-2} + \sigma_{N}^{-\frac{1}{1-\tilde\gamma}}\big)}
 + 2 \rate^2,
\end{align*}
where
\begin{equation}\label{eq:deftildegamma}
\tilde{\gamma} = \frac{1}{2} \Big(\gamma + d\big(\frac{1}{2}-\frac{1}{p}\big)\Big) \vee \frac{\gamma+1}{2} \in (0,1) .
\end{equation}

\end{itemize}
\end{theorem}

\begin{remark}
\label{rk:rhoN}
In Sections~\ref{subsec:discuss-rateregul} and~\ref{subsec:discuss-rateregul-0}, we establish rates of convergence on $\rate$ and $\ratetilde$. These rates depend on whether $\sigma > 0$ or $\sigma=0$ and depend on the regularity assumptions on $u^\circ$ and $F^\circ$. Namely,
\begin{itemize}
\item when $\sigma > 0$, if Assumptions~\ref{assump-PDE}~and~\ref{assump-particles} hold, and if $\lVert \langle v\rangle^{2k} F^\circ \rVert_{L^2_{x,v}}<\infty$ and $\lVert F^\circ \rVert_{\gamma,2}<\infty$, then $\rate \lesssim |\sigma_{N}-\sigma|^\frac{1}{2} + N^{-\frac{1}{2}\beta(\gamma\wedge 1)}+ N^{-\beta(\gamma-\frac{d}{p})} + N^{-(\alpha-\beta)}$, see Proposition~\ref{prop:rate-rhoN}; and $\ratetilde \lesssim N^{-\frac{1}{2}\beta(\gamma\wedge 1)}+ N^{-\beta(\gamma-\frac{d}{p})} + N^{-(\alpha-\beta)}$, see Remark~\ref{rk:ratetilde-sigma>0};

\item when $\sigma = 0$, if Assumptions~\ref{assump-PDE}~and~\ref{assump-particles} hold, and if $\langle v\rangle^{2k} F^\circ \in L^2(\T^d\times\R^d)$, $F \in L^\infty([0,T]; H^\gamma_{2}(\T^d\times \R^d))$ and $\langle v\rangle^k \nabla_{v} F \in L^2([0,T];L^2(\T^d\times\R^d))$, then $\rate \lesssim \sigma_{N}^\frac{1}{2}$, see Proposition~\ref{prop:rate-rhoN-sigma0}; and $\ratetilde \lesssim N^{-\frac{1}{2}\beta(\gamma\wedge 1)}$, see Remark~\ref{rk:ratetilde-sigma=0}.
\end{itemize}
\end{remark}

\begin{remark}
\label{rk:rate-init}
Let us now comment on the two initial error terms $\big(\E \| u^{\circ,N}-u^\circ \|_{\gamma, p}^{2q}\big)^{\frac{1}{q}}$ and $\big(\E \big\| \langle v \rangle^{k} (F_{0}^{N}-\vartheta^N\ast F^\circ) \big\|_{L^{2}_{x,v}}^{2q}\big)^{\frac{1}{q}}$. There is a lot of freedom in the choice of $u^{\circ,N}$, but a particularly obvious one is to simply choose $u^{\circ,N}=u^\circ$ so that the first initial error vanishes.

For the second term, one can for instance follow the approach proposed in \cite[Proposition B.1]{ORT2}: assuming that $F^\circ$ is compactly supported and sufficiently regular, and that $(X^i_{0},V^i_{0})_{i\in \N}$ are i.i.d. with distribution $F^\circ$, one can adapt \cite[Proposition B.1]{ORT2} to derive $\big(\E \big\| \langle v \rangle^{k} (F_{0}^{N}-\vartheta^N\ast F^\circ) \big\|_{L^{2}_{x,v}}^{2q}\big)^{\frac{1}{q}}\lesssim N^{2d\beta(1-\frac{1}{2q})-1}$, for $q\geq 1$.
\end{remark}

\begin{remark}
In light of the two previous remarks, one can expect under reasonable assumptions (as listed above) that the rate of convergence in the case $\sigma>0$ is algebraic: the main terms that drive this error are $N^{-2(\alpha-\beta)}$, which comes from the deterministic approximation of $F$ by $F\ast \vartheta^N$ in the weighted $L^2$ norm; $N^{-2\beta(\gamma-\frac{d}{p})}$ which comes from the deterministic approximation of $u$ by $u\ast \vartheta^N$ in the $H^\gamma_{p}$ norm; $N^{-(\frac{1}{2} - d\beta -(d+1)\alpha)}$ which is the rate of convergence of the martingale term from the expansion of $F^N$ in the weighted $L^2$ norm; and $|\sigma-\sigma_{N}|$ coming from the approximation of $(u,F)$ by the auxiliary PDE $(u^{(N)},F^{(N)})$.

In the case $\sigma=0$, similar errors appear, but they are dominated by $\rate$, which we do not expect to be better than $\lesssim \sigma_{N}$. Since we had to assume \eqref{eq:assump-sigma}, in this case the rate is logarithmic.
\end{remark}

With almost the same assumptions, the bound from Theorem~\ref{th:convBessel} leads to a bound in energy norm.
\begin{theorem}
\label{th:convenergy}
With the same assumptions and notations as in Theorem~\ref{th:convBessel},
assuming further that $\sup_{t\in[0,T]} \big\| m_0(F_t) \big\|_{L^\infty_{x}}$ is finite, the following holds:
\begin{itemize}
\item If $\sigma>0$, then for any $q\in [1,+\infty)$, there exists $C = C(q,d,\gamma,p,A,\sigma)>0$ such that for any $N\in \N^*$,
\begin{align*}
\Big(\E & \sup_{t\in[0,T]} \big\|u^{N}_t-u_t \big\|_{L^2_{x}}^{2q}\Big)^{\frac{1}{q}} + \Big(\E \Big|\int_{0}^{T} \big\| \nabla (u_{s}^{N}- u_{s}) \big\|_{L^2_{x}}^{2} \dd s \Big|^{q}\Big)^{\frac{1}{q}}\\
& \leq C \bigg( \big(\E \big\| u^{\circ,N}-u^\circ \big\|_{\gamma, p}^{2q}\big)^{\frac{1}{q}} +   \big(\E \big\| \langle v \rangle^{k} (F_{0}^{N}-\vartheta^N\ast F^\circ) \big\|_{L^{2}_{x,v}}^{2q}\big)^{\frac{1}{q}}\\
&\quad + N^{-2(\alpha-\beta)} +N^{-2\beta(\gamma-\frac{d}{p})} + N^{-(\frac{1}{2} - d\beta -(d+1)\alpha)} + \ratetilde^2 \bigg) + C \rate^2 .
\end{align*}

\item If $\sigma=0$, then for any $q\in [1,+\infty)$ and any $\varepsilon>0$, there exists $C = C(q,d,\gamma,p,A,\varepsilon)>0$ such that for any $N\in \N^*$,
\begin{align*}
\Big(\E & \sup_{t\in[0,T]} \big\|u^{N}_t-u_t \big\|_{L^2_{x}}^{2q}\Big)^{\frac{1}{q}} + \Big(\E \Big|\int_{0}^{T} \big\| \nabla (u_{s}^{N}- u_{s}) \big\|_{L^2_{x}}^{2} \dd s \Big|^{q}\Big)^{\frac{1}{q}} \\
& \leq C \bigg( \big(\E \big\| u^{\circ,N}-u^\circ \big\|_{\gamma, p}^{2q}\big)^{\frac{1}{q}} +   \big(\E \big\| \langle v \rangle^{k} (F_{0}^{N}-\vartheta^N\ast F^\circ) \big\|_{L^{2}_{x,v}}^{2q}\big)^{\frac{1}{q}} +\ratetilde^2 \\
&\quad + N^{-2(\alpha-\beta)} + N^{-2\beta(\gamma-\frac{d}{p})+\varepsilon} + N^{-(\frac{1}{2} - d\beta -(d+1)\alpha)}
\bigg) \, \frac{1}{1-\tilde\gamma} e^{CT\big(1+\sigma_{N}^{-2} + \sigma_{N}^{-\frac{1}{1-\tilde\gamma}}\big)}
 + C \rate^2.
\end{align*}
\end{itemize}
\end{theorem}

\begin{remark}
Let us discuss a sufficient condition that permits to fulfill the additional assumption $\sup_{t\in[0,T]} \big\| m_0(F_t) \big\|_{L^\infty_{x}}<\infty$ from Theorem~\ref{th:convenergy}.
Assume that for some $s> \frac{d}{2}$, there exists a strong solution ${F\in L^\infty([0,T];H^s(\T^d\times \R^d))}$ and $u\in L^\infty([0,T];L^\infty(\T^d)^d)$, see Remark~\ref{rk:reg-U-F} for a discussion on the literature where such regularity is achieved. Then Lemma~\ref{lem:momentsF2} implies that $\langle v\rangle^k F \in L^\infty([0,T];L^2(\T^d\times\R^d))$. Thus $\langle v\rangle^k F \in L^\infty([0,T];H^s(\T^d)\times L^2(\R^d))$ and by Sobolev embedding, $\sup_{t\in [0,T]}\| m_0(F_t)\|_{L^\infty_{x}}$ is finite.

Note also that the same arguments permit to establish that $\langle v\rangle^k \nabla_{v} F \in L^2([0,T];L^2(\T^d\times\R^d))$, which is the assumption made in Remark~\ref{rk:rhoN} in the case $\sigma=0$.
\end{remark}

We conclude this section with a result concerning the trajectorial convergence of particles. Combined with Theorem~\ref{th:convBessel}, the following statement gives a rate of convergence in the strong propagation of chaos property of the fluid-particle system.
Its simple proof relies on a parallel coupling between particles in the finite system and independent particles in the limit fluid, using as well the Lipschitz regularity of $u$ derived in Section~\ref{sec:prelim}. It will be carried out in details in Section~\ref{subsec:prooflastcorollary}.

\begin{theorem}
\label{th:strongPoC}
With the same assumptions and notations as in Theorem~\ref{th:convBessel}, given a filtered probability space $(\Omega,\mathcal{F},(\mathcal{F}_{t})_{t\geq 0},\PP)$ and a family of independent $(\mathcal{F}_{t})_{t\geq 0}$-Brownian motions $(B^i)_{i\in \N}$, consider the strong solution of the following SDE for each $i\in \N$:
\begin{equation*}
\begin{cases}
& \dd \overline{X}^{i}_t= \overline{V}^{i}_t \dd t, \quad t>0,\\
& \dd \overline{V}^{i}_t= (u_{t}(\overline{X}_{t}^{i})-\overline{V}_{t}^{i}) \dd t + \sigma \dd B_{t}^{i}, \quad t>0 ,
\end{cases}
\end{equation*}
with initial conditions given by a family of random variables $(X_{0}^i,V_{0}^i)_{i\in \N}$. For each $N\in \N^*$, let $(X^{i,N},V^{i,N})_{i\in \N}$ be the solution of the fluid-particle system~\eqref{Part} with the same initial conditions $(X_{0}^i,V_{0}^i)_{i\in \N}$ and same driving Brownian motion $(B^i)_{i\in \N}$, whose existence is given by Proposition~\ref{prop:existence-IPS}.
Then for any $q\geq 1$ and any $N\in \N^*$,
\begin{equation*}
\max_{i\in \{1,\dots, N\}} \E \sup_{t\in [0,T]} \big|(X_{t}^{i,N},V_{t}^{i,N}) - (\overline{X}_{t}^i,\overline{V}_{t}^i)\big|^q \lesssim |\sigma-\sigma_{N}|^q + \E \sup_{t\in[0,T]} \big\| u^{N}_t- u_t \big\|_{ \gamma,p}^{q} .
\end{equation*}
\end{theorem}

\section{Preliminary results on the PDE system}
\label{sec:prelim}

\subsection{Bounded weak solutions}

In Theorems~\ref{th:convBessel} and~\ref{th:convenergy}, we compare the finite particle system (see beginning of Section~\ref{sec:fluid-particle}) to weak solutions of the PDE system~\eqref{eq:PDE}. We  adopt below the definition of weak solutions that is classically used in the literature.
\begin{definition}
\label{def:weaksolution}
We say that $(u,F)$ is a weak solution to the Vlasov(--Fokker--Planck)--Navier--Stokes system \eqref{eq:PDE} if
\begin{align*}
&  u \in L^\infty([0,T];L^2(\T^d)^d) \cap L^2([0,T];H^1(\T^d)^d),\\
& \dive u_{t} =0  \quad \text{for a.e. } (t,x)\in(0,T)\times\T^d,\\
& F(t,x ,v )\geq0,\quad \text{for a.e. } (t,x ,v)\in(0,T)\times\T^d\times\R^d, \\
& F\in L^\infty([0,T];L^\infty(\T^d\times\R^d)\cap L^1(\T^d\times\R^d)), \\
& |v |^2F \in L^\infty([0,T]; L^1(\T^d\times\R^d)) ;
\end{align*}
and $(u,F)$ satisfy, for all $\phi\in \mathcal{C}^\infty([0,T]\times\mathbb{T}^d\times\mathbb{R}^d)$ with compact support in $v $ and with $\phi(T,\cdot,\cdot)=0$,
\begin{align*}
-\int^T_0 &\int_{\mathbb{R}^d}\int_{\mathbb{T}^d} \bigl[\partial_s\phi+v \cdot\nabla_x \phi+(u-v )\cdot\nabla_v \phi - \frac{\sigma^2}{2} \Delta_v \phi \bigr] F \dd x \dd v \dd s
 =\int_{\mathbb{R}^d}\int_{\mathbb{T}^d} \phi(0,x ,v) F^\circ \dd x \dd v ;
\end{align*}
and for all $\psi\in \mathcal{C}^\infty([0,T]\times\mathbb{T}^d;\R^d)$ with $\nabla \cdot \psi =0 $ and $ \psi(T,\cdot) =0$,
\begin{equation}
\label{eq:uweak}
\begin{split}
 - \int_0^T\int_{\mathbb{T}^d} u \cdot \partial_s \psi \dd x \dd s
 & +\int_0^T\int_{\mathbb{T}^d} \big((u \cdot \nabla) \psi \cdot u \big) \dd x \dd s +\int_0^T \int_{\mathbb{T}^d}  u \cdot \Delta \psi \dd x \dd s \\
& =-\int_0^T\int_{\mathbb{R}^d}\int_{\mathbb{T}^d} (u-v) F \cdot \psi \dd x \dd v \dd s + \int_{\T^d} u^\circ \psi(0,x)  \dd x.
\end{split}
\end{equation}
\end{definition}

The existence of a weak solution to the system \eqref{eq:PDE} satisfying an energy inequality was established in \cite{ChaeKangLee,Boudin2009,PinheiroPlanas}, under the assumption that the initial data $(u^\circ,F^\circ)$ satisfy $u^\circ \in L^2(\T^d)^d$ with $ \dive u^\circ =0$, and $F^\circ \in L^\infty(\T^d\times\R^d)\cap L^1(\T^d\times\R^d)$ with $ M_2(F^\circ) < +\infty$, which is covered by Assumption~\ref{assump-PDE}.

\medskip

More specifically in this paper, we restrict our attention to \emph{bounded} weak solutions to system \eqref{eq:PDE}, that is, weak solutions $(u, F)$ such that $u\in  L^\infty([0,T]; L^\infty(\T^d)^d)$. This corresponds to Assumption~\ref{assump-PDE}\ref{regul}.

\begin{remark}
\label{rk:reg-U-F}
Let us comment here on the literature related to the existence of bounded weak solutions.

First, when $\sigma>0$, the result of \cite{Flandoli2} establishes that in dimension $2$, for divergence-free $u^\circ \in H^2(\T^2)^2$ and nonnegative $F^\circ\in L^1\cap L^\infty(\T^2\times \R^2)$ with finite sixth moment, there exists a unique weak solution $(u,F)$ of \eqref{eq:PDE} such that $u\in  L^\infty([0,T]; L^\infty(\T^2)^2)$.
In dimension $3$, the small-data result for $\sigma=1$ in \cite{GoudonHeMoussaZhang} yields a global classical solution $(u,F)$ with  $ F = \mu + \sqrt{\mu} f \geq 0$. Here, $\mu(v)$ represents the normalized Maxwellian, and the initial conditions are assumed to be such that
$\|u^\circ\|_{H^s_x}+\|f^\circ\|_{H^s_{x,v}}$ is sufficiently small, with $s\ge 2$. Since $s>3/2$, for $d=3$ the Sobolev embedding $H^s(\T^3)\hookrightarrow L^\infty(\T^3)$ implies $u(t)\in L^\infty(\T^3)^3$ for each $t\ge 0$.  Thus, in the perturbative three-dimensional regime discussed in \cite{GoudonHeMoussaZhang}, the assumption that $u\in  L^\infty([0,T]; L^\infty(\T^3)^3)$ follows directly from the theory.

In the case $ \sigma =0$, the system~\eqref{eq:PDE} corresponds to the constant-density case of the equation analyzed in \cite{Choi_2015}. Under a smallness condition on
$
\|u^\circ\|_{H^2_x} + \|F^\circ\|_{H^2_{x,v}},
$
and assuming that $F^\circ$ has compact support in $x$ and $v$,
the result in \cite{Choi_2015} guarantees a unique strong solution with
$
u \in \mathcal{C}([0,T];H^2(\mathbb T^3)^3).
$
Hence in that situation,
$u\in  L^\infty([0,T]; L^\infty(\T^3)^3)$
 is a consequence of the Sobolev embedding theorem.

Since our results can be established assuming only the existence of a weak solution with bounded velocity field $u$, we prefer to assume simply that Assumption~\ref{assump-PDE}\ref{regul} holds in order to encompass various situations ($\sigma>0$ or $\sigma=0$, $d=2$ or $d=3$), rather than treating separately these various cases.
\end{remark}

We next observe that with assumptions on the $k$-th moment of $F^\circ$, we can control the $k$-th moment of $F$ for any bounded weak solution.

\begin{lemma}\label{lem:momentsF}
Let $(u, F)$ be any bounded weak solution of \eqref{eq:PDE}.
If for some $k\geq 3$ there is $M_k(F^\circ)< +\infty $, then there
exists a constant $C>0$ such that for all $t\in[0,T]$, we have
\begin{align*}
& M_k(F_t)  \\
& \leq \exp\Bigl(C (\|F^\circ\|_{L^\infty_{x,v}} +\sigma +1)\int_0^t \|u_s\|_{L^{k+d}_x} \dd s \Bigr)
\Bigl(M_k(F^\circ) + \sigma \|F\|_{L^1_{t,x,v}} + (\|F^\circ\|_{L^\infty_{x,v}} +1)\int_0^t \|u_s \|_{L^{k+d}_x} \dd s \Bigr).
\end{align*}
\end{lemma}
\begin{proof}
The proof goes along the same lines as Lemma 3.6 of \cite{Flandoli2}.
\end{proof}
Therefore, for any bounded weak solution, all  moments $M_k(F_t)$
 remain bounded in $L^\infty(0,T)$.\\
We now recall the following lemma from \cite[Lemma 1]{Boudin2009}.
\begin{lemma}\label{lemmamoments}
Let $k>0$ and let $F$ be a nonnegative real function in $L^\infty([0,T];\T^d\times\R^d)$, such that $m_k (F_t)(x )<+\infty$ for a.e. $(t,x)$. Then it holds, for any $l<k$,
\[m_l (F_t)(x )\leq C \bigl(\|F\|_{L^\infty_{t,x,v}}+1\bigr)m_k( F_t)(x )^{\frac{l+d}{k+d}} ~\text{ for a.e. } (t,x ).\]
\end{lemma}
In particular, one has for any $q \geq 1$,
\begin{align}\label{L2norm_m_0F}
    \|m_0 (F_t) \|_{L^q_x} & \leq C \bigl(\|F\|_{L^\infty_{t,x,v}}+1\bigr) \|M_{d(q-1)}(F)\|^{1/q}_{L^\infty_t}, \\ \label{L2norm_m_1F}
    \|m_1(F_t) \|_{L^q_x}  &\leq  C \bigl(\|F\|_{L^\infty_{t,x,v}}+1\bigr) \|M_{q+d(q-1)}(F)\|^{1/q}_{L^\infty_t} .
\end{align}

As a simple consequence of the two previous lemmas, we can control the $0$th-order and $1$st-order moments in velocity of any bounded weak solution.
\begin{proposition}
\label{prop:boundmomentsF}
Let $p\geq 1$ and $(u, F)$ be any bounded weak solution of \eqref{eq:PDE}.
If $M_{p+d(p-1)}(F^\circ)< +\infty$, then
\begin{equation*}
\sup_{t\in [0,T]} \lVert m_{0}(F_{t})\rVert_{L^p_{x}} + \sup_{t\in [0,T]} \lVert m_{1}(F_{t})\rVert_{L^p_{x}} <\infty .
\end{equation*}
\end{proposition}

\begin{proof}
Combining Lemma~\ref{lem:momentsF} and \eqref{L2norm_m_0F}, one gets for $m_{0}(F_{t})$ that
\begin{align*}
&\|m_0 (F_t) \|_{L^p_x}^p
\lesssim \bigl(\|F\|_{L^\infty_{t,x,v}}+1\bigr)^p \\
&\quad \times \exp\Bigl(C (\|F^\circ\|_{L^\infty_{x,v}} +\sigma +1)\int_0^T \|u_s\|_{L^{dp}_x} \dd s \Bigr)
\Bigl(M_{d(p-1)}(F^\circ) + \sigma \|F\|_{L^1_{t,x,v}} + (\|F^\circ\|_{L^\infty_{x,v}} +1) \int_0^T \|u_s \|_{L^{dp}_x} \dd s \Bigr).
\end{align*}
By the Definition~\ref{def:weaksolution} of a weak solution, $F$ is integrable and bounded in all three variables. It remains to control $ \lVert u \rVert_{L^1_{t}L^{dp}_{x}}$, which follows simply from the assumption that $u\in  L^\infty([0,T]; L^\infty(\T^d)^d)$. The bound on $\lVert m_{1}(F_{t})\rVert_{L^p_{x}}$ is obtained similarly.
\end{proof}

The bound from Proposition~\ref{prop:boundmomentsF} will be used several times in the forthcoming proofs, because it gives a control on the forcing term of the Navier--Stokes component of \eqref{eq:PDE}. Indeed, it holds
\[\Bigl|\int_{\R^d} (u_t(x)-v) F_t(x,v) \dd v\Bigr|
 \leq\|u\|_{L^\infty_{t,x}} m_0(F_t)+ m_1(F_t).
\]
Therefore, if $M_{p+d(p-1)}(F^\circ)$ is finite,  the forcing term belongs to $L^\infty([0,T];L^p(\T^d)^d)$.

\subsection{Mild solutions}

If the initial datum $u^\circ$ is in $L^p(\T^d)^d$, then the weak solution $u$ is also a mild solution in $[0,T]$, see \emph{e.g.}~\cite{Fabes1972}. Thus it satisfies
\begin{equation}\label{eq:umild}
u_{t} = e^{t  \Delta}u^\circ - \int_{0}^{t}   e^{\left(  t-s\right)  \Delta} P\big[ (u_{s} \cdot \nabla)u_s\big] \dd s  - \int_{0}^{t}   e^{\left(  t-s\right)  \Delta} P\Big[\int_{\R^d} (u_s-v) F_s(\cdot ,v) \dd v \Big] \dd s,
\end{equation}
where $P $ is the Leray projector (recall Section~\ref{subsec:notations}).

Let us give a technical lemma that will be useful several times in the forthcoming computations.
\begin{lemma}
\label{lem:divfreecommute}
Let $f,g \in L^2([0,T];H^1(\T^d)^d)$ with $f$ that is divergence-free. Let $p>d$ and $\gamma\in (\frac{d}{p},1)$. Then there exists $C>0$ such that for any $t\in [0,T]$,
\begin{equation*}
\int_{0}^{t} \Big\Vert e^{\left(  t-s\right) \Delta} P \big[ (f_{s} \cdot \nabla)g_s\big] \Big\Vert _{\gamma,p} \dd s \leq C \int_{0}^{t}  \frac{1}{(t-s)^{\frac{1+\gamma}{2}}} \left\Vert f_{s} \otimes g_{s} \right\Vert _{L^p_{x}}  \dd s .
\end{equation*}
\end{lemma}

\begin{proof}
Observe that the divergence-free property of $f$ implies that $(f\cdot \nabla)g = \nabla\cdot \big(f \otimes g\big)$. Hence by the fact that $\nabla$ commutes with $P$, the property of the semigroup \eqref{eq:Bessel-heat-estimate} and that $P$ is continuous in $L^p$, the conclusion follows.
\end{proof}

We now state the regularity result for bounded weak solutions to system \eqref{eq:PDE}.
\begin{proposition}
\label{prop:CoBessel-u}
Let Assumptions~\ref{assump-PDE}\ref{inivel} and \ref{momenthigh} hold.
Let $(u, F)$ be any bounded weak solution of system \eqref{eq:PDE}. Then the following holds:
\begin{enumerate}[label=(\alph*)]
\item\label{item:reg-u} $u\in \mathcal{C}( [0,T];  H^\gamma_p(\T^d)^d)$;

\item\label{item:reg-nablau} $\sup_{t\in (0,T]} t^{\frac{1}{2}} \lVert \nabla u_{t} \rVert_{\gamma,p} <\infty$.

\end{enumerate}
\end{proposition}

\begin{proof}
Let us prove \ref{item:reg-u} first. Since $p>d$, $\gamma\in (\frac{d}{p},1)$, then $u^\circ \in H^\gamma_p(\T^d)^d$ implies $u^\circ\in L^\infty(\T^d)^d$. We then write in mild formulation
\begin{equation*}
u_{t} = e^{t  \Delta}u^\circ - \int_{0}^{t}   e^{\left(  t-s\right)  \Delta} P\big[ (u_{s} \cdot \nabla)u_s\big] \dd s  - \int_{0}^{t}   e^{\left(  t-s\right)  \Delta} P\Big[\int_{\R^d} (u_s-v) F_s \dd v \Big] \dd s,
\end{equation*}
and estimate, using Lemma~\ref{lem:divfreecommute}, the property of the semigroup \eqref{eq:Bessel-heat-estimate} and that $P$ is continuous in $L^p$,
\begin{align*}
\bigl \Vert  u_t  \bigr \Vert _{\gamma, p}
&\leq\left\Vert  e^{t\Delta}u^\circ\right\Vert _{\gamma, p} + \int_{0}^{t}\left\Vert e^{\left(  t-s\right)  \Delta} P \big[ (u_{s} \cdot \nabla)u_s\big] \right\Vert _{\gamma,p} \dd s  +  \int_{0}^{t}\Big\Vert  e^{\left(  t-s\right)  \Delta} P\Big[\int_{\R^d} (u_s-v) F_s \dd v \Big] \Big\Vert _{\gamma,p} \dd s\\
& \lesssim \|u^\circ\|_{\gamma, p} + \int_{0}^{t}  \frac{1}{(t-s)^{\frac{1+\gamma}{2}}} \left\Vert |u_{s}|^2 \right\Vert _{L^p_{x}}  \dd s
 +\int_{0}^{t} \frac{1}{(t-s)^{\frac{\gamma}{2}}} \Big\Vert  \int_{\R^d} (u_s-v) F_s \dd v\Big\Vert _{L^p_{x}} \dd s.
\end{align*}
For the term $\Vert |u_{s}|^2 \Vert _{L^p_{x}}$, recall that $u$ is a bounded solution. Then in view of Proposition~\ref{prop:boundmomentsF} and the subsequent computation, we know that for $p\geq2$ and  $\gamma\in [0, 1)$, the forcing term $\int_{\R^d}(u_s-v) F_s \dd v$ belongs to  $L^\infty([0,T];L^p(\T^d)^d)$. Thus we have that $ u \in \mathcal{C} ([0,T];H^\gamma_p(\T^d)^d)$.

Now for \ref{item:reg-nablau}, we have
\begin{equation*}
\nabla u_{t} = \nabla e^{t  \Delta}u^\circ - \int_{0}^{t} \nabla e^{\left(  t-s\right)  \Delta} P\big[ (u_{s} \cdot \nabla)u_s\big] \dd s  - \int_{0}^{t} \nabla e^{\left(  t-s\right)  \Delta} P\Big[\int_{\R^d} (u_s-v) F_s \dd v \Big] \dd s.
\end{equation*}
Similarly to the first part of the proof, we get
\begin{align*}
\bigl \Vert \nabla u_t  \bigr \Vert _{\gamma, p}
&\leq\left\Vert \nabla e^{t\Delta}u^\circ\right\Vert _{\gamma, p}  + \int_{0}^{t} \left\Vert \nabla e^{\left(  t-s\right)  \Delta} P \big[ (u_{s} \cdot \nabla)u_s\big] \right\Vert _{\gamma,p} \dd s  +  \int_{0}^{t}\Big\Vert \nabla e^{\left(  t-s\right)  \Delta} P\Big[\int_{\R^d}(u_s-v) F_s \dd v \Big] \Big\Vert _{\gamma,p} \dd s\\
& \lesssim t^{-\frac{1}{2}}\|u^\circ\|_{\gamma, p} + \int_{0}^{t}  \frac{1}{(t-s)^{\frac{1}{2}}} \left\Vert (u_{s} \cdot \nabla)u_s \right\Vert _{\gamma,p}  \dd s
 +\int_{0}^{t} \frac{1}{(t-s)^{\frac{1+\gamma}{2}}} \Big\Vert  \int_{\R^d}(u_s-v) F_s \dd v\Big\Vert _{L^p_{x}} \dd s.
\end{align*}
Since $H^\gamma_{p}$ is an algebra for $\gamma>d/p$, it follows from part~\ref{item:reg-u} that
$$\Vert (u_{s} \cdot \nabla)u_s \Vert _{\gamma,p} \lesssim \Vert \nabla u_s \Vert _{\gamma,p}  \sup_{s\leq T} \Vert u_{s} \Vert _{\gamma,p} \lesssim \Vert \nabla u_s \Vert _{\gamma,p}.$$
As in the previous part, we also use $\int_{\R^d}(u_s-v) F_s \dd v \in L^\infty([0,T];L^p(\T^d)^d)$. Hence
\begin{align*}
\bigl \Vert \nabla u_t  \bigr \Vert _{\gamma, p}
& \lesssim 1+ t^{-\frac{1}{2}}\|u^\circ\|_{\gamma, p} + \int_{0}^{t}  \frac{1}{(t-s)^{\frac{1}{2}}} \left\Vert \nabla u_s \right\Vert _{\gamma,p}  \dd s .
\end{align*}
Hence by a Gr\"onwall-type Lemma for convolution (see e.g. \cite[Lemma 7.1.1]{Henry}), it comes
\begin{align*}
\bigl \Vert \nabla u_t  \bigr \Vert _{\gamma, p}
& \lesssim 1+ t^{-\frac{1}{2}}\|u^\circ\|_{\gamma, p} + \int_{0}^{t}  \big(1+ s^{-\frac{1}{2}}\|u^\circ\|_{\gamma, p} \big) \dd s ,
\end{align*}
and the result follows.
\end{proof}

\paragraph{Estimates for the moments.}
In this work, we will frequently use the following bounds on $m_0$ and $m_1$ expressed in terms of the Japanese bracket $ \langle v \rangle$.

\begin{lemma}
\label{lem:japanese}
Let $f$ and $g$ be nonnegative real functions defined on $\T^d\times\R^d$ with $m_0(f)$, $m_1(f)$, $m_0(g)$ and $m_1(g)$ which are finite almost everywhere.
Let $k>1+\frac{d}{2}$ (\emph{i.e.} $k\geq 3$ since $k$ is an integer), then there exists a constant $C=C(k,d)>0$ such that
\begin{equation*}
\lVert m_{0}(f) -m_{0}(g) \rVert_{L^2_x} \leq C \lVert \langle v \rangle^k (f-g) \rVert_{L^2_{x,v}},
\end{equation*}
and
\begin{equation*}
\lVert m_{1}(f) - m_{1}(g)\rVert_{L^2_x} \leq C \lVert \langle v \rangle^k (f-g) \rVert_{L^2_{x,v}}.
\end{equation*}
\end{lemma}
\begin{proof}
Let $h=f-g$. Using the Cauchy-Schwarz inequality, we have
\begin{align*}
\lVert m_{0}(f)-m_{0}(g) \rVert_{L^2_x}^2
&= \int_{\T^d} \left(\int_{\R^d} h(x,v) \dd v\right)^2\dd x= \int_{\T^d} \left(\int_{\R^d}\frac{\langle v \rangle^{k}}{\langle v \rangle^{k}}  h(x,v) \dd v\right)^2\dd x\\
&\leq \int_{\T^d} \left(\int_{\R^d} \langle v \rangle^{2k} |h(x,v)|^2 \dd v\right)\left(\int_{\R^d} \frac{1}{\langle v \rangle^{2k}} \dd v \right) \dd x\\
&\lesssim \int_{\T^d} \int_{\R^d} \langle v \rangle^{2k} |h(x,v)|^2 \dd v \dd x,
\end{align*}
as soon as $k>\frac{d}{2}$,
and
\begin{align*}
\lVert m_{1}(f) -m_{1}(g) \rVert_{L^2_x}^2
&= \int_{\T^d} \left(\int_{\R^d} |v| h(x,v) \dd v\right)^2\dd x= \int_{\T^d} \left(\int_{\R^d}\frac{|v|\langle v \rangle^{k}}{\langle v \rangle^{k}}  h(x,v) \dd v\right)^2\dd x\\
&\leq \int_{\T^d} \left(\int_{\R^d} \langle v \rangle^{2k} |h(x,v)|^2 \dd v\right)\left(\int_{\R^d} \frac{|v|^2}{\langle v \rangle^{2k}} \dd v \right) \dd x\\
&\lesssim \int_{\T^d} \int_{\R^d} \langle v \rangle^{2k} |h(x,v)|^2 \dd v \dd x,
\end{align*}
as soon as $k>1+\frac{d}{2}$.
\end{proof}

\begin{lemma}
\label{lem:momentsF2}
Let Assumption~\ref{assump-PDE}\ref{hyp:F0k} holds for some $k\in \N$. Let $\sigma\in [0,+\infty)$. Let $\mathfrak{u} \in L^\infty([0,T]; L^\infty(\T^d)^d)$ be a given function and let $F^\mathfrak{u}$ be a weak solution on $[0,T]$ of
\begin{equation}\label{eq:linearisedTransport}
\begin{cases}
&\partial_{t} F^\mathfrak{u}_{t}(x,v) + v \cdot \nabla_{x} F^\mathfrak{u}_{t}(x,v) + \dive_v((\mathfrak{u}_{t}(x)-v) F^\mathfrak{u}_{t}(x,v))= \frac{\sigma^{2}}{2} \Delta_v F^\mathfrak{u}_{t}(x,v) , \quad t>0, \ (x,v)\in \T^d\times \R^{d},\\
& F^\mathfrak{u}_{0}(x,v) = F^\circ(x,v),  \quad (x,v)\in \T^{d}\times \R^d.
\end{cases}
\end{equation}
Then
\begin{equation*}%
\begin{split}
\sup_{t\in [0,T]}& \iint \langle v\rangle^{2k} |F^\mathfrak{u}_{t}(x,v)|^2 \dd x \dd v
+\frac{\sigma^2}{2} \int_{0}^T \iint \langle v\rangle^{2k} |\nabla_{v} F^\mathfrak{u}_{s}(x,v)|^2 \dd x \dd v \dd s \\
&\leq  e^{T (1+ 2k \lVert \mathfrak{u}\rVert_{L^\infty_{t,x}} +2k^2\sigma^2)} \iint \langle v\rangle^{2k} |F^\circ(x,v)|^2 \dd x \dd v.
\end{split}
\end{equation*}
\end{lemma}

\begin{proof}
Up to regularising $F^\circ$ and $F^\mathfrak{u}$ (the equation is linear) and applying eventually Fatou's Lemma, we can work directly as if $F^\mathfrak{u}$ were strong solution of \eqref{eq:linearisedTransport}. Then we have
\begin{align*}
\frac{\dd}{\dd t} \iint  \langle v\rangle^{2k} |F^\mathfrak{u}_{t}|^2 \dd x \dd v
&= 2 \iint \langle v\rangle^{2k} F^\mathfrak{u}_{t} \, \partial_{t} F^\mathfrak{u}_{t} \dd x \dd v \\
&= - 2 \iint \langle v\rangle^{2k} F^\mathfrak{u}_{t}(x,v) \, v\cdot \nabla_{x} F^\mathfrak{u}_{t}(x,v) \dd x \dd v \\
&\quad - 2\iint \langle v\rangle^{2k} F^\mathfrak{u}_{t}(x,v) \, \dive_v((\mathfrak{u}_{t}(x)-v) F^\mathfrak{u}_{t}(x,v)) \dd x \dd v \\
&\quad + \sigma^2 \iint \langle v\rangle^{2k} F^\mathfrak{u}_{t} \, \Delta_{v}F^\mathfrak{u}_{t} \dd x \dd v,
\end{align*}
so that an integration-by-parts yields
\begin{align*}
\frac{\dd}{\dd t} \iint  \langle v\rangle^{2k} |F^\mathfrak{u}_{t}|^2 \dd x \dd v
+ \sigma^2 \iint \langle v\rangle^{2k} |\nabla_{v} F^\mathfrak{u}_{t}|^2 \dd x \dd v
&= - 2 \iint \langle v\rangle^{2k} F^\mathfrak{u}_{t}(x,v) \, v\cdot \nabla_{x} F^\mathfrak{u}_{t}(x,v) \dd x \dd v \\
&\quad - 2\iint \langle v\rangle^{2k} F^\mathfrak{u}_{t}(x,v) \, \dive_v((\mathfrak{u}_{t}(x)-v) F^\mathfrak{u}_{t}(x,v)) \dd x \dd v \\
&\quad -2k \sigma^2 \iint v \langle v\rangle^{2k-2} F^\mathfrak{u}_{t} \cdot \nabla_{v}F^\mathfrak{u}_{t} \dd x \dd v \\
& \eqqcolon I_{1}+I_{2}+I_{3}.
\end{align*}
Writing
\begin{equation*}
I_{1} = - \iint \langle v\rangle^{2k}  v\cdot \nabla_{x} |F^\mathfrak{u}_{t}(x,v)|^2 \dd x \dd v
\end{equation*}
and by integration-by-parts, we get $I_{1}=0$.

For $I_{2}$, integrate-by-parts again and split the integral to get
\begin{align*}
I_{2} &= 4k \iint \langle v\rangle^{2k-2} v \cdot (\mathfrak{u}_{t}(x)-v) |F^\mathfrak{u}_{t}(x,v)|^2 \dd x \dd v + 2 \iint \langle v\rangle^{2k} \, (\mathfrak{u}_{t}(x)-v) F^\mathfrak{u}_{t}(x,v) \cdot \nabla_{v}F^\mathfrak{u}_{t}(x,v) \dd x \dd v\\
&= 4k \iint \langle v\rangle^{2k-2} v \cdot (\mathfrak{u}_{t}(x)-v) |F^\mathfrak{u}_{t}(x,v)|^2 \dd x \dd v + \iint \langle v\rangle^{2k} \, (\mathfrak{u}_{t}(x)-v)  \cdot \nabla_{v}|F^\mathfrak{u}_{t}(x,v)|^2 \dd x \dd v \\
&= 2k \iint \langle v\rangle^{2k-2} v \cdot (\mathfrak{u}_{t}(x)-v) |F^\mathfrak{u}_{t}(x,v)|^2 \dd x \dd v + \iint \langle v\rangle^{2k} |F^\mathfrak{u}_{t}(x,v)|^2 \dd x \dd v.
\end{align*}
Hence
\begin{align*}
|I_{2}| \leq ( 2k \lVert \mathfrak{u}\rVert_{L^\infty_{t,x}} +1) \iint \langle v\rangle^{2k} |F^\mathfrak{u}_{t}|^2 \dd x \dd v.
\end{align*}

As for $I_{3}$, it vanishes in the case $\sigma=0$, so we are only dealing with the case $\sigma >0$ here. Using Young's inequality gives
\begin{align*}
|I_{3}| &\leq 2k\sigma^2  \iint \langle v\rangle^{2k} |F^\mathfrak{u}_{t}|\, |\nabla_{v}F^\mathfrak{u}_{t}| \dd x \dd v\\
&\leq \frac{\sigma^2}{2}  \iint \langle v\rangle^{2k} |\nabla_{v}F^\mathfrak{u}_{t}|^2 \dd x \dd v + 2k^2\sigma^2  \iint \langle v\rangle^{2k} |F^\mathfrak{u}_{t}|^2 \dd x \dd v .
\end{align*}

Gathering the previous bounds yields
\begin{align*}
\frac{\dd}{\dd t} \iint  \langle v\rangle^{2k} |F^\mathfrak{u}_{t}|^2 \dd x \dd v
+ \frac{\sigma^2}{2} \iint \langle v\rangle^{2k} |\nabla_{v} F^\mathfrak{u}_{t}|^2 \dd x \dd v \leq ( 1+ 2k \lVert \mathfrak{u}\rVert_{L^\infty_{t,x}} +2k^2\sigma^2) \iint \langle v\rangle^{2k} |F^\mathfrak{u}_{t}|^2 \dd x \dd v.
\end{align*}
By Gr\"onwall's Lemma, the conclusion holds.
\end{proof}

\section{Preliminary results and \emph{a priori} estimates on the fluid-particle system}
\label{sec:fluid-particle}

 We derive here the mild equation associated to $u^N$ in~\eqref{Part}. Recall that some important properties of the Leray projector are stated in Section~\ref{subsec:notations}. Apply first the Leray projector $P$ on the first equation of the system \eqref{Part}, commute with the differential operators, use that $P\nabla p^N=0$, $P\nabla u^N=0$ and $Pu^N=u^N$, then the mild formulation reads
\begin{equation}\label{eq:mildeq}
\begin{split}
u_{t}^{N} &= e^{t  \Delta}u_{0}^{N} - \int_{0}^{t}   e^{\left(  t-s\right)  \Delta} P\big[ (u^N_{s} \cdot \nabla)\bchi(u^N_s)\big] \dd s \\
&\quad - \int_{0}^{t}   e^{\left(  t-s\right)  \Delta} P\Big[ \frac{1}{N}\sum_{i=1}^{N} ( \bchi(u_{s}^{N}(X_{s}^{i,N}))-V_{s}^{i,N}) \, \delta_{X^{i,N}_s}^{N} \Big] \dd s.
\end{split}
\end{equation}

We start by stating a well-posedness result for the system~\eqref{Part}; the proof is detailed in Appendix~\ref{app:proof-prop-IPS}.
\begin{proposition}
\label{prop:existence-IPS}
Let Assumption~\ref{assump-particles}\ref{hyp:ICu0} holds and let $(\Omega,\mathcal{F},(\mathcal{F}_{t})_{t\geq 0},\PP)$ be a filtered probability space satisfying the usual conditions. Let $(B^i)_{i\in \N}$ be a family of independent $(\mathcal{F}_{t})_{t\geq 0}$-Brownian motions and let $(X^i_{0},V^i_{0})_{i\in \N}$ be $\mathcal{F}_{0}$-measurable random variables in $\T^d\times \R^d$. Then for each fixed $N\in \N^*$, there exists $\Omega_{N}\in \mathcal{F}$ and a unique process $(u^N,X^{1,N},V^{1,N},\dots,X^{N,N},V^{N,N})$ such that $\PP(\Omega_{N})=1$ and 
\begin{itemize}
\item for all $\omega\in \Omega_{N}$, $u^N$ is the unique solution of the mild equation~\eqref{eq:mildeq};

\item for all $\omega\in \Omega_{N}$, for each $i\in \{1,\dots, N\}$, $(X^{i,N},V^{i,N})$ solves the ODE from~\eqref{Part};

\item $(u^N,X^{1,N},V^{1,N},\dots,X^{N,N},V^{N,N})$ is measurable and adapted to the filtration $(\mathcal{F}_{t})_{t\geq 0}$.

\end{itemize}
\end{proposition}

\subsection{Preliminary computations}

\paragraph{Equation satisfied by $F^N$.}
From equation~\eqref{eq:defF}, we observe that $F^N$, the regularization of the empirical measure, verifies for all $x,v$,
\begin{equation}\label{defFN}
F_{t}^{N}(x,v) = \frac1N\sum_{i=1}^N \vartheta^N(x-X_t^{i,N},v-V_t^{i,N}).
\end{equation}
By Itô's formula on the process $(\vartheta^N(x-X_t^{i,N},v-V_t^{i,N}))_{t\geq 0}$, it comes
\begin{align*}\label{itothetaN}
\vartheta^{N}(x-X_t^{i,N},v-V_t^{i,N})
&= \vartheta^{N}(x-X_0^{i,N},v-V_0^{i,N})  - \int_0^t \nabla_x \vartheta^{1,N}(x-X_s^{i,N})\, \vartheta^{2,N}(v-V_s^{i,N}) \cdot V_s^{i,N} \dd s\\
&\quad - \int_0^t \vartheta^{1,N}(x-X_s^{i,N})\, \nabla_v \vartheta^{2,N}(v-V_s^{i,N}))\cdot \dd V_s^{i,N}\\
&\quad + \frac{\sigma^{2}_N}{2} \int_0^t \vartheta^{1,N}(x-X_s^{i,N}) \, \Delta_v \vartheta^{2,N} (v-V_s^{i,N}) \dd s.
\end{align*}
Recall the definition of the empirical measure $S^N$ in \eqref{eq:depempmeas}. By summing over $i=1,\dots,N$, we obtain \forxvt,
\begin{align*}%
F_{t}^{N}(x,v) &= F_{0}^{N}(x,v)
 - \int_{0}^{t} \langle S_s^N, \nabla_x \vartheta^{1,N}(x-\bullet_X)\, \vartheta^{2,N}(v-\bullet_V) \cdot \bullet_V\rangle \dd s\\
&\quad  - \int_{0}^{t} \langle S_s^N, \vartheta^{1,N}(x-\bullet_X)\, \nabla_v \vartheta^{2,N}(v-\bullet_V))\cdot (\bchi(u_s^N(\bullet_X))-\bullet_V)\rangle \dd s\\
&\quad + M_{t}^{N}(x,v) + \frac{\sigma^{2}_N}{2} \int_{0}^{t} \Delta_v F_{s}^{N}(x,v) \dd s,
\end{align*}
where we used $\bullet_X$ and $\bullet_V$ to represent integration variables over the measure $S_s^N$, and where
\begin{equation}\label{eq:defMN}
M_{t}^{N}(x,v) \coloneqq -\frac{\sigma_{N}}{N} \sum_{i=1}^{N} \int_{0}^{t} \nabla_{v} \vartheta^N(x-X^{i,N}_s, v-V^{i,N}_s) \cdot \dd B_{s}^{i}.
\end{equation}
To make notations lighter, we will use $\textbf{v}$ as a dummy variable, thus replacing  $\langle S_s^N, \nabla_x \vartheta^{1,N}(x-\bullet_X)\, \vartheta^{2,N}(v-\bullet_V) \cdot \bullet_V\rangle$ by $\langle \textbf{v} S_s^N, \nabla_x \vartheta^{1,N}(x-\bullet_X)\, \vartheta^{2,N}(v-\bullet_V) \rangle$, or even
$\dive_{x} (\vartheta^N\ast (\dv S_{s}^{N}))$.
Hence \forxvt,
 \begin{align}\label{equaregula}
F_{t}^{N}(x,v)
&= F_{0}^{N}(x,v)
 - \int_{0}^{t} \dive_{x} \left(\vartheta^N\ast (\dv S_{s}^{N})\right)(x,v) \dd s\nonumber\\
&\quad - \int_{0}^{t} \dive_{v} \left(\vartheta^{N}\ast \left((\bchi(u_{s}^{N})-\dv )S_{s}^{N}\right)\right)(x,v) \dd s\\
&\quad+ M_{t}^{N}(x,v) + \frac{\sigma^{2}_N}{2} \int_{0}^{t} \Delta_v F_{s}^{N}(x,v) \dd s.\nonumber
\end{align}

\paragraph{Useful formulas.}
In the sequel, the following simple equalities will be used several times.

For $\mu$ a probability measure on $\T^d\times \R^d$,
\begin{align}\label{eq:mollvmu}
\vartheta^N \ast (\dv \mu)(x,v)
&= \iint v'\, \vartheta^N(x-x',v-v')\, \mu(\dd x',\dd v') \nonumber\\
&= v \iint \vartheta^N(x-x',v-v')\, \mu(\dd x',\dd v') - \iint (v-v')\, \vartheta^N(x-x',v-v')\, \mu(\dd x',\dd v') \nonumber\\
&= v\, \vartheta^N\ast \mu(x,v) - (\dv\vartheta^N)\ast\mu(x,v).
\end{align}

\medskip

We will make use of the properties of $\vartheta$ to deduce several rates throughout the computations:

\begin{itemize}
\item Using the compact support of $\vartheta^2$, see \eqref{hyp:theta2}, one easily gets that for any $\mu,\nu \in \mathcal{P}(\T^d\times \R^d)$,
\begin{equation}\label{eq:rateTheta2}
\big|(\dv \vartheta^N)\ast(\mu-\nu)\big| \leq N^{-\alpha} (\vartheta^N\ast \mu + \vartheta^N\ast \nu) .
\end{equation}

\item For any $(x,v)\in \T^d\times \R^d$ and $N\in \N^*$, we have, using \eqref{hyp:theta1}:
\begin{align*}
|\nabla_{x} \vartheta^N(x,v) \cdot v |
&\leq N^{(d+1)\beta} |\nabla_{x}\vartheta^1(N^\beta x)| \, N^{d\alpha}\vartheta^2(N^\alpha v)\, |v| \\
&\lesssim \vartheta^{1,N}(x) \, |v| N^{d\alpha} \, \vartheta^2(N^\alpha v)\, N^\beta ;
\end{align*}
then using that $\vartheta^2$ has compact support, see \eqref{hyp:theta2}, and that $\alpha > \beta$, by Assumption~\ref{assump-particles}\ref{hyp:alphabeta}, we deduce that
\begin{align}\label{eq:boundVGradTheta}
|\nabla_{x} \vartheta^N(x,v) \cdot v | \lesssim N^{\beta-\alpha}\, \vartheta^N(x,v) \leq \vartheta^N(x,v).
\end{align}

\end{itemize}

\subsection{\emph{A priori} estimate on $F_{t}^{N}$}

\begin{proposition}\label{moment}
Let Assumption~\ref{assump-particles} hold.
For $k$ as in Assumption~\ref{assump-particles}\ref{hyp:F0Nk} and for any cut-off $A>0$ from \eqref{def:cutoff}, there exists $C = C(T,A,k,d)>0$ such that
\begin{equation}
\sup_{N\in \N^*}  e^{-C T \sigma_{N}^{-2}} \bigg(\E \sup_{t\in[0,T]} \big\| \langle v \rangle^{k} F_{t}^{N} \big\|_{L^{2}_{x,v}}^{2q}
+ \sigma_N^{2q}\, \E \bigg[ \Big(\int_{0}^T \big\lVert \langle v \rangle^k \nabla_{v} F^N_{s}\big\rVert_{L^{2}_{x,v}}^{2} \dd s\Big)^q\bigg] \bigg)^{1/q}<\infty.
\end{equation}
\end{proposition}

\begin{proof} From Equation \eqref{equaregula}, we apply It\^o's formula to the process $F_t^N$ to obtain
\begin{align*}
|F_{t}^{N}|^{2}
=& |F_{0}^{N}|^{2}
 -2 \int_{0}^{t} F_s^N\dive_{x} \left(\vartheta^N\ast (\dv S_{s}^{N})\right) \dd s -2 \int_{0}^{t} F_s^N\dive_{v} \left(\vartheta^{N}\ast \left((\bchi(u_{s}^{N})-\dv )S_{s}^{N}\right) \right) \dd s \\
&- \frac{2\sigma_{N}}{N} \sum_{i=1}^{N} \int_{0}^t F_{s}^{N} \nabla_{v} \vartheta^N(x-X^{i,N}_s, v-V^{i,N}_s) \cdot \dd B_{s}^{i}\\
& + \frac{\sigma^{2}_N}{N} \int_{0}^t (S_{s}^{N} \ast |\nabla_v \vartheta^N|^{2}) \dd s
+\sigma^2_N \int_0^t F_s^N \Delta_v F_s^N \dd s.
\end{align*}
Multiplying by $\langle v \rangle^{2k}$, integrating on the variables $x\in\T^d$ and $v\in\R^d$ (which is possible since $F^N$ is almost surely compactly supported), and by integration-by-parts, we have
\begin{equation}\label{eq:decompMomentsF}
\begin{split}
\iint  & \langle v \rangle^{2k} |F_{t}^{N}|^{2} \dd x \dd v
- \iint \langle v \rangle^{2k} |F_{0}^{N}|^{2} \dd x \dd v \\
=& -2 \int_{0}^t \iint \langle v \rangle^{2k}  F_{s}^{N}  \dive_{x}( \vartheta^N\ast (\dv S_{s}^{N})) \dd x \dd v \dd s \\
& - 2\int_{0}^t \iint\langle v \rangle^{2k}  F_{s}^{N}  \dive_{v}( \vartheta^N\ast ((\bchi(u_s^{N})-\dv) S_{s}^{N})) \dd x \dd v \dd s \\%
& - \frac{2 \sigma_{N}}{N} \sum_{i=1}^{N}\iint  \int_{0}^t \langle v \rangle^{2k} F_{s}^{N} \nabla_{v}\vartheta^N(x-X^{i,N}_s, v-V^{i,N}_s) \cdot \dd B_{s}^{i} \dd x \dd v\\
& + \frac{\sigma_N^{2}}{N} \int_{0}^t \iint\langle v \rangle^{2k}  (S_{s}^{N} \ast |\nabla_v \vartheta^N|^{2}) \dd x \dd v \dd s %
  - \sigma_N^{2} \int_{0}^t \iint  \langle v \rangle^{2k}  |\nabla_v F_{s}^{N}|^{2} \dd x \dd v \dd s \\
&  -\sigma_N^2 \int_{0}^t \iint 2k \langle v \rangle^{2k-2} v \cdot \big( F_{s}^{N} \nabla_v F_{s}^{N} \big) \dd x \dd v \dd s\\ %
\eqqcolon& \sum_{k=1}^6 I_{k}(t) .
\end{split}
\end{equation}

\paragraph{Bound on $I_{1}$.} Using \eqref{eq:mollvmu}, we have
\begin{align}\label{eq:decompconvol1}
\vartheta^N\ast (\dv S_{s}^{N})(x,v)
&= v F^N_{s}(x,v) - (\dv \vartheta^N)\ast S^N_{s}(x,v),
\end{align}
so that
\begin{align*}%
\dive_{x}( \vartheta^N\ast (\dv S_{s}^{N}))(x,v)
&= v\cdot \nabla_{x} F_{s}^{N}(x,v) - (\nabla_{x} \vartheta^N \cdot \dv) \ast S_{s}^{N}(x,v).
\end{align*}
The previous equality yields
\begin{align}\label{eq:1111}
\iint \langle &v \rangle^{2k}  F_{s}^{N} \dive_{x}( \vartheta^N\ast (\dv S_{s}^{N})) \dd x \dd v \nonumber\\
&=\iint \langle v \rangle^{2k} F_{s}^{N}\, v\cdot \nabla_{x} F_{s}^{N} \dd x \dd v  -\iint \langle v \rangle^{2k} F_{s}^{N} \,  (\nabla_{x} \vartheta^N \cdot \dv) \ast  S_{s}^{N} \dd x \dd v \nonumber\\
&= \frac{1}{2}\iint \langle v \rangle^{2k}  v\cdot \nabla_{x} (F_{s}^{N})^2 \dd x \dd v
 -\iint \langle v \rangle^{2k} F_{s}^{N} \,  (\nabla_{x} \vartheta^N \cdot \dv) \ast  S_{s}^{N} \dd x \dd v \nonumber\\
&= -\iint \langle v \rangle^{2k} F_{s}^{N} (\nabla_{x} \vartheta^N \cdot \dv) \ast  S_{s}^{N} \dd x \dd v .
\end{align}
Using \eqref{eq:boundVGradTheta} thus yields
\begin{align}\label{eq:boundI1}
|I_{1}(t)| = 2 \Big|\int_{0}^t\iint \langle v \rangle^{2k}  F_{s}^{N} \, (\nabla_{x} \vartheta^{N} \cdot \dv) \ast  S_{s}^{N} \dd x \dd v \dd s\Big|
 \lesssim \int_{0}^t\iint \langle v \rangle^{2k} | F_{s}^{N}|^{2} \dd x \dd v \dd s.
\end{align}

\paragraph{Bound on $I_{2}$.}
For this term, we have
\begin{align*}
\iint \langle v \rangle^{2k}  F_{s}^{N}  \dive_{v}( \vartheta^N\ast ((\bchi(u_s^{N})-\dv) S_{s}^{N})) \dd x \dd v
 & = \iint \langle v \rangle^{2k}  F_{s}^{N}  \dive_{v}\left( \vartheta^{N}\ast  (\bchi(u_{s}^{N}) S_{s}^{N})\right) \dd x \dd v \\
&\quad - \iint \langle v \rangle^{2k}  F_{s}^{N}  \dive_{v}(\vartheta^{N}\ast (\dv S_{s}^{N})) \dd x \dd v\\
&\eqqcolon I_{2,1} - I_{2,2}.
\end{align*}
Doing an integration-by-parts, using that $\sup_{x} |\bchi(x)|\leq 1+A$ and then Young's inequality, we get
\begin{align}\label{tres}
| I_{2,1} |
& \leq (1+A)\iint \langle v \rangle^{2k}  |\nabla_{v} F_{s}^{N}| \,F_{s}^{N} \dd x \dd v + C (1+A)\iint \langle v \rangle^{2k-1}  | F_{s}^{N}|^{2} \dd x \dd v \nonumber \\
&  \leq  \delta \iint \langle v \rangle^{2k}  |\nabla_{v} F_{s}^{N}|^{2}  \dd x \dd v  +  \frac{(1+A)^{2}}{4\delta} \iint \langle v \rangle^{2k}  | F_{s}^{N}|^{2} \dd x \dd v \nonumber \\
& \quad +  C (1+A)  \iint \langle v \rangle^{2k}  | F_{s}^{N}|^{2} \dd x \dd v ,
\end{align}
for some $ \delta >0 $ to be chosen later.
Now, using \eqref{eq:decompconvol1}, we have
\begin{align*}
I_{2,2} &=  \iint \langle v \rangle^{2k}  F_{s}^{N}   \dive_{v} (\dv  F_{s}^{N}) \dd x \dd v -  \iint \langle v \rangle^{2k}  F_{s}^{N}  \dive_{v}( (\dv \vartheta^{N})\ast  S_{s}^{N}) \dd x \dd v\\
&\eqqcolon I_{2,2,1} - I_{2,2,2}.
\end{align*}
For $I_{2,2,1}$, we use $ \iint \langle v \rangle^{2k} F_{s}^{N} \dive_{v} (\dv F_{s}^{N}) \dd x \dd v =  \iint \langle v \rangle^{2k} |F_{s}^N|^2 + \frac{1}{2} \langle v \rangle^{2k}  v\cdot \nabla_{v} |F_{s}^N|^2  \dd x \dd v$ and integrate-by-parts on the second summand to get
\begin{align}\label{eq:boundI221}
|I_{2,2,1}| \lesssim \iint \langle v \rangle^{2k} | F_{s}^{N}|^{2} \dd x \dd v.
\end{align}
As for $I_{2,2,2}$, doing again an integration-by-parts yields
\begin{align*}
| I_{2,2,2} | & \leq   \iint \langle v \rangle^{2k} |\nabla_{v}  F_{s}^{N}|  | (\dv \vartheta^{N})\ast  S_{s}^{N})| \dd x \dd v +  C \iint \langle v \rangle^{2k-1} F_{s}^{N}\,  | (\dv \vartheta^{N})\ast  S_{s}^{N})| \dd x \dd v .
\end{align*}
Now applying Young's inequality and, as in \eqref{eq:boundVGradTheta} using that $\vartheta^2$ has compact support to deduce $| (\dv \vartheta^{N})\ast  S_{s}^{N})| \lesssim F_{s}^{N}$, one gets that
\begin{align}\label{eq:boundI222}
| I_{2,2,2} | & \leq  \delta \iint \langle v \rangle^{2k} |\nabla_{v}  F_{s}^{N}|^{2} \dd x \dd v +  (C+\frac{1}{4\delta}) \iint \langle v \rangle^{2k} | F_{s}^{N}|^{2} \dd x \dd v .
\end{align}
Hence \eqref{eq:boundI221} and \eqref{eq:boundI222} yield
\begin{align}\label{cinco}
| I_{2,2} | \leq  \delta \iint \langle v \rangle^{2k} |\nabla_{v}  F_{s}^{N}|^{2} \dd x \dd v +  (C+\frac{1}{4\delta}) \iint \langle v \rangle^{2k} | F_{s}^{N}|^{2} \dd x \dd v ,
\end{align}
so that, using \eqref{tres} and \eqref{cinco},
\begin{equation}\label{eq:boundI2}
|I_{2}(t)| \leq 2\delta \int_{0}^t \iint \langle v \rangle^{2k} |\nabla_{v}  F_{s}^{N}|^{2} \dd x \dd v \dd s +  (C+\frac{1}{4\delta}) \int_{0}^t \iint \langle v \rangle^{2k} | F_{s}^{N}|^{2} \dd x \dd v \dd s.
\end{equation}

\paragraph{Bound on $I_{3}$.}
By the stochastic Fubini theorem,
\begin{align*}
\E \sup_{s\in [0,t]} |I_{3}(s)|^q  = \E \sup_{s\in [0,t]} \Big| \int_{0}^{s} \frac{2\sigma_{N}}{N} \sum_{i=1}^{N}   \iint \langle v \rangle^{2k} F_{r}^{N} \nabla_{v} \vartheta^{N}(x-X_{r}^{i,N}, v-V_{r}^{i,N}) \dd x \dd v \cdot \dd B_{r}^{i} \Big|^{q}.
\end{align*}
It follows from the Burkholder-Davis-Gundy inequality that
\begin{align*}
\E \sup_{s\in [0,t]} |I_{3}(s)|^q
& \leq C_{q} \frac{\sigma_{N}^q}{N^{q}} \E\bigg[ \Big( \sum_{i=1}^{N}  \int_{0}^{t}  \Big| \iint \langle v \rangle^{2k}  F_{r}^{N} \nabla_{v} \vartheta^{N}(x-X_{r}^{i,N}, v-V_{r}^{i,N})  \dd x \dd v \Big|^{2} \dd r \Big)^{\frac{q}{2}} \bigg] \\
& = C_{q} \frac{\sigma_{N}^q}{N^{q}} \E\bigg[ \Big( \sum_{i=1}^{N}  \int_{0}^{t}  \Big| \iint \frac{\langle v \rangle^{2k+d+\varepsilon}}{\langle v \rangle^{d+\varepsilon}}  F_{r}^{N} \nabla_{v} \vartheta^{N}(x-X_{r}^{i,N}, v-V_{r}^{i,N})  \dd x \dd v \Big|^{2} \dd r \Big)^{\frac{q}{2}} \bigg],
\end{align*}
for any $\varepsilon>0$. Then by the Cauchy-Schwarz inequality, we get
\begin{align*}
\E \sup_{s\in [0,t]} |I_{3}(s)|^q
& \lesssim \frac{\sigma_{N}^q}{N^{q}} \E \bigg[ \Big(\sum_{i=1}^{N} \int_{0}^{t} \iint  \langle v \rangle^{4k+2d+2\varepsilon} | F_{r}^{N}|^{2}  |\nabla_{v} \vartheta^{N}(x-X_{r}^{i,N}, v-V_{r}^{i,N})|^{2} \dd x \dd v \dd r\Big)^{\frac{q}{2}} \bigg] .
\end{align*}
Now observe that
\begin{align*}
&\iint  \langle v \rangle^{4k+2d+2\varepsilon} | F_{r}^{N}|^{2}  |\nabla_{v} \vartheta^{N}(x-X_{r}^{i,N}, v-V_{r}^{i,N})|^{2} \dd x \dd v \\
&\quad = \iint  \langle v + V_{r}^{i,N} \rangle^{2k+2d+2\varepsilon} \langle v + V_{r}^{i,N} \rangle^{2k} | F_{r}^{N}(x,v+V_{r}^{i,N})|^{2}  |\nabla_{v} \vartheta^{N}(x-X_{r}^{i,N}, v)|^{2} \dd x \dd v .
\end{align*}
Using the inequality $(1+|a+b|^2) \leq 2(1+|a|^2) (1+|b|^2)$,
\begin{align*}
 \langle v + V_{r}^{i,N} \rangle^{2k+2d+2\varepsilon} |\nabla_{v} \vartheta^N(x-X_{r}^{i,N}, v)|^{2}
&  \lesssim  \langle V_{r}^{i,N} \rangle^{2k+2d+2\varepsilon} \,  \langle v  \rangle^{2k+2d+2\varepsilon} |\nabla_{v} \vartheta^N(x-X_{r}^{i,N}, v)|^{2} \\
&  \lesssim  \langle V_{r}^{i,N} \rangle^{2k+2d+2\varepsilon}  |\nabla_{v} \vartheta^N(x-X_{r}^{i,N}, v)|^{2} \\
&\lesssim  \langle V_{r}^{i,N} \rangle^{2k+2d+2\varepsilon}\, N^{2(d+1)\alpha + 2d\beta} ,
\end{align*}
where it was used in the second line that $\vartheta^{2,N}$ has compact support (uniformly in $N$), and in the third line that $ \lVert \nabla_{v} \vartheta^N \rVert_{L^\infty(\T^d\times \R^d)} \lesssim N^{(d+1)\alpha + d\beta}$.
Then we get that
\begin{align*}
&\E \sup_{s\in [0,t]} |I_{3}(s)|^q \\
&\quad \lesssim \frac{\sigma_{N}^q N^{q((d+1)\alpha + d\beta)}}{N^{q}} \E \bigg[ \Big(\sum_{i=1}^{N} \int_{0}^{t} \langle V_{r}^{i,N} \rangle^{2k+2d+2\varepsilon} \iint  \langle v+V_{r}^{i,N} \rangle^{2k} | F_{r}^{N}(x,v+V_{r}^{i,N})|^{2}  \dd x \dd v \dd r\Big)^{\frac{q}{2}} \bigg]\\
&\quad \lesssim \frac{\sigma_{N}^q}{N^{q(\frac{1}{2}-(d+1)\alpha - d\beta)}} \E \bigg[ \Big(\frac{1}{N}\sum_{i=1}^{N} \sup_{r\in [0,T]} \langle V_{r}^{i,N} \rangle^{2k+2d+2\varepsilon} \int_{0}^{t} \iint  \langle v \rangle^{2k} | F_{r}^{N}|^{2}  \dd x \dd v \dd r\Big)^{\frac{q}{2}} \bigg] .
\end{align*}
Hence, applying Young's inequality first, then Jensen's inequality,
\begin{align}\label{siete}
&\E \sup_{s\in [0,t]} |I_{3}(s)|^q \nonumber\\
&\quad\lesssim  \frac{\sigma_{N}^q}{N^{q(\frac{1}{2}-(d+1)\alpha - d\beta)}} \E \bigg[ \Big(\frac{1}{N}\sum_{i=1}^{N} \sup_{r\in [0,T]} \langle V_{r}^{i,N} \rangle^{2k+2d+2\varepsilon}\Big)^q + \Big( \int_{0}^{t} \iint  \langle v \rangle^{2k} | F_{r}^{N}|^{2}  \dd x \dd v \dd r\Big)^q \bigg] \nonumber \\
&\quad\lesssim  \E \bigg[ \Big( \int_{0}^{t} \iint  \langle v \rangle^{2k} | F_{r}^{N}|^{2}  \dd x \dd v \dd r\Big)^q \bigg]
+\frac{\sigma_{N}^q}{N^{q(\frac{1}{2}-(d+1)\alpha - d\beta)}} \frac{1}{N}\sum_{i=1}^{N} \E \Big[ \sup_{r\in [0,T]} \langle V_{r}^{i,N} \rangle^{q(2k+2d+2\varepsilon)} \Big] .
\end{align}
Finally, observe that for any $r\geq 0$,
\begin{equation*}
V_{r}^{i,N} = e^{-r} V_{0}^{i,N} + \int_{0}^r e^{-(r-s)} \bchi(u_{s}(X^{i,N}_{s})) \dd s + \int_{0}^r e^{-(r-s)} \dd B^i_{s},
\end{equation*}
which implies, since $\bchi$ is a bounded function, that all the velocities $V^{i,N}$ have moments that are bounded uniformly in $i$ and $N$, i.e. for any $m>0$,
\begin{equation}\label{eq:momentsV}
\sup_{N\in \N^*} \sup_{i\in\{1,\dots, N\}} \E \Big[ \sup_{r\in [0,T]} \langle V_{r}^{i,N} \rangle^{m} \Big]  <\infty.
\end{equation}
Using \eqref{eq:momentsV} in \eqref{siete} yields
\begin{equation}\label{eq:boundI3}
\E \sup_{s\in [0,t]} |I_{3}(s)|^q \lesssim \E \bigg[ \Big( \int_{0}^{t} \iint  \langle v \rangle^{2k} | F_{r}^{N}|^{2}  \dd x \dd v \dd r\Big)^q \bigg]
+\frac{\sigma_{N}^q}{N^{q(\frac{1}{2}-(d+1)\alpha - d\beta)}} .
\end{equation}

\paragraph{Bound on $I_{4}$.} Separating variables and performing a change of variables,
\begin{align*}
&\frac{\sigma_{N}^2}{N} \iint \langle v \rangle^{2k} (S_{s}^{N} \ast |\nabla_{v} \vartheta^{N}|^{2}) \dd x \dd v \\
&\quad = \frac{\sigma_{N}^2}{N^2} \sum_{i=1}^{N}  \int_{\R^{d}}  \langle v \rangle^{2k} |\nabla_{v} \vartheta^{2,N}(v-V_{s}^{i,N})|^{2} \dd v  \int_{\T^{d}} |\vartheta^{1,N}(x-X_{s}^{i,N})|^{2} \dd x \\
&=  \frac{\sigma_{N}^2}{N^{2-d\beta- (d+2)\alpha}}  \sum_{i=1}^{N}  \int_{\R^{d}}  \big(1+ |N^{-\alpha}v + V^{i,N}_{s}|^2 \big)^k |\nabla_{v} \vartheta^{2}(v)|^{2} \dd v  \int_{\T^{d}} |\vartheta^{1}(x)|^{2} \dd x.
\end{align*}
Then using the inequalities $(1+|a+b|^2) \leq 2(1+|a|^2) (1+|b|^2)$ and $\langle v \rangle^{2k} |\nabla_{v} \vartheta^2(v)| \lesssim |\nabla_{v} \vartheta^2(v)|$, which holds since $\vartheta^2$ has compact support, we get
\begin{align}\label{seis-0}
&\frac{\sigma_{N}^2}{N} \iint \langle v \rangle^{2k} (S_{s}^{N} \ast |\nabla_{v} \vartheta^{N}|^{2}) \dd x \dd v \nonumber\\
&\quad \lesssim  \frac{\sigma_{N}^2}{N^{2-d\beta- (d+2)\alpha}}  \int_{\R^{d}} |\nabla_{v} \vartheta^{2}(v)|^{2} \dd v  \int_{\T^{d}}  |\vartheta^{1}(x)|^{2} \dd x ~ \sum_{i=1}^{N}  \langle V_{s}^{i,N}\rangle^{2k}  \nonumber\\
&\quad \lesssim  \frac{\sigma_{N}^2}{N^{2- (d+2)\alpha-d\beta}}  \sum_{i=1}^{N}  \langle V_{s}^{i,N}\rangle^{2k} .
\end{align}
Hence it follows from \eqref{seis-0} and the bound \eqref{eq:momentsV} on the moments of $\langle V_{s}^{i,N}\rangle$ that
\begin{equation}\label{eq:boundI4}
\E \sup_{s\in [0,t]}|I_{4}(s)|^q \lesssim \frac{\sigma_{N}^{2q}}{N^{q(1- (d+2)\alpha-d\beta)}}.
\end{equation}

\paragraph{Bound on $I_{6}$.}
For this term, as in \eqref{eq:boundI221}, integrate-by-parts to get
\begin{align}\label{eq:boundI6}
|I_{6}(t)| &= \sigma_N^2 \left| \int_{0}^t\iint 2k \langle v \rangle^{2k-2} v \cdot \big(F_{s}^{N} \nabla_v F_{s}^{N}\big) \dd x \dd v \dd s \right| \nonumber\\
&=\sigma_N^2 \left| \int_{0}^t\iint k \langle v \rangle^{2k-2} v \cdot  \nabla_v |F_{s}^{N}|^2 \dd x \dd v \dd s \right|\nonumber\\
 &\lesssim \sigma_N^2 \int_{0}^t \iint \langle v \rangle^{2k-2}  |F_{s}^{N}|^{2} \dd x \dd v \dd s \nonumber\\
&\lesssim \sigma_N^2\int_{0}^t \iint \langle v \rangle^{2k}  |F_{s}^{N}|^{2}  \dd x \dd v \dd s.
\end{align}
\paragraph{Conclusion.}
Choosing $\delta \equiv \delta_{N} = \frac{\sigma_{N}^2}{4}$, passing $-  \frac{\sigma_N^{2}}{2} \int_{0}^t \iint \langle v \rangle^{2k} |\nabla_{v}  F_{s}^{N}|^{2} \dd x \dd v \dd s$ coming from \eqref{eq:boundI2} and $I_{5}$ to the left-hand side of \eqref{eq:decompMomentsF}, one can now take the $L^q(\Omega)$ norm and plug the bounds \eqref{eq:boundI1}, \eqref{eq:boundI2}, \eqref{eq:boundI3}, \eqref{eq:boundI4} and \eqref{eq:boundI6} into \eqref{eq:decompMomentsF} to obtain
\begin{equation}\label{eq:vkFN-intermediate}
\begin{split}
& \Bigg(\E  \bigg[ \sup_{s\in [0,t]} \big\| \langle v \rangle^{k} F_{s}^{N}  \big\|_{L^{2}_{x,v}}^{2} + \frac{\sigma_N^{2}}{2} \int_{0}^t \big\lVert \langle v \rangle^k \nabla_{v} F^N_{s} \big\rVert_{L^{2}_{x,v}}^{2} \dd s \bigg]^q\Bigg)^\frac{1}{q} \\
 &\quad \leq  \big(\E \big\| \langle v \rangle^{k} F_{0}^{N} \big\|_{L^{2}_{x,v}}^{2q}\big)^{\frac{1}{q}}
 + C (1+\sigma_{N}^2 + \sigma_{N}^{-2}) \int_{0}^{t} \big(\E \sup_{r\in [0,s]} \big\| \langle v \rangle^{k} F_{r}^{N} \big\|_{L^{2}_{x,v}}^{2q}\big)^{\frac{1}{q}}\, \dd s \\
 &\quad\quad + C \frac{\sigma_{N}^{2}}{N^{1- (d+2)\alpha-d\beta}}
 +C \frac{\sigma_{N}}{N^{\frac{1}{2}-(d+1)\alpha - d\beta}} .
\end{split}
\end{equation}
For fixed $N$, using the smoothness and compact support of $\vartheta^N$, it is clear that for any $t\in [0,T]$, $\E \big[ \big(\int_{0}^t \lVert \langle v \rangle^k \nabla_{v} F^N_{s}\rVert_{L^{2}_{x,v}}^{2} \dd s\big)^q\big]$ is finite. Recall also that $\E \| \langle v \rangle^{k} F_{0}^{N} \|_{L^{2}_{x,v}}^{2q}<\infty$ by Assumption~\ref{assump-particles}\ref{hyp:F0Nk}.
Then applying Gr\"onwall's Lemma in \eqref{eq:vkFN-intermediate}, we get that for any $t\in [0,T]$,
\begin{align*}
& \Big(\E \sup_{s\in [0,t]} \big\| \langle v \rangle^{k} F_{s}^{N} \big\|_{L^{2}_{x,v}}^{2q} \Big)^{\frac{1}{q}} + \frac{\sigma_N^{2q}}{2} \bigg(\E \bigg[ \Big(\int_{0}^t \big\lVert \langle v \rangle^k \nabla_{v} F^N_{s}\big\rVert_{L^{2}_{x,v}}^{2} \dd s\Big)^q\bigg]\bigg)^\frac{1}{q} \\
 &\quad \leq \bigg( \big(\E \big\| \langle v \rangle^{k} F_{0}^{N} \big\|_{L^{2}_{x,v}}^{2q}\big)^{\frac{1}{q}}
 + C \frac{\sigma_{N}^{2}}{N^{1- (d+2)\alpha-d\beta}}
 +C \frac{\sigma_{N}}{N^{\frac{1}{2}-(d+1)\alpha - d\beta}} \bigg)
\exp \left(C (1+\sigma_{N}^2 + \sigma_{N}^{-2}) t\right) .
\end{align*}
In view of Assumption~\ref{assump-particles}\ref{hyp:alphabeta}, one has $1- (d+2)\alpha-d\beta> 2(\frac{1}{2}-(d+1)\alpha - d\beta)$, so with Assumption~\ref{assump-particles}\ref{hyp:sigma} and \ref{hyp:F0Nk}, this entails that the right-hand side of the previous inequality is bounded by
\begin{equation*}
\Big(\sup_{N\in \N^*} \big(\E \big\| \langle v \rangle^{k} F_{0}^{N} \big\|_{L^{2}_{x,v}}^{2q}\big)^{\frac{1}{q}} + C \Big) \exp \left(C \sigma_{N}^{-2} T\right),
\end{equation*}
which leads to the result.
\end{proof}

\subsection{\emph{A priori} estimate on $u_{t}^{N}$}

Recall that the mild formulation satisfied by $u^N$ is given in~\eqref{eq:mildeq}.

\begin{proposition}\label{propu}
Let Assumption~\ref{assump-particles} hold with $p>d$ and $\gamma\in (\frac{d}{p}, 1)$. If $d=3$, assume further that $\frac{\gamma}{2} + \frac{3}{2}(\frac{1}{2}-\frac{1}{p}) <1$.
Then there exists $C = C(T,A,k,\gamma,p,d)>0$ such that for any $q \geq 1$,
\begin{equation*}
\sup_{N\in\N^*}   e^{-C T \sigma_{N}^{-2}} \, \Big(\mathbb{E}\Big[\sup_{t\in[0,T]} \left\Vert  u^N_{t} \right\Vert _{\gamma,p}^{q}\Big]\Big)^\frac{1}{q}  <\infty.
\end{equation*}
\end{proposition}

\begin{proof}
From \eqref{eq:mildeq} and by the triangle inequality we have
\begin{align}
\bigl \Vert  u^N_t  \bigr \Vert _{\gamma, p}
&\leq \left\Vert e^{t\Delta}u^N_0\right\Vert _{\gamma, p}  + \int_{0}^{t}\left\Vert e^{\left(  t-s\right)  \Delta} P \big[ (u^N_{s} \cdot \nabla)\bchi(u^N_s)\big] \right\Vert _{\gamma,p} \dd s \label{eq:second} \vphantom{\Bigg(}\\
&\quad  +  \int_{0}^{t}\Big\Vert  e^{\left(  t-s\right)  \Delta} P\Big[ \frac{1}{N}\sum_{i=1}^{N} ( \bchi(u_{s}^{N}(X_{s}^{i,N}))-V_{s}^{i,N}) \, \delta_{X^{i,N}_s}^{N} \Big] \Big\Vert _{\gamma,p} \dd s. \label{eq:third}
\end{align}

For $\Vert e^{t\Delta}u^N_0\Vert _{\gamma, p}$, by a convolution inequality, it holds
\begin{equation}\label{eq:Prop32-bound1}
\Vert e^{t\Delta} u^N_0\Vert_{\gamma, p}  \leq \| e^{t\Delta}\|_{L^p\to L^p}  \Vert  u^N_0\Vert_{\gamma, p} = \|u_{0}^N\|_{\gamma, p}.
\end{equation}

 Let us come to the second term in the right-hand side of \eqref{eq:second}.
 By Lemma~\ref{lem:divfreecommute}, one has
\begin{align*}
 \int_{0}^{t}\left\Vert e^{\left(  t-s\right)  \Delta} P \big[ (u^N_{s} \cdot \nabla)\bchi(u^N_s)\big] \right\Vert _{\gamma,p} \dd s
 &\lesssim \int_{0}^{t}  \frac{1}{(t-s)^{\frac{1+\gamma}{2}}} \left\Vert u_{s}^{N} \otimes \bchi(u^N_s) \right\Vert _{L^p_x}  \dd s,
\end{align*}
and using the boundedness of $\bchi$, it follows that
\begin{align}\label{eq:Prop32-bound2}
 \int_{0}^{t}\left\Vert e^{\left(  t-s\right)  \Delta} P \big[ (u^N_{s} \cdot \nabla)\bchi(u^N_s)\big] \right\Vert _{\gamma,p} \dd s
 &\lesssim (1+A) \int_{0}^{t}  \frac{1}{(t-s)^{\frac{1+\gamma}{2}}} \left\Vert u_{s}^{N} \right\Vert _{L^p_x}  \dd s.
\end{align}

Similarly, for the third term \eqref{eq:third} we have, using the Gaussian estimate~\eqref{eq:Bessel-heat-estimate} and the fact that $P$ is continuous in $L^p$ spaces,
\begin{align*}
\int_{0}^{t}&  \Big\Vert e^{\left(t-s\right) \Delta} P\Big[ \frac{1}{N}\sum_{i=1}^{N} ( \bchi(u_{s}^{N}(X_{s}^{i,N}))-V_{s}^{i,N}) \delta_{X^{i,N}_s}^{N} \Big] \Big\Vert _{\gamma, p} \dd s \\
&\lesssim \int_{0}^{t} \frac{1}{(t-s)^{\frac{\gamma}{2} + \frac{d}{2}(\frac{1}{2}-\frac{1}{p})}} \Big\Vert \frac{1}{N}\sum_{i=1}^{N} ( \bchi(u_{s}^{N}(X_{s}^{i,N}))-V_{s}^{i,N}) \delta_{X^{i,N}_s}^{N}  \Big\Vert _{L^2_{x}} \dd s,
\end{align*}
with the integral in time being finite under the assumption on $\gamma$ and $p$.
Now we express the previous sum in terms of the moments defined in \eqref{eq:defMoments}: using the definitions of the mollified delta function \eqref{eq:mollifieddelta} and the moments in velocity \eqref{eq:defMoments}, using the symmetry assumption on $\vartheta^2$, see \eqref{hyp:theta2}, we get that
\begin{align*}
\Big| \frac{1}{N}\sum_{i=1}^{N} V_{s}^{i,N} \delta_{X^{i,N}_s}^{N}(x) \Big|
&= \Big| -\int_{\R^d} \frac{1}{N}\sum_{i=1}^{N} V_{s}^{i,N} \delta_{X^{i,N}_s}^{N}(x)\, \vartheta^{2,N}(v-V_{s}^{i,N}) \dd v \Big|\\
&= \Big| \int_{\R^d} \frac{1}{N}\sum_{i=1}^{N} v\, \delta_{X^{i,N}_s}^{N}(x)\, \vartheta^{2,N}(v-V_{s}^{i,N}) \dd v \Big|\\
& \leq \int_{\R^d} \big| v F^N_{s}(x,v) \big|  \dd v = m_{1}(F^N_{s})(x).
\end{align*}
Hence
\begin{align*}
\Big| \frac{1}{N}\sum_{i=1}^{N} ( \bchi(u_{s}^{N}(X_{s}^{i,N}))-V_{s}^{i,N}) \delta_{X^{i,N}_s}^{N} \Big|
&\leq (1+A) \frac{1}{N}\sum_{i=1}^{N} \delta_{X^{i,N}_s}^{N} + \Big| \frac{1}{N}\sum_{i=1}^{N} V_{s}^{i,N} \delta_{X^{i,N}_s}^{N} \Big| \\
&\leq (1+A)\, m_{0}(F^N_{s}) + m_{1}(F^N_{s}).
\end{align*}
Using the Gaussian estimate~\eqref{eq:Bessel-heat-estimate} and the fact that $P$ is continuous in $L^p$ yields
\begin{align}\label{eq:Prop32-bound3}
\int_{0}^{t} &\Big\Vert   e^{\left(  t-s\right)  \Delta}P \frac{1}{N}\sum_{i=1}^{N} ( \bchi(u_{s}^{N}(X_{s}^{i,N}))-V_{s}^{i,N}) \delta_{X^{i,N}_s}^{N} \Big\Vert _{\gamma, p} \dd s \nonumber \\
&\lesssim \Big( \sup_{s\in [0,T]} \lVert m_{0}(F^N_{s}) \rVert_{L^2_{x}} + \lVert m_{1}(F^N_{s}) \rVert_{L^2_{x}} \Big) \, T^{1- \frac{\gamma}{2} - \frac{d}{2}(\frac{1}{2}-\frac{1}{p})} .
\end{align}

Plugging the estimates \eqref{eq:Prop32-bound1}, \eqref{eq:Prop32-bound2} and \eqref{eq:Prop32-bound3} into \eqref{eq:second}-\eqref{eq:third}, we obtain that
\begin{align*}
\Vert  u^N_t \Vert _{\gamma, p}  \lesssim \| u_{0}^{N}  \|_{\gamma,p} +
  \sup_{s\in [0,T]} \Big( \lVert m_{0}(F^N_{s}) \rVert_{L^2_x} + \lVert m_{1}(F^N_{s}) \rVert_{L^2_x} \Big) + \int_{0}^{t}  \frac{1}{(t-s)^{\frac{1+\gamma}{2}}} \left\Vert  u_{s}^{N}  \right\Vert _{\gamma,p} \dd s.
\end{align*}
By a Gr\"onwall-type Lemma for convolution (see e.g. \cite[Lemma 7.1.1]{Henry}), it comes
\begin{equation}\label{eq:bounduN-moments}
\sup_{[0,T]} \Vert u^N_t \Vert _{\gamma, p}  \lesssim \| u_{0}^{N} \|_{\gamma,p} +
 \sup_{s\in [0,T]} \Big( \lVert m_{0}(F^N_{s}) \rVert_{L^2_x} + \lVert m_{1}(F^N_{s}) \rVert_{L^2_x} \Big).
\end{equation}
In view of Lemma~\ref{lem:japanese}, it remains to observe that for $k>1+\frac{d}{2}$ (\emph{i.e.} $k\geq 3$ since $k$ is an integer),
\begin{equation*}
\lVert m_{0}(F^N_{s}) \rVert_{L^2_x} \lesssim  \lVert \langle v \rangle^k F^N_{s} \rVert_{L^2_{x,v}} ~~\text{ and }~~ \lVert m_{1}(F^N_{s}) \rVert_{L^2_x} \lesssim \lVert \langle v \rangle^k F^N_{s} \rVert_{L^2_{x,v}}.
\end{equation*}
Thus, raising the expression \eqref{eq:bounduN-moments} to the power $q$ and taking expectation, then controlling the norm of $u_{0}^N$ with Assumption~\ref{assump-particles}\ref{hyp:ICu0} and using Proposition~\ref{moment} concludes the proof.
\end{proof}

\section{Quantitative estimation of $F_{t}^{N}- F_{t}$}
\label{sec:FN-F}

In the whole section, let $(u,F)$ be a weak solution of \eqref{eq:PDE} and assume that Assumption~\ref{assump-PDE} holds. Then the Navier--Stokes solution $u$ solves the mild equation \eqref{eq:umild} and according to Proposition~\ref{prop:CoBessel-u}, $u\in \mathcal{C}([0,T];H^\gamma_{p}(\T^d)^d)$.

\subsection{The auxiliary PDE}

To prove Theorem~\ref{th:convBessel}, one should estimate $F_{t}^{N}- F_{t}$ in a weighted $L^2$ norm. We proceed by decomposing it into the sum of $F_{t}^{N}- \tilde F_{t}^N$ and $\tilde F_{t}^{N}- F_{t}$, for the auxiliary mapping $\tilde F_{t}^N$ that we describe now. Consider the intermediate PDE system $(u^{(N)},F^{(N)})$ given by
\begin{equation}\label{eq:intermediatePDE}
\begin{cases}
&\partial_t u^{(N)}_{t} - \Delta u^{(N)}_{t} + (u^{(N)}_{t} \cdot \nabla) [\bchi(u^{(N)}_{t})] + \nabla p^{(N)}_t +\displaystyle\int_{\R^d} (\bchi(u^{(N)}_{t})-v') F^{(N)}_{t}(x,v') \dd v'=0 ,\\
& \dive u^{(N)}_{t} =0 , \\
& u^{(N)}_{0} = u^\circ, \\
&\partial_{t} F^{(N)}_{t}(x,v) + v \cdot \nabla_{x} F^{(N)}_{t}(x,v) + \dive_v\big((\bchi(u^{(N)}_{t}(x))-v) F^{(N)}_{t}(x,v)\big) = \frac{\sigma_{N}^{2}}{2} \Delta_v F^{(N)}_{t}(x,v) ,\\
& F^{(N)}_{0} = F^\circ,
\end{cases}
\end{equation}
where we emphasize that the difference with the original PDE~\eqref{eq:PDE} lies here in the diffusion parameter which is $\sigma_{N}>0$ as well as the presence of the cut-off $\bchi$. Now let us define
\begin{equation}\label{eq:deftildeFN}
\tilde{F}^{N}_t= F^{(N)}_t\ast \vartheta^{N}.
\end{equation}
Compared to the original PDE~\eqref{eq:PDE}, under the Assumption~\ref{assump-PDE}, the existence of a bounded weak solution is more classically ensured here thanks to the presence of the cut-offs in~\eqref{eq:intermediatePDE}. We further need some boundedness properties on $(u^{(N)}, F^{(N)})$ to hold uniformly in $N$.
\begin{lemma}
\label{lem:boundsuNFN}
Let Assumption~\ref{assump-PDE} hold and $(\sigma_{N})_{N\in \N^*}$ satisfy Assumption~\ref{assump-particles}\ref{hyp:sigma}. Then the following holds:
\begin{align*}
\sup_{N\in \N^*}\sup_{t\in [0,T]} \big\lVert \langle v \rangle^k F^{(N)}_{t} \big\rVert_{L^2_{x,v}}<\infty , \qquad \sup_{N\in \N^*} \sup_{t\in [0,T]} \big\lVert u^{(N)}_{t} \big\rVert_{\gamma,p}<\infty \\
\text{ and } \quad \sup_{N\in \N^*}\Big( \sup_{t\in [0,T]} \lVert m_{0}(F^{(N)}_{t})\rVert_{L^2_{x}} + \sup_{t\in [0,T]} \lVert m_{1}(F^{(N)}_{t})\rVert_{L^2_{x}} \Big) <\infty .
\end{align*}
\end{lemma}

\begin{proof}
Applying Lemma~\ref{lem:momentsF2} with $\mathfrak{u} = \bchi(u) \in L^\infty([0,T];L^\infty(\T^d))$, one gets that $\sup_{N\in \N^*}\sup_{t\in [0,T]} \big\lVert \langle v \rangle^k F^{(N)}_{t} \big\rVert_{L^2_{x,v}}$ is finite. Then by this bound on $F^{(N)}$ and a slight adaptation of Proposition~\ref{prop:CoBessel-u}, one can prove thanks to the presence of the cut-off in~\eqref{eq:intermediatePDE} that $\sup_{N\in \N^*} \sup_{t\in [0,T]} \big\lVert u^{(N)}_{t} \big\rVert_{\gamma,p}$ is finite.

Finally, Lemma~\ref{lem:japanese} gives that
\begin{equation*}
\lVert m_{0}(F^{(N)}_{t})\rVert_{L^2_{x}} + \lVert m_{1}(F^{(N)}_{t})\rVert_{L^2_{x}}
\lesssim \big\lVert \langle v \rangle^k F^{(N)}_{t} \big\rVert_{L^2_{x,v}},
\end{equation*}
which is bounded uniformly in $t\in [0,T]$ and $N\in \N^*$.
\end{proof}

In the rest of this section, we focus first on $\E \big\| \langle v \rangle^{k}  (F_{t}^{N}- \tilde{F}^{N}_t) \big\|_{L^{2}_{x,v}}^{2q}$ in Sections~\ref{subsec:prelimLemmas} and~\ref{subsec:FN-tildeFN}. The second term involving $\tilde F_{t}^{N}- F_{t}$ will enter a more general bound on the quantity $\rate$ defined in~\eqref{eq:def-rate}, which will be discussed in Sections~\ref{subsec:discuss-rateregul} and~\ref{subsec:discuss-rateregul-0}.

\subsection{Preliminary lemmas}
\label{subsec:prelimLemmas}

\begin{lemma}\label{lem:rateHolderReg}
Let $f\in H^\gamma_{p}(\T^d)$ with $\gamma>\frac{d}{p}$ and $\mu$ a probability measure on $\T^d\times \R^d$. Then there exists $C = C(\gamma,d,p)>0$ such that for any $N\in \N$ and any $(x,v)\in \T^d\times \R^d$,
\begin{align*}
\big|\vartheta^N \ast \big[(\bchi(f(x)) - \bchi(f)) \, \mu\big]\big|(x,v) \leq \frac{C}{N^{\beta(\gamma-\frac{d}{p})}} \lVert f\rVert_{\gamma,p}\, \vartheta^N \ast \mu(x,v).
\end{align*}
\end{lemma}

\begin{proof}
Remark that, since $\|\chi_A'\|_\infty\leq 1$, and since $f\in H^\gamma_{p}(\T^d) \hookrightarrow \mathcal{C}^{\gamma-\frac{d}{p}}(\T^d)$  by Sobolev embedding, it comes
\begin{align*}
\big|\vartheta^N&\ast \big[(\bchi(f(x)) - \bchi(f)) \mu\big]\big|(x,v)\\
&\leq \int_{\T^{d}\times \R^{d}} |f(x) - f(x')|\, \vartheta^N(x-x',v-v')  \, \mu(\text{d} x' ,\text{d} v') \\
&\leq C \|f\|_{\gamma,p} \int_{\T^{d}\times \R^{d}} |x-x'|^{\gamma-\frac{d}{p}} \vartheta^N(x-x',v-v') \, \mu(\text{d} x', \text{d} v') .
\end{align*}
By periodisation, this last bound is also
\[
\|f\|_{\gamma,p} \int_{[-\frac{1}{2},\frac{1}{2})^d \times \R^{d}} \mathbbm{1}_{[-\frac{1}{2},\frac{1}{2})^d}(x-x')
 |x-x'|^{\gamma-\frac{d}{p}} \vartheta^{1,N}_0(x-x')\vartheta^{2,N}(v-v') \,  \mu(\text{d} x', \text{d} v'),
\]
where we recall that $\vartheta_{0}^{1,N}$ was defined in \eqref{eq:defTheta0}.
Remark that for all $(x-x') \in [-\frac{1}{2},\frac{1}{2})^d$, we have
\[
|x-x'|^{\gamma-\frac{d}{p}} \vartheta^{1,N}_0(x-x')  = N^{-\beta(\gamma-\frac{d}{p})} |N^{\beta}(x-x')|^{\gamma-\frac{d}{p}}\vartheta^{1,N}_0(x-x') \leq N^{-\beta(\gamma-\frac{d}{p})} \sqrt{d}^{\gamma-\frac{d}{p}}\vartheta^{1,N}_0(x-x'),
\]
using in the last inequality that the support of $\vartheta^{1,N}_0$ is in $\big(-\frac{N^{-\beta}}{2},\frac{N^{-\beta}}{2}\big)^d$.
\end{proof}

\begin{lemma}\label{lem:regf}
There exists $C>0$ such that for any measurable $f$ from $ \T^d\times \R^d \to \R$ and any $N\in \N^*$, it holds:
\begin{equation*}
\int_{\R^d} \int_{\T^{d}} \langle v \rangle^{2k} | \vartheta^N \ast f(x,v) |^2 \dd x \dd v \leq C\, \int_{\R^d} \int_{\T^{d}} \langle v \rangle^{2k} | f(x,v) |^2 \dd x \dd v .
\end{equation*}
\end{lemma}

\begin{proof}
We have, using the fact that $\vartheta^N$ is a density,
\begin{align*}
\iint \langle v \rangle^{2k}  |\vartheta^N \ast f(x,v) |^{2} \dd x \dd v
&= \iint \langle v \rangle^{2k} \left|\iint f(x',v') \, \vartheta^N(x-x', v-v') \dd x' \dd v' \right|^2  \dd x \dd v  \\
&\leq \iint \langle v \rangle^{2k} \iint |f(x',v')|^2\, \vartheta^N(x-x', v-v') \dd x' \dd v'  \dd x \dd v.
\end{align*}
Now, by a change of variables and since $\langle v+v' \rangle^{2k} \lesssim \langle v\rangle^{2k}\, \langle v' \rangle^{2k} $, one gets
\begin{align*}
\iint \langle v \rangle^{2k}  |\vartheta^N \ast f(x,v) |^{2} \dd x \dd v
&\lesssim \int_{\T^d} \Big(\int_{\R^d} \langle v'\rangle^{2k} |f(x',v')|^2 \dd v' \Big) \Big(\iint \langle v\rangle^{2k} \vartheta^N(x-x', v) \dd x  \dd v\Big) \dd x' .
\end{align*}
For the last integral we have, using \eqref{eq:defTheta0} and \eqref{eq:defThetaN},
\begin{align*}
\sup_{N} \iint \langle v\rangle^{2k} \vartheta^N(x-x', v) \dd x  \dd v = \sup_{N}\int_{\T^d} \vartheta^1(x) \dd x \int_{\R^d} \langle v\rangle^{2k} \vartheta^{2,N}(v)   \dd v <\infty,
\end{align*}
where we used that $\vartheta^{2}$ has compact support (see \eqref{hyp:theta2}) to bound the last integral uniformly with respect to $N$. Hence
\begin{align*}
\iint \langle v \rangle^{2k}  |\vartheta^N \ast f(x,v) |^{2} \dd x \dd v
\lesssim  \iint \langle v'\rangle^{2k} |f(x',v')|^2 \dd x' \dd v' .
\end{align*}
\end{proof}

\subsection{Main estimate on the moments of $F^N$-$\tilde{F}^N$}
\label{subsec:FN-tildeFN}

Recall that for each $N\in \N^*$, we defined in \eqref{eq:intermediatePDE} an auxiliary PDE solution $(u^{(N)},F^{(N)})$ and $\tilde{F}^N$ is given in \eqref{eq:deftildeFN}.
\begin{proposition}\label{prop:momentsFN-tildeFN}
Let Assumption~\ref{assump-PDE} and the same assumptions as in Proposition \ref{propu} hold. In particular, recall that $p> d$, $\gamma\in (\frac{d}{p}, 1)$ and that if $d=3$, $\frac{\gamma}{2} + \frac{3}{2}(\frac{1}{2}-\frac{1}{p}) <1$. Let $q \geq 1$ and $\varepsilon\geq 0$, with $\varepsilon=0$ allowed if $\lim \sigma_{N}=\sigma>0$. Then there exists $C=C(A,k,\gamma,p,d)>0$ and $\mathfrak{C}_{\varepsilon} >0$ such that for any $t\leq T$ and any $N\in \N^*$, we have
\begin{align}\label{eq:diffFN}
& \bigg(\E \sup_{s\in[0,t]} \big\|\langle v \rangle^{k} (F_{s}^{N} - \tilde{F}_{s}^{N}) \big\|_{L^{2}_{x,v} }^{2q}
+ \sigma_N^{2q}\, \E \bigg[ \Big(\int_{0}^t \big\lVert \langle v \rangle^k \nabla_{v} (F_{s}^{N} - \tilde{F}_{s}^{N})\big\rVert_{L^{2}_{x,v}}^{2} \dd s\Big)^q\bigg] \bigg)^{\frac{1}{q}} \nonumber\\
 & \leq C  \big(\E \big\| \langle v \rangle^{k} (F_{0}^{N}-\tilde{F}_{0}^N) \big\|_{L^{2}_{x,v}}^{2q}\big)^{\frac{1}{q}}
  + C \Big( N^{-2(\alpha-\beta)} + \mathfrak{C}_{\varepsilon} N^{-2\beta(\gamma-\frac{d}{p})+\varepsilon} + N^{-(\frac{1}{2} - d\beta -(d+1)\alpha)} \Big)\nonumber\\
 &\quad + C (1+\sigma_{N}^{-2}) \int_{0}^t \big(\E \sup_{r\in[0,s]} \big\| \langle v \rangle^{k} (F_{r}^{N} - \tilde{F}_{r}^{N}) \big\|_{L^{2}_{x,v} }^{2q}\big)^{\frac{1}{q}} \dd s \\
&\quad + C (1+\sigma_{N}^{-2}) \int_{0}^{t}  \big(\E \big\|u^{N}_s -u^{(N)}_s \big\|_{L^\infty_{x}}^{2q}\big)^{\frac{1}{q}} \dd s.\nonumber
\end{align}
Here and in the rest of the paper, if $\sigma>0$, then one can choose $\varepsilon=0$ and $\mathfrak{C}_{0}<\infty$.
\end{proposition}

\begin{proof}
In view of the boundedness of $(u^{(N)})_{N\in \N^*}$ in $\mathcal{C}([0,T]; H^\gamma_{p}(\mathbb{T}^d)^d)$, see Lemma~\ref{lem:boundsuNFN}, it follows from a Sobolev embedding that $u^{(N)}$ is bounded, so $u^{(N)} = \bchi(u^{(N)}) \equiv  \bchi_A(u^{(N)})$ for $A$ large enough and independent of $N$. By integration-by-parts, it comes that $\vartheta^N \ast (\dv \cdot \nabla_{x} F^{(N)}_{s}) = \dive_{x} \big(\vartheta^N\ast (\dv F^{(N)}_{s}) \big)$. Using the two previous facts, we first observe that
$\tilde{F}_{t}^{N}$ verifies, \forxvt,
\begin{equation}\label{eq:tildeFN}
\begin{split}
\tilde{F}_{t}^{N}(x,v) &= \tilde{F}_{0}^{N}(x,v) - \int_{0}^{t} \dive_{x} \big(\vartheta^N\ast (\dv F^{(N)}_{s}) \big)(x,v) \dd s\\
&\quad  - \int_{0}^{t} \dive_{v} \big(\vartheta^N\ast \big((\bchi(u^{(N)}_{s})-\dv )F^{(N)}_{s}\big) \big)(x,v) \dd s + \frac{\sigma_{N}^{2}}{2} \int_{0}^{t} \Delta_v \tilde{F}_{s}^{N}(x,v) \dd s .
\end{split}
\end{equation}

Now denote $f^{N}_t=F_{t}^{N}-\tilde{F}_{t}^{N}$. Using \eqref{equaregula} and \eqref{eq:tildeFN},
\begin{align*}
f_{t}^N(x,v) &= f_{0}^N(x,v) - \int_{0}^{t} \dive_{x} \big(\vartheta^N\ast (\dv (S^N_{s}-F^{(N)}_{s})) \big)(x,v) \dd s\\
&\quad  - \int_{0}^{t} \dive_{v} \big(\vartheta^N\ast \big[(\bchi(u^N_{s})-\dv )S^N_{s} - (\bchi(u^{(N)}_{s})-\dv )F^{(N)}_{s} \big] \big)(x,v) \dd s \\
&\quad+ M^N_{t}(x,v) + \frac{\sigma_{N}^{2}}{2} \int_{0}^{t} \Delta_v f_{s}^{N}(x,v) \dd s ,
\end{align*}
where we recall that the martingale $M^N$ was defined in \eqref{eq:defMN}.
Then, by Itô's formula, we have
\begin{align*}
|f^N_{t}|^2 &= |f^N_{0}|^2 -2 \int_{0}^t f^N_{s}\, \dive_{x} \big(\vartheta^N\ast (\dv (S^N_{s}-F^{(N)}_{s})) \big) \dd s \\
&\quad -2 \int_{0}^t f^N_{s}\, \dive_{v} \big(\vartheta^N\ast \big[(\bchi(u^N_{s})-\dv )S^N_{s} - (\bchi(u^{(N)}_{s})-\dv )F^{(N)}_{s} \big] \big) \dd s \\
&\quad - \frac{2\sigma_{N}}{N} \sum_{i=1}^{N} \int_{0}^t f_{s}^{N}\, \nabla_{v} \vartheta^N(x-X^{i,N}_s, v-V^{i,N}_s) \cdot \dd B_{s}^{i}
 + \frac{\sigma^{2}_N}{N} \int_{0}^t (S_{s}^{N} \ast |\nabla_v \vartheta^N|^{2}) \dd s \\
&\quad + \sigma_{N}^{2} \int_{0}^{t} f^N_{s}\, \Delta_v f_{s}^{N} \dd s.
\end{align*}
Let $\eta$ be a smooth nonnegative function supported on the ball of radius $2$ and such that $\eta\equiv 1$ on the ball of radius $1$. For any $R>0$ and $v\in \R^d$, let $\eta_{R}(v)=\eta(\frac{v}{R})$. We will need $\eta_{R}$ in the following to ensure integrability, but we will eventually get rid of it by letting $R\to +\infty$.
In the previous equality, multiplying by $\langle v \rangle^{2k} \eta_{R}$, integrating on the variables $x\in\T^d$ and $v\in\R^d$, and by integration-by-parts, it comes
\begin{align}\label{eq:decompMomentsfN}
\iint & \langle v \rangle^{2k}\eta_{R}\, |f_{t}^{N}|^{2} \dd x \dd v
- \iint \langle v \rangle^{2k}\eta_{R}\, |f_{0}^{N}|^{2} \dd x \dd v \nonumber\\
=& -2 \int_{0}^t \iint \langle v \rangle^{2k}\eta_{R}\,  f^N_{s}\, \dive_{x} \big(\vartheta^N\ast (\dv (S^N_{s}-F^{(N)}_{s})) \big) \dd x \dd v \dd s \nonumber\\
& - 2\int_{0}^t \iint \langle v \rangle^{2k}\eta_{R}\,  f^N_{s}\, \dive_{v} \big(\vartheta^N\ast \big[(\bchi(u^N_{s})-\dv )S^N_{s} - (\bchi(u^{(N)}_{s})-\dv )F^{(N)}_{s} \big] \big) \dd x \dd v \dd s \nonumber\\
& - \frac{2 \sigma_{N}}{N} \sum_{i=1}^{N} \iint  \int_{0}^t \langle v \rangle^{2k} \eta_{R}\, f_{s}^{N}\, \nabla_{v}\vartheta^N(x-X^{i,N}_s, v-V^{i,N}_s) \cdot \dd B_{s}^{i} \dd x \dd v\\
& + \frac{\sigma_N^{2}}{N} \int_{0}^t \iint\langle v \rangle^{2k} \eta_{R}\, (S_{s}^{N} \ast |\nabla_v \vartheta^N|^{2}) \dd x \dd v \dd s
  - \sigma_N^{2} \int_{0}^t \iint  \langle v \rangle^{2k} \eta_{R}\, |\nabla_v f_{s}^{N}|^{2} \dd x \dd v \dd s \nonumber\\
&  -\sigma_N^2 \int_{0}^t \iint 2k \langle v \rangle^{2k-2} \eta_{R}\,  v \cdot \big( f_{s}^{N} \nabla_v f_{s}^{N} \big) \dd x \dd v \dd s \nonumber\\
& - \sigma_N^{2} \int_{0}^t \iint  \langle v \rangle^{2k} \nabla_{v}\eta_{R}\cdot \big( f_{s}^{N} \nabla_v f_{s}^{N} \big) \dd x \dd v \dd s \eqqcolon \sum_{k=1}^7 J_{k}(t) . \nonumber
\end{align}

\paragraph{Bound on $J_{1}$.}
As in \eqref{eq:1111}, we have
\begin{align*}
\iint \langle v \rangle^{2k} \eta_{R}\, f_{s}^{N} \dive_{x} \big(\vartheta^N\ast (\dv (S^N_{s}-F^{(N)}_{s})) \big) \dd x \dd v
= - \iint \langle v \rangle^{2k} \eta_{R}\, f_{s}^{N}\, ( \nabla_{x}\vartheta^N \cdot \dv) \ast (S_{s}^{N}- F^{(N)}_{s}) \dd x \dd v.
\end{align*}
Hence, by the Young and triangle inequalities, it comes
\begin{align}\label{eq:J1-1}
\Big| \iint&  \langle v \rangle^{2k} \eta_{R}\, f_{s}^{N} \dive_{x} \big(\vartheta^N\ast (\dv (S^N_{s}-F^{(N)}_{s})) \big) \dd x \dd v \Big| \nonumber\\
&\leq \frac{1}{2} \iint \langle v \rangle^{2k} \eta_{R}\, |f_{s}^{N}|^2 \dd x \dd v
+ \frac{1}{2} \iint \langle v \rangle^{2k} \eta_{R}\, \big| ( \nabla_{x}\vartheta^N \cdot \dv) \ast (S_{s}^{N}- F^{(N)}_{s})\big|^2 \dd x \dd v \nonumber\\
&\lesssim \iint \langle v \rangle^{2k} \eta_{R}\, |f_{s}^{N}|^2 \dd x \dd v
+ \iint \langle v \rangle^{2k} \eta_{R}\, \big| ( \nabla_{x}\vartheta^N \cdot \dv) \ast S_{s}^{N}\big|^2 \dd x \dd v \nonumber\\
&\quad + \iint \langle v \rangle^{2k} \eta_{R}\, \big| ( \nabla_{x}\vartheta^N \cdot \dv) \ast F^{(N)}_{s}\big|^2 \dd x \dd v.
\end{align}
Using \eqref{eq:boundVGradTheta}, \emph{i.e.} that $|\nabla_{x} \vartheta^N(x,v) \cdot v | \lesssim N^{\beta-\alpha}\, \vartheta^N(x,v)$, we obtain
\begin{equation*}
 \iint \langle v \rangle^{2k} \eta_{R}\, \big| ( \nabla_{x}\vartheta^N \cdot \dv) \ast S_{s}^{N}\big|^2 \dd x \dd v
 \lesssim \frac{1}{N^{2\alpha-2\beta}} \iint \langle v \rangle^{2k} \eta_{R}\, |F_{s}^{N}|^{2} \dd x \dd v ,
\end{equation*}
and similarly (using in addition that $F^{(N)}_{s}\geq 0$)
\begin{equation*}
\iint \langle v \rangle^{2k} \eta_{R}\, \big| ( \nabla_{x}\vartheta^N \cdot \dv) \ast F^{(N)}_{s}\big|^2 \dd x \dd v
 \lesssim \frac{1}{N^{2\alpha-2\beta}} \iint \langle v \rangle^{2k} \eta_{R}\, |\tilde{F}_{s}^{N}|^{2} \dd x \dd v .
\end{equation*}
Hence, plugging the last two inequalities in \eqref{eq:J1-1}, it comes
\begin{equation}\label{eq:boundJ1}
|J_{1}(t)| \lesssim \int_{0}^t \iint \langle v \rangle^{2k} \eta_{R}\, |f_{s}^{N}|^2 \dd x \dd v \dd s
+\frac{1}{N^{2\alpha-2\beta}} \int_{0}^t \iint \langle v \rangle^{2k} \eta_{R}\, \big(|F_{s}^{N}|^{2} +|\tilde{F}_{s}^{N}|^{2} \big) \dd x \dd v .
\end{equation}

\paragraph{Bound on $J_{2}$.}
By the triangle inequality, we have
\begin{align*}
|J_2(t)| &\leq2 \left|\int_{0}^t\iint \langle v \rangle^{2k} \eta_{R}\, f^N_{s} \dive_{v} \big(\vartheta^N\ast \big[\bchi(u^N_{s})S^N_{s} - \bchi(u^N_{s})F^{(N)}_{s} \big] \big) \dd x \dd v \dd s\right|\\
&\quad  +2\left|\int_{0}^t\iint \langle v \rangle^{2k} \eta_{R}\, f^N_{s} \dive_{v} \big(\vartheta^N\ast \big[\bchi(u^N_{s})F^{(N)}_{s} - \bchi(u^{(N)}_{s})F^{(N)}_{s} \big] \big) \dd x \dd v \dd s\right|\\
&\quad+2\left|\int_{0}^t\iint \langle v \rangle^{2k} \eta_{R}\, f^N_{s} \dive_{v} \big(\vartheta^N\ast \big[\dv (S^N_{s} - F^{(N)}_{s} )\big] \big) \dd x \dd v \dd s\right|\\
&\eqqcolon J_{2,1}+J_{2,2}+J_{2,3} .
\end{align*}

For $J_{2,1}$,
\begin{align*}
J_{2,1}&\leq
2\left|\int_{0}^t\iint \langle v \rangle^{2k} \eta_{R}\, f^N_{s} \dive_{v} \big(\vartheta^N\ast \big[(\bchi(u^N_{s}) - \bchi(u^N_{s}(x)))(S^N_{s} - F^{(N)}_{s})\big] \big) \dd x \dd v \dd s\right| \\
&\quad+2\left|\int_{0}^t\iint \langle v \rangle^{2k} \eta_{R}\, f^N_{s} \bchi(u^N_{s}(x)) \cdot\nabla_{v}\big(\vartheta^N\ast \big[S^N_{s} - F^{(N)}_{s} \big] \big) \dd x \dd v \dd s\right| \\
& \eqqcolon J_{2,1,1} + J_{2,1,2}.
\end{align*}
For the first term, using integration-by-parts gives
\begin{align*}%
J_{2,1,1}
&\leq 4k\left|\int_{0}^t\iint \langle v \rangle^{2k-2} v\cdot \eta_{R}\, f^N_{s} \big(\vartheta^N\ast \big[(\bchi(u^N_{s}) - \bchi(u^N_{s}(x)))(S^N_{s} - F^{(N)}_{s})\big] \big) \dd x \dd v \dd s\right| \\
&\quad+2\left|\int_{0}^t\iint \langle v \rangle^{2k} \eta_{R}\, \nabla_v f^N_{s}\cdot \big(\vartheta^N\ast \big[(\bchi(u^N_{s}) - \bchi(u^N_{s}(x)))(S^N_{s} - F^{(N)}_{s})\big] \big) \dd x \dd v \dd s\right|\\
&\quad+2\left|\int_{0}^t\iint \langle v \rangle^{2k} \nabla_v\eta_{R}\cdot f^N_{s} \big(\vartheta^N\ast \big[(\bchi(u^N_{s}) - \bchi(u^N_{s}(x)))(S^N_{s} - F^{(N)}_{s})\big] \big) \dd x \dd v \dd s\right|.
\end{align*}
In view of Lemma~\ref{lem:rateHolderReg} and using the inequality $\langle v \rangle^{2k-2} |v| \leq \langle v \rangle^{2k}$, it comes
\begin{align*}
J_{2,1,1}
&\leq C \|u^{N}_t\|_{\gamma,p}  \int_{0}^t\iint \langle v \rangle^{2k} \big(\eta_{R}+|\nabla_{v}\eta_{R}|\big) |f^N_{s}|
N^{-\beta(\gamma-\frac{d}{p})}  (F^N_s + \tilde{F}^N_s) \dd x \dd v \dd s \\
&\quad+ C \|u^{N}_t\|_{\gamma,p}  \int_{0}^t\iint \langle v \rangle^{2k} \eta_{R}\, |\nabla_v f^N_{s}|
N^{-\beta(\gamma-\frac{d}{p})}  (F^N_s + \tilde{F}^N_s) \dd x \dd v \dd s.
\end{align*}
We now observe that $\eta_{R}+|\nabla_{v}\eta_{R}|$ is bounded by a function with support in the ball of radius $2R$, and with a slight abuse of notation, we still denote this upper bound by $\eta_{R}$. Hence
using Young's inequality, we obtain, for some $\delta>0$ that will be fixed later,
\begin{align*}
J_{2,1,1}
&\leq C\int_{0}^t\iint \langle v \rangle^{2k} \eta_{R}\, |f^N_{s}|^2 \dd x \dd v \dd s
+\delta \int_{0}^t\iint \langle v \rangle^{2k} \eta_{R}\, |\nabla_v f^N_{s}|^2  \dd x \dd v \dd s\\
&\quad+C (1+\frac{1}{\delta}) \|u^{N}_t\|_{\gamma,p}^2 N^{-2\beta(\gamma-\frac{d}{p})}  \int_{0}^t\iint \langle v \rangle^{2k} \eta_{R}\, (|F^N_s|^2+|\tilde{F}^N_s|^2) \dd x \dd v \dd s.
\end{align*}
Now for the term $J_{2,1,2}$, using that $\bchi\equiv \bchi_A$ is bounded by $(1+A)$ and by Young's inequality, we obtain for the same $\delta$ as before
\begin{align*}
J_{2,1,2}
&\leq 2(1+A) \int_{0}^t\iint \langle v \rangle^{2k} \eta_{R}\, | f^N_{s}| |\nabla_{v} f^N_s| \dd x \dd v \dd s\\
&\leq \frac{(1+A)^2}{\delta} \int_{0}^t\iint \langle v \rangle^{2k} \eta_{R}\, | f^N_{s}|^2 \dd x \dd v \dd s+\delta \int_{0}^t\iint \langle v \rangle^{2k} \eta_{R}\, |\nabla_{v} f^N_s|^2 \dd x \dd v \dd s.
\end{align*}
Hence, gathering the bounds on $J_{2,1,1}$ and $J_{2,1,2}$, there exists $C=C(A,\gamma,p,d)$ such that
\begin{align}\label{eq:boundJ21}
J_{2,1} &\leq C\big(1+\frac{1}{\delta}\big) \int_{0}^t\iint \langle v \rangle^{2k} \eta_{R}\, |f^N_{s}|^2 \dd x \dd v \dd s
+2\delta \int_{0}^t\iint \langle v \rangle^{2k} \eta_{R}\, |\nabla_v f^N_{s}|^2  \dd x \dd v \dd s\nonumber\\
&\quad+\frac{C  \|u^{N}_t\|_{\gamma,p}^2}{N^{2\beta(\gamma-\frac{d}{p})}} \big(1+\frac{1}{\delta}\big) \int_{0}^t\iint \langle v \rangle^{2k} \eta_{R}\, (|F^N_s|^2+|\tilde{F}^N_s|^2) \dd x \dd v \dd s.
\end{align}

\medskip

Now consider $J_{2,2}$. Inserting the term $-\bchi(u^{(N)}_{s}(x))+\bchi(u^{N}_{s}(x))$ as a constant in the convolution with $\vartheta^N$, we get
\begin{align*}
J_{2,2} &= 2 \left|\int_{0}^t \iint \langle v \rangle^{2k} \eta_{R}\, f^N_{s}\, \dive_{v} \big(\vartheta^N\ast \big[\bchi(u^{(N)}_{s})F^{(N)}_{s} - \bchi(u^N_{s})F^{(N)}_{s} \big] \big) \dd x \dd v \dd s\right|\\
&  \leq 2 \Big|\int_{0}^t \iint \langle v \rangle^{2k} \eta_{R}\, f_{s}^{N} (\bchi(u^{(N)}_{s})- \bchi(u^{N}_{s})) \cdot \nabla_{v} \tilde{F}_{s}^{N}  \dd x \dd v \dd s \Big| \\
&\quad  + 2 \Big|\int_{0}^t \iint \langle v \rangle^{2k} \eta_{R}\, f_{s}^{N} \dive_{v}\big( \vartheta^N\ast \big[(\bchi(u^{(N)}_{s})-\bchi(u^{(N)}_{s}(x))- \bchi(u^{N}_{s})+\bchi(u^{N}_{s}(x))) F^{(N)}_{s}\big] \big) \dd x \dd v \dd s \Big|.
\end{align*}
By integration-by-parts, using again that $\eta_{R}+|\nabla_{v}\eta_{R}|$ can be bounded by $C \eta_{R}$ up to a slight abuse of notation, then using that $\bchi$ is Lipschitz continuous
 and by Young's inequality, we arrive at
\begin{align*}
J_{2,2}&\leq 2 \int_{0}^t \iint \langle v \rangle^{2k} \eta_{R}\, \big((2k+C) f_{s}^{N} + |\nabla_{v} f_{s}^{N}| \big) \, |\bchi(u^{(N)}_{s})- \bchi(u^{N}_{s})|\, \tilde{F}_{s}^{N}  \dd x \dd v \dd s \\
&\quad + 2 \int_{0}^t \iint \langle v \rangle^{2k} \eta_{R}\, \big((2k+C) f_{s}^{N} + |\nabla_{v} f_{s}^{N}| \big)\\
&\hspace{2cm} \times \big( \vartheta^N\ast \big[(\bchi(u^{(N)}_{s})-\bchi(u^{(N)}_{s}(x))- \bchi(u^{N}_{s})+\bchi(u^{N}_{s}(x))) F^{(N)}_{s}\big] \big) \dd x \dd v \dd s\\
&\leq (4k+2C) \int_{0}^t \iint \langle v \rangle^{2k} \eta_{R}\, |f_{s}^{N}|^2  \dd x \dd v \dd s + (2k+C+\frac{1}{\delta}) \int_{0}^t \iint \langle v \rangle^{2k} \eta_{R}\, |u^{(N)}_{s} - u^{N}_{s}|^2 \, |\tilde{F}_{s}^{N}|^2  \dd x \dd v \dd s\\
&\quad + 2\delta \int_{0}^t \iint \langle v \rangle^{2k} \eta_{R}\, |\nabla_{v} f_{s}^{N}|^2  \dd x \dd v \dd s\\
&\quad + (2k+C+ \frac{1}{\delta}) \int_{0}^t \iint \langle v \rangle^{2k} \eta_{R}\, \Big| \vartheta^N\ast \Big[\big(\bchi(u^{(N)}_{s})-\bchi(u^{(N)}_{s}(x))- \bchi(u^{N}_{s})+\bchi(u^{N}_{s}(x))\big) F^{(N)}_{s}\Big] \Big|^2 \dd x \dd v \dd s.
\end{align*}
In view of Lemma~\ref{lem:rateHolderReg}, it comes that, for $C\equiv C(k,\gamma,d,p)$,
\begin{equation}\label{eq:boundJ22}
\begin{split}
J_{2,2} &\leq C \int_{0}^t \iint \langle v \rangle^{2k} \eta_{R}\, |f_{s}^{N}|^2  \dd x \dd v \dd s + C\big(1+\frac{1}{\delta}\big) \int_{0}^t \iint \langle v \rangle^{2k} \eta_{R}\, |u^{(N)}_{s} - u^{N}_{s}|^2 \, |\tilde{F}_{s}^{N}|^2  \dd x \dd v \dd s\\
&\quad + 2\delta \int_{0}^t \iint \langle v \rangle^{2k} \eta_{R}\, |\nabla_{v} f_{s}^{N}|^2  \dd x \dd v \dd s\\
&\quad + \frac{C(1+\frac{1}{\delta})}{N^{2\beta(\gamma-\frac{d}{p})}}  \Big( \sup_{t\in [0,T]} \| u^{(N)}_t \|_{\gamma, p}^{2} +\sup_{t\in [0,T]} \| u^{N}_t \|_{\gamma, p}^{2}\Big)  \int_{0}^t \iint \langle v \rangle^{2k} \eta_{R}\, | \tilde{F}_{s}^{N}|^{2} \dd x \dd v \dd s.
\end{split}
\end{equation}

\medskip

Finally for $J_{2,3}$, use \eqref{eq:mollvmu} to obtain
\begin{align*}
J_{2,3} &  \leq 2\Big|\int_{0}^t \iint \langle v \rangle^{2k} \eta_{R}\, f_{s}^{N}  \dive_{v} (\dv  f_{s}^{N} )  \dd x \dd v \dd s \Big|
+ 2\Big|\int_{0}^t \iint \langle v \rangle^{2k} \eta_{R}\, f_{s}^{N} \dive_{v}\big( (\dv \vartheta^N) \ast ( S_{s}^{N} -F^{(N)}_{s})\big) \dd x \dd v \dd s \Big|\\
& \eqqcolon J_{2,3,1} + J_{2,3,2}.
\end{align*}
Proceeding exactly as in \eqref{eq:boundI221},
\begin{align*}
J_{2,3,1} \lesssim \int_{0}^t \iint \langle v \rangle^{2k} \eta_{R}\, |f_{s}^{N}|^{2} \dd x \dd v \dd s.
\end{align*}
As for $J_{2,3,2}$, integrate-by-parts and use Young's inequality to get
\begin{align*}
J_{2,3,2} & \leq (2k+C) \int_{0}^t \iint \langle v \rangle^{2k} \eta_{R}\, |f_{s}^{N}|^2 \dd x \dd v \dd s
+ 2\delta \int_{0}^t \iint \langle v \rangle^{2k} \eta_{R}\, |\nabla_{v} f_{s}^{N}|^2 \dd x \dd v \dd s \\
&\quad  + (2k+C+\frac{1}{2\delta}) \int_{0}^t \iint \langle v \rangle^{2k} \eta_{R}\, |(\dv \vartheta^N) \ast ( S_{s}^{N} -F^{(N)}_{s})|^2 \dd x \dd v \dd s .
\end{align*}
Now from \eqref{eq:rateTheta2}, it comes
\begin{align*}
J_{2,3,2} & \leq (2k+C) \int_{0}^t \iint \langle v \rangle^{2k} \eta_{R}\, |f_{s}^{N}|^2 \dd x \dd v \dd s
+ 2\delta \int_{0}^t \iint \langle v \rangle^{2k} \eta_{R}\, |\nabla_{v} f_{s}^{N}|^2 \dd x \dd v \dd s \\
&\quad + 2N^{-2\alpha} \big(2k+C+\frac{1}{2\delta}\big) \int_{0}^t \iint \langle v \rangle^{2k} \eta_{R}\, \big(|F^N_{s}|^2 + |\tilde{F}^N_{s}|^2 \big) \dd x \dd v \dd s .
\end{align*}
Hence, summing the bounds on $J_{2,3,1}$ and $J_{2,3,2}$, we have for some $C\equiv C(k)$,
\begin{equation}\label{eq:boundJ23}
\begin{split}
J_{2,3} &\leq
C \int_{0}^t \iint \langle v \rangle^{2k} \eta_{R}\, |f_{s}^{N}|^2 \dd x \dd v \dd s
+ 2\delta \int_{0}^t \iint \langle v \rangle^{2k} \eta_{R}\, |\nabla_{v} f_{s}^{N}|^2 \dd x \dd v \dd s \\
&\quad  + C(1+\frac{1}{\delta}) N^{-2\alpha} \int_{0}^t \iint \langle v \rangle^{2k} \eta_{R}\, \big(|F^N_{s}|^2 + |\tilde{F}^N_{s}|^2 \big) \dd x \dd v \dd s .
\end{split}
\end{equation}

In view of the inequalities \eqref{eq:boundJ21}, \eqref{eq:boundJ22} and \eqref{eq:boundJ23}, we finally get the following bound for $J_{2}$:
\begin{align}\label{eq:boundJ2}
|J_{2}(t)| &\leq C\big(1+\frac{1}{\delta}\big) \int_{0}^t \iint \langle v \rangle^{2k} \eta_{R}\, |f^N_{s}|^2 \dd x \dd v \dd s
+6\delta \int_{0}^t \iint \langle v \rangle^{2k} \eta_{R}\, |\nabla_v f^N_{s}|^2  \dd x \dd v \dd s \nonumber\\
&\quad + C \big(1+\frac{1}{\delta}\big) \bigg(\frac{\sup_{s\in [0,T]} (\|u^{(N)}_{s}\|_{\gamma,p}^2 + \|u^{N}_{s}\|_{\gamma,p}^2)}{N^{2\beta(\gamma-\frac{d}{p})}} + N^{-2\alpha}\bigg) \int_{0}^t \iint \langle v \rangle^{2k} \eta_{R}\, (|F^N_s|^2+|\tilde{F}^N_s|^2) \dd x \dd v \dd s \nonumber\\
&\quad + C\big(1+\frac{1}{\delta}\big) \int_{0}^t \iint \langle v \rangle^{2k} \eta_{R}\, |u^{(N)}_{s} - u^{N}_{s}|^2 \, |\tilde{F}_{s}^{N}|^2  \dd x \dd v \dd s.
\end{align}

\paragraph{Bound on $J_{3}$.}
Similarly to \eqref{eq:boundI3}, we have
\begin{equation}\label{eq:boundJ3}
\E \sup_{s\in [0,t]} |J_{3}(s)|^q \lesssim \E \bigg[ \Big( \int_{0}^{t} \iint \langle v \rangle^{2k} \eta_{R}\, |f_{r}^{N}|^{2}  \dd x \dd v \dd r\Big)^q \bigg]
+\frac{\sigma_{N}^q}{N^{q(\frac{1}{2} - d\beta-(d+1)\alpha)}} .
\end{equation}

\paragraph{Bound on $J_{4}$.}
This term is the same as $I_{4}$ in \eqref{eq:decompMomentsF}, up to the factor $\eta_{R}$. Hence, similarly to \eqref{eq:boundI4}, we obtain
\begin{equation}\label{eq:boundJ4}
\E \sup_{s\in [0,t]}|J_{4}(s)|^q \lesssim \frac{\sigma_{N}^{2q}}{N^{q(1-d\beta- (d+2)\alpha)}}.
\end{equation}

\paragraph{Bound on $J_{6}$.}
This term is estimated like $I_{6}$ in \eqref{eq:decompMomentsF}. Thus, similarly to \eqref{eq:boundI6} and using further here that $\eta_{R}+|\nabla_{v}\eta_{R}|$ can be bounded by $C \eta_{R}$ up to a slight abuse of notation, we get
\begin{equation}\label{eq:boundJ6}
|J_{6}(t)| \lesssim \sigma_{N}^2 \int_{0}^t \iint \langle v \rangle^{2k} \eta_{R}\, |f_{t}^{N}|^{2} \dd x \dd v \dd s .
\end{equation}

\paragraph{Bound on $J_{7}$.}
This term is very close to $J_{6}$ and we proceed similarly. Namely, afeter integration-by-parts, we find that
\begin{align*}
|J_{7}(t)|
&= \frac{1}{2} \sigma_N^{2} \Big| \int_{0}^t \iint  \langle v \rangle^{2k} \nabla_{v}\eta_{R}\cdot \nabla_v |f_{s}^{N}|^2 \dd x \dd v \dd s \Big| \\
&\lesssim \sigma_N^{2} \int_{0}^t \iint  \langle v \rangle^{2k} \big(|\nabla_{v}\eta_{R}| + |\Delta\eta_{R}| \big) |f_{s}^{N}|^2 \dd x \dd v \dd s .
\end{align*}
As before, $|\nabla_{v}\eta_{R}| + |\Delta \eta_{R}|$ can be bounded by $C \eta_{R}$ by a slight abuse of notation (in fact it is even decreasing as $1/R$), and we retrieve the same bound as for $J_{6}$:
\begin{equation}\label{eq:boundJ7}
|J_{7}(t)| \lesssim \sigma_{N}^2 \int_{0}^t \iint \langle v \rangle^{2k} \eta_{R}\, |f_{t}^{N}|^{2} \dd x \dd v \dd s .
\end{equation}

\paragraph{Conclusion.}
In view of the decomposition of $\iint \langle v \rangle^{2k} |f_{t}^{N}|^{2} \dd x \dd v - \iint \langle v \rangle^{2k} |f_{0}^{N}|^{2} \dd x \dd v$ as $\sum_{k=1}^7 J_{k}(t)$ from \eqref{eq:decompMomentsfN}, move the negative term $J_{5}$ to the left-hand side of the inequality, then use the bound~\eqref{eq:boundJ2} on $J_{2}$ with the choice $\delta = \frac{\sigma_{N}^2}{12}$, which yields
\begin{align*}
\sup_{s\in [0,t]}& \big\lVert \langle v \rangle^{k} \sqrt{\eta_{R}}\, f_{s}^{N} \big\rVert_{L^2_{x,v}}^2 + \frac{\sigma_{N}^2}{2} \int_{0}^t \big\lVert \langle v \rangle^{k} \sqrt{\eta_{R}}\,  \nabla_{v} f_{s}^{N} \big\rVert_{L^2_{x,v}}^2 \dd s \\
&\leq \big\lVert \langle v \rangle^{k} \sqrt{\eta_{R}}\, f_{0}^{N} \big\rVert_{L^2_{x,v}}^2 + \sup_{s\in [0,t]} \big(|J_{1}(s)| + |J_{3}(s)| + |J_{4}(s)| + |J_{6}(s)|+ |J_{7}(s)| \big)\\
&\quad + C(1+\sigma_{N}^{-2}) \int_{0}^t \iint \langle v \rangle^{2k} {\eta_{R}}\, |f^N_{s}|^2 \dd x \dd v \dd s \\
&\quad + C (1+\sigma_{N}^{-2}) \bigg( \frac{\sup_{s\in [0,T]} (\|u^{(N)}_{s}\|_{\gamma,p}^2 + \|u^{N}_{s}\|_{\gamma,p}^2)}{N^{2\beta(\gamma-\frac{d}{p})}} + N^{-2\alpha}\bigg) \int_{0}^t \iint \langle v \rangle^{2k} {\eta_{R}}\, (|F^N_s|^2+|\tilde{F}^N_s|^2) \dd x \dd v \dd s\\
&\quad+ C(1+\sigma_{N}^{-2}) \int_{0}^t \iint \langle v \rangle^{2k} {\eta_{R}}\, |u^{(N)}_{s} - u^{N}_{s}|^2 \, |\tilde{F}_{s}^{N}|^2  \dd x \dd v \dd s.
\end{align*}
Now in the previous inequality we let $R\to+\infty$ and use the monotone convergence theorem, apply the $L^q(\Omega)$ norm, then use the bounds on $J_{1}, J_{3}, J_{4}, J_{6}$ and $J_{7}$ obtained respectively in \eqref{eq:boundJ1}, \eqref{eq:boundJ3}, \eqref{eq:boundJ4}, \eqref{eq:boundJ6} and \eqref{eq:boundJ7}, to deduce that
\begin{align*}
\Bigg( &\E \sup_{s\in [0,t]} \big\lVert \langle v \rangle^{k} f_{s}^{N} \big\rVert_{L^2_{x,v}}^{2q} + \frac{\sigma_{N}^{2q}}{2^q} \E \bigg[\bigg( \int_{0}^t \big\lVert \langle v \rangle^{k}\,  \nabla_{v} f_{s}^{N} \big\rVert_{L^2_{x,v}}^{2} \dd s\bigg)^q\bigg] \Bigg)^\frac{1}{q} \\
&\lesssim \Big( \E \big\lVert \langle v \rangle^{k} f_{0}^{N} \big\rVert_{L^2_{x,v}}^{2q}\Big)^\frac{1}{q}
+ (1+\sigma_{N}^{-2}) \bigg(\E\bigg[\Big(\int_{0}^t \iint \langle v \rangle^{2k} |f^N_{s}|^2 \dd x \dd v \dd s\Big)^q\bigg] \bigg)^\frac{1}{q}\\
&\quad +  \bigg(\E\bigg[\bigg( \Big(\frac{1}{N^{2(\alpha-\beta)}} +(1+\sigma_{N}^{-2}) \Big(\frac{\sup_{s\in [0,T]} (\|u^{(N)}_{s}\|_{\gamma,p}^2 + \|u^{N}_{s}\|_{\gamma,p}^2)}{N^{2\beta(\gamma-\frac{d}{p})}} + \frac{1}{N^{2\alpha}}\Big)\Big) \\
&\hspace{6cm} \times \int_{0}^t \iint \langle v \rangle^{2k} (|F^N_s|^2+|\tilde{F}^N_s|^2) \dd x \dd v \dd s\bigg)^q\bigg] \bigg)^\frac{1}{q} \\
&\quad + (1+\sigma_{N}^{-2})  \bigg(\E\bigg[\bigg(\int_{0}^t \iint \langle v \rangle^{2k}  |u^{(N)}_{s} - u^{N}_{s}|^2 \, |\tilde{F}_{s}^{N}|^2  \dd x \dd v \dd s\bigg)^q\bigg] \bigg)^\frac{1}{q} +\frac{\sigma_{N}}{N^{\frac{1}{2} - d\beta -(d+1)\alpha}} +  \frac{\sigma_{N}^{2}}{N^{1-d\beta- (d+2)\alpha}} .
\end{align*}
Now observe that under Assumption~\ref{assump-particles}\ref{hyp:sigma}, $\sigma_{N}$ is bounded. Besides, Assumption~\ref{assump-particles}\ref{hyp:alphabeta} implies that $0<\frac{1}{2} -d\beta-(d+1)\alpha < 1- d\beta- (d+2)\alpha$. Hence one has $\frac{\sigma_{N}}{N^{1/2 - d\beta -(d+1)\alpha}} +  \frac{\sigma_{N}^{2}}{N^{1-d\beta- (d+2)\alpha}} \lesssim \frac{\sigma_{N}}{N^{1/2 - d\beta -(d+1)\alpha}}$. Besides, applying the Cauchy-Schwarz inequality on the third term in the RHS of the previous inequality; upon using that $u^{(N)}$ is bounded in $\mathcal{C}([0,T]; H^\gamma_{p}(\mathbb{T}^d)^d)$ uniformly in $N$, as recalled at the beginning of this proof; and using the \emph{a priori} bound on $u^N$ from Proposition~\ref{propu}, we get
\begin{align*}
\Bigg( &\E \sup_{s\in [0,t]} \big\lVert \langle v \rangle^{k} f_{s}^{N} \big\rVert_{L^2_{x,v}}^{2q} + \frac{\sigma_{N}^{2q}}{2^q} \E \bigg[\bigg( \int_{0}^t \big\lVert  \langle v \rangle^{k}\, \nabla_{v} f_{s}^{N} \big\rVert_{L^2_{x,v}}^{2} \dd s\bigg)^q\bigg] \Bigg)^\frac{1}{q} \\
&\lesssim \Big( \E \big\lVert \langle v \rangle^{k} f_{0}^{N} \big\rVert_{L^2_{x,v}}^{2q}\Big)^\frac{1}{q}
+ (1+\sigma_{N}^{-2}) \bigg(\E\bigg[\Big(\int_{0}^t \iint \langle v \rangle^{2k} |f^N_{s}|^2 \dd x \dd v \dd s\Big)^q\bigg] \bigg)^\frac{1}{q}\\
&\quad +  \Big(\frac{1}{N^{2(\alpha-\beta)}} + \frac{(1+\sigma_{N}^{-2})(1+ e^{C T \sigma_{N}^{-2}})}{N^{2\beta(\gamma-\frac{d}{p})}}\Big)
\bigg( \E \bigg[ \bigg(  \int_{0}^t \iint \langle v \rangle^{2k} (|F^N_s|^2+|\tilde{F}^N_s|^2) \dd x \dd v \dd s\bigg)^{2q}\bigg] \bigg)^\frac{1}{2q} \\
&\quad + (1+\sigma_{N}^{-2})  \bigg(\E\bigg[\bigg(\int_{0}^t \iint \langle v \rangle^{2k}  |u^{(N)}_{s} - u^{N}_{s}|^2 \, |\tilde{F}_{s}^{N}|^2  \dd x \dd v \dd s\bigg)^q\bigg] \bigg)^\frac{1}{q}+\frac{\sigma_{N}}{N^{\frac{1}{2} - d\beta -(d+1)\alpha}} .
\end{align*}
Considering the third summand in the RHS of the previous inequality,  Proposition~\ref{moment} yields that the moments of $|F^N_s|^2$ are bounded by $C e^{C T \sigma_{N}^{-2}}$; while the moments of $|\tilde{F}^N_s|^2$ are bounded by a constant, using Lemma~\ref{lem:regf} and Lemma~\ref{lem:momentsF2}.
Moreover, for any $\varepsilon>0$, Assumption~\ref{assump-particles}\ref{hyp:sigma} implies that there exists $\mathfrak{C}_{\varepsilon}>0$ such that for any $N\in \N^*$,
\begin{equation*}
N^{-\varepsilon} (1+\sigma_{N}^{-2})(1+ e^{C T \sigma_{N}^{-2}}) e^{C T \sigma_{N}^{-2}} \leq \mathfrak{C}_{\varepsilon} .
\end{equation*}
Thus
\begin{align*}
\Bigg( &\E \sup_{s\in [0,t]} \big\lVert \langle v \rangle^{k} f_{s}^{N} \big\rVert_{L^2_{x,v}}^{2q} + \frac{\sigma_{N}^{2q}}{2^q} \E \bigg[\bigg( \int_{0}^t \big\lVert \langle v \rangle^{k}\,  \nabla_{v} f_{s}^{N} \big\rVert_{L^2_{x,v}}^{2} \dd s\bigg)^q\bigg] \Bigg)^\frac{1}{q} \\
&\lesssim \Big( \E \big\lVert \langle v \rangle^{k} f_{0}^{N} \big\rVert_{L^2_{x,v}}^{2q}\Big)^\frac{1}{q}
+ (1+\sigma_{N}^{-2}) \bigg(\E\bigg[\Big(\int_{0}^t \iint \langle v \rangle^{2k} |f^N_{s}|^2 \dd x \dd v \dd s\Big)^q\bigg] \bigg)^\frac{1}{q}\\
&\quad + N^{-2(\alpha-\beta)} + \mathfrak{C}_{\varepsilon} N^{-2\beta(\gamma-\frac{d}{p})+\varepsilon} + N^{-(\frac{1}{2} - d\beta -(d+1)\alpha)}\\
&\quad + (1+\sigma_{N}^{-2})  \bigg(\E\bigg[\bigg(\int_{0}^t \iint \langle v \rangle^{2k}  |u^{(N)}_{s} - u^{N}_{s}|^2 \, |\tilde{F}_{s}^{N}|^2  \dd x \dd v \dd s\bigg)^q\bigg] \bigg)^\frac{1}{q} .
\end{align*}
Finally, upon applying successively on the last term of the previous inequality the Lemma~\ref{lem:regf}, then the Lemma~\ref{lem:momentsF2} with Assumption~\ref{assump-PDE}\ref{hyp:F0k} and \ref{regul}, it comes
\begin{align*}
\bigg(\E\bigg[\bigg(\int_{0}^t &\iint \langle v \rangle^{2k}  |u^{(N)}_{s} - u^{N}_{s}|^2 \, |\tilde{F}_{s}^{N}|^2  \dd x \dd v \dd s\bigg)^q\bigg] \bigg)^\frac{1}{q}\\
&\leq \bigg(\E\bigg[\bigg(\int_{0}^t \| u^{(N)}_{s} - u^{N}_{s}\|_{L^\infty_{x}}^{2q} \iint \langle v \rangle^{2k} \, |\tilde{F}_{s}^{N}|^2  \dd x \dd v \dd s\bigg)^q\bigg] \bigg)^\frac{1}{q} \\
&\lesssim \bigg(\E\bigg[\bigg(\int_{0}^t \| u^{(N)}_{s} - u^{N}_{s}\|_{L^\infty_{x}}^{2} \iint \langle v \rangle^{2k} \, |F^{(N)}_{s}|^2  \dd x \dd v \dd s\bigg)^q\bigg] \bigg)^\frac{1}{q} \\
&\lesssim \bigg(\E\bigg[\bigg(\int_{0}^t \| u^{(N)}_{s} - u^{N}_{s}\|_{L^\infty_{x}}^{2}\dd s\bigg)^q\bigg] \bigg)^\frac{1}{q}
\end{align*}
and the desired conclusion follows by Minkowski's inequality.
\end{proof}

\subsection{Deterministic approximation of $(u,F)$: the case $\sigma > 0$}
\label{subsec:discuss-rateregul}

In this section and the next one, we aim at bounding the error term $\rate$ defined in~\eqref{eq:def-rate}.
We start with the case when the density of particles is regularised by the presence of the Laplace operator. In this situation, no smoothness is required on the initial conditions $(u^\circ, F^\circ)$ and the solution $(u,F)$ to \eqref{eq:PDE} does not need to be strong. The error $\rate$ is decomposed as follows:
\begin{align}
\label{eq:decomprate}
\rate^2
&\leq \sup_{t\in[0,T]} \big\| u_t^{(N)} - u_t \big\|_{ \gamma,p}^2 + 2\sup_{t\in[0,T]} \big\| \langle v \rangle^{k} (F_{t}\ast \vartheta^{N} - F_{t}) \big\|_{L^{2}_{x,v}}^2 + 2 \sup_{t\in[0,T]} \big\| \langle v \rangle^{k} (F^{(N)}_{t} - F_{t})\ast \vartheta^{N} \big\|_{L^{2}_{x,v}}^2\nonumber\\
&\eqqcolon \sup_{t\in[0,T]} \big\| u_t^{(N)} - u_t \big\|_{ \gamma,p}^2 + 2 \rateone^2 + 2\ratetwo^2 .
\end{align}
We first estimate $\rateone$, while $\ratetwo$ and $\big\| u_t^{(N)} - u_t \big\|_{ \gamma,p}$ will be bounded jointly.

\begin{proposition}
\label{prop:rho1N}
Let Assumptions~\ref{assump-PDE} and \ref{assump-particles} hold, with $\sigma > 0$.
Then there exists a constant $C=C(T,k,A,\gamma,d,p)>0$ such that for any $N\in \N^*$,
\begin{align*}
\rateone=  \sup_{t\in[0,T]}  \big\| \langle v \rangle^{k} (F_{t}\ast \vartheta^{N}-F_{t}) \big\|_{L^{2}_{x,v}}
&\leq C \Big(  \big\| \langle v \rangle^{k} (F_{0}\ast \vartheta^{N}-F_{0})  \big\|_{L^{2}_{x,v}}  + \frac{1}{N^{\beta(\gamma-\frac{d}{p})}} + \frac{1}{N^{\alpha-\beta}} \Big).
\end{align*}
\end{proposition}

\begin{remark}
Lemma~\ref{lem:rate-approx-L2} below can be used to deduce a convergence rate on the term $\big\| \langle v \rangle^{k} (F_{0}- F_{0}\ast \vartheta^{N}) \big\|_{L^{2}_{x,v}}$, provided that $F_{0}$ has some Bessel regularity. This permits to obtain a completely explicit rate of convergence for $\rate$.
\end{remark}

\begin{proof}
Denote $\overline{F}_{t}^{N} = F_{t}\ast \vartheta^{N}$. We set $g^{N}_t=F_t-\overline{F}_{t}^{N}$. Then in view of \eqref{eq:PDE} and \eqref{eq:tildeFN}, $g^{N}_t$ satisfies
\begin{align*}
g_{t}^{N} = g_{0}^{N} + \frac{\sigma^{2}}{2} \int_{0}^{t} \Delta_v g_{s}^{N} \dd s
- \int_{0}^{t} \dive_{v} ( (u-v )g_{s}^{N} ) \dd s
- \int_{0}^{t} \dive_{x} ( v g_{s}^{N} ) \dd s
+ \int_{0}^{t} ( K^{1}+ K^{2} ) \dd s ,
\end{align*}
with the following error terms
\[
K^{1}= \dive_{v} \big((u-v )\overline{F}_{s}^{N}\big)- \dive_{v} \big(\vartheta^{N} \ast (u-v )F_{s}\big)
\]
and
\[
K^{2}= \dive_{x}( v\overline{F}_{s}^{N})- \dive_{x} ( \vartheta^{N} \ast v F_{s} ).
\]

 Consider a  nonnegative smooth cut-off function $\eta$ supported on the ball of radius $2$ and such that
 $\eta\equiv 1$ on the ball of radius $1$. For each $R>0$, introduce the rescaled functions
 $\eta_R (v) =  \eta( \frac{v}{R} )$.
 By the chain rule, one gets
\begin{align}\label{soma}
 \iint \langle v \rangle^{2k} \eta_R\, |g_{t}^{N}|^{2}  \dd x \dd v &=
\iint \langle v \rangle^{2k} \eta_R\, |g_{0}^{N}|^{2} \dd x \dd v
+ \sigma^{2}\int_{0}^{t} \iint  \langle v \rangle^{2k} \eta_R\, g_{s}^{N} \Delta_{v} g_{s}^{N} \dd x \dd v \dd s \nonumber\\
&\quad -2 \int_{0}^{t} \iint \langle v \rangle^{2k} \eta_R\, g_{s}^{N}  \dive_{v} ( (u-v )g_{s}^{N} ) \dd x \dd v \dd s \nonumber\\
&\quad - \int_{0}^{t} \iint \langle v \rangle^{2k} \eta_R \dive_{x} ( v |g_{s}^{N}|^{2} ) \dd x \dd v \dd s\\
&\quad +2 \int_{0}^{t} \iint \langle v \rangle^{2k} \eta_R\, g_{s}^{N}  K^{1} \dd x \dd v \dd s+ 2 \int_{0}^{t} \iint  \langle v \rangle^{2k} \eta_R\, g_{s}^{N}  K^{2} \dd x \dd v \dd s \nonumber\\
&\eqqcolon \sum_{k=1}^6 I_{k}.\nonumber
\end{align}

 For $I_{2}$, after integration-by-parts, we make the following decomposition:
\begin{align*}
I_{2}
&= -  \sigma^{2}  \int_{0}^{t} \iint  \langle v \rangle^{2k} \eta_R|\nabla_{v} g_{s}^{N}|^{2} \dd x \dd v \dd s\\
&\quad - \sigma^{2} \int_{0}^{t} \iint \langle v \rangle^{2k} g_{s}^{N}\, \nabla_{v} \eta_R \cdot \nabla_{v} g_{s}^{N} \dd x \dd v \dd s\\
&\quad  -\sigma^{2} 2k\int_{0}^{t} \iint \langle v \rangle^{2k-2}v \cdot \eta_R \,  g_{s}^{N}\, \nabla_{v} g_{s}^{N}  \dd x \dd v \dd s
\eqqcolon I_{2,1} + I_{2,2} +I_{2,3} .
\end{align*}
The negative term $I_{2,1}$ is transferred in the LHS of equality~\eqref{soma}. For $I_{2,2}$, we use that $\|\nabla \eta_{R}\|_{\infty}\lesssim R^{-1}$ and then Young's inequality, to write
\begin{align*}
I_{2,2}
&\lesssim  \frac{1}{R}  \int_{0}^{t} \iint  \langle v \rangle^{2k} | g_{s}^{N}|^{2} \dd x \dd v \dd s
+  \frac{1}{R}   \int_{0}^{t} \iint   \langle v \rangle^{2k} |\nabla_{v} g_{s}^{N}|^{2} \dd x \dd v \dd s .
\end{align*}
For $I_{2,3}$, apply Young's inequality with some $\delta>0$ that will be fixed later, to get
\begin{align*}
I_{2,3}
\leq \frac{C}{\delta} \int_{0}^{t} \iint  \langle v \rangle^{2k} \eta_R | g_{s}^{N}|^{2}\dd x \dd v \dd s + \delta \int_{0}^{t} \iint   \langle v \rangle^{2k} \eta_R |\nabla_{v} g_{s}^{N}|^{2} \dd x \dd v \dd s.
\end{align*}
Hence, gathering the previous inequalities, we get
\begin{align}
\label{eq:I2-App}
I_{2} &+ \sigma^{2}  \int_{0}^{t} \iint  \langle v \rangle^{2k} \eta_R|\nabla_{v} g_{s}^{N}|^{2} \dd x \dd v \dd s \nonumber \\
& \leq \int_{0}^{t} \iint \Big(\frac{C}{\delta} \eta_R + \frac{C}{R} \Big) \langle v \rangle^{2k} | g_{s}^{N}|^{2}\dd x \dd v \dd s +  \int_{0}^{t} \iint \Big(\frac{C}{R} + \delta\eta_R\Big) \langle v \rangle^{2k} |\nabla_{v} g_{s}^{N}|^{2} \dd x \dd v \dd s .
\end{align}

For $I_{3}$, integrate-by-parts in a similar way to term $I_{2}$ in the proof of Lemma~\ref{lem:momentsF2}, to obtain
\begin{align}
\label{c3}
| I_{3}| & \leq C \int_{0}^{t} \iint \langle v \rangle^{2k} \eta_R\, |g_{s}^{N}|^{2} \dd x\dd v \dd s
+  \frac{C}{R}  \int_{0}^{t} \iint \langle v \rangle^{2k} (|u|+|v|) \, |g_{s}^{N}|^{2} \dd x\dd v \dd s \nonumber\\
& \leq   C \int_{0}^{t} \iint \langle v \rangle^{2k} \eta_R\, |g_{s}^{N}|^{2} \dd x\dd v \dd s
+  \frac{C \|u\|_{\lambda,p}}{R}  \int_{0}^{t} \iint \langle v \rangle^{2k} (|F_t|^2 + |\overline{F}_{t}^{N}|^{2})  \dd x\dd v \dd s \nonumber\\
&\quad +  \frac{C }{R}  \int_{0}^{t} \iint \langle v \rangle^{2k+1} (|F_t|^2 + |\overline{F}_{t}^{N}|^{2})  \dd x\dd v \dd s.
\end{align}

By integration-by-parts, one gets $I_{4}=0$.
Now for $I_{5}$,
\begin{align*}
|I_{5}| &= \bigg|\int_{0}^{t} \iint  \langle v \rangle^{2k} \eta_R\, g_{s}^{N} \Big( \dive_{v}\big( (u-v )\overline{F}_{s}^{N}\big)- \dive_{v} \big( \vartheta^{N} \ast (u-v )F_{s} \big)\Big) \dd x\dd v \dd s\bigg|\\
& \leq \bigg|\int_{0}^{t} \iint  \langle v \rangle^{2k} \eta_R\,  g_{s}^{N} \Big(\dive_{v}(  u\overline{F}_{s}^{N})- \dive_{v} ( \vartheta^{N} \ast u F_{s} ) \Big) \dd x\dd v \dd s\bigg|\\
&\quad + \bigg|\int_{0}^{t} \iint  \langle v \rangle^{2k} \eta_R\,  g_{s}^{N} \Big( \dive_{v}( v \overline{F}_{s}^{N})- \dive_{v} ( \vartheta^{N} \ast  v F_{s} )\Big) \dd x\dd v \dd s\bigg|\\
&\eqqcolon I_{5,1} +I_{5,2}.
\end{align*}
For $I_{5,1}$, by integration-by-parts, using boundedness of $\nabla \eta$, Lemma~\ref{lem:rateHolderReg} and Young's inequality with some $\delta>0$, it comes that
\begin{align}
\label{c4-1}
I_{5,1}
&\leq C \int_{0}^t \iint \langle v \rangle^{2k} \Big(2k \eta_{R}\, |g^N_{s}| + \frac{1}{R} |g^N_{s}| + \eta_{R}\, |\nabla g_{s}^N| \Big) \big| u\overline{F}_{s}^{N} - \vartheta^{N} \ast u F_{s} \big| \dd x \dd v \dd s\nonumber\\
&\leq \delta \int_{0}^{t} \iint \langle v \rangle^{2k} \eta_R\, |\nabla_{v} g_{s}^{N}|^{2} \dd x\dd v \dd s
+  \frac{C}{R} \int_{0}^{t} \iint \langle v \rangle^{2k}  | g_{s}^{N}|^{2} \dd x\dd v \dd s \nonumber\\
&\quad +  C \int_{0}^{t} \iint \langle v \rangle^{2k} \eta_R\, | g_{s}^{N}|^{2} \dd x\dd v \dd s
  +  C\frac{ \| u\|_{ \gamma,p}^{2}}{N^{2\beta(\gamma-\frac{d}{p})}} \sup_{s\in [0,T]} \big\| \langle v \rangle^{k} \overline{F}^N_s \big\|_{L^{2}_{x,v}}^{2}.
\end{align}
Similarly, for $I_{5,2}$,  note that
\[ v\overline{F}_{s}^{N}- \vartheta^{N} \ast v F_{s} = - (\vartheta^{N}v) \ast  F_{s}.\]
Then, by integration-by-parts, using boundedness of $\nabla \eta$  and Young's inequality with some $\delta>0$, and using that $\vartheta^2$ is compactly supported, it follows that
\begin{align}
\label{c4-2}
I_{5,2}& \leq  \delta \int_{0}^{t} \iint\langle v \rangle^{2k} \eta_R\,  |\nabla_{v} g_{s}^{N}|^{2} \dd x\dd v \dd s
+ \frac{C}{R} \int_{0}^{t} \iint \langle v \rangle^{2k} |g_{s}^{N}|^{2} \dd x\dd v \dd s \nonumber\\
&\quad +  C \int_{0}^{t} \iint \langle v \rangle^{2k} \eta_R\, | g_{s}^{N}|^{2} \dd x\dd v \dd s
 + \frac{C}{N^{2\alpha}} \sup_{s\in [0,T]} \big\| \langle v \rangle^{k} \overline{F}^N_s \big\|_{L^{2}_{x,v}}^{2}.
\end{align}
Hence combining \eqref{c4-1} and \eqref{c4-2} yields
\begin{equation}
\label{eq:App-I5}
\begin{split}
|I_{5} |
&\leq  \delta \int_{0}^{t} \iint\langle v \rangle^{2k} \eta_R\,  |\nabla_{v} g_{s}^{N}|^{2} \dd x\dd v \dd s
+ \frac{C}{R} \int_{0}^{t} \iint \langle v \rangle^{2k} |g_{s}^{N}|^{2} \dd x\dd v \dd s \\
&\quad +  C \int_{0}^{t} \iint \langle v \rangle^{2k} \eta_R\, | g_{s}^{N}|^{2} \dd x\dd v \dd s
 + C \Big(\frac{1}{N^{2\alpha}} + \frac{ \| u\|_{ \gamma,p}^{2}}{N^{2\beta(\gamma-\frac{d}{p})}} \Big) \sup_{s\in [0,T]} \big\| \langle v \rangle^{k} \overline{F}^N_s \big\|_{L^{2}_{x,v}}^{2}.
\end{split}
\end{equation}

For $I_{6}$, use that $v\overline{F}_{s}^{N}- \vartheta^{N} \ast (v F_{s}) =  (v\vartheta^{N}) \ast  F_{s}$ and \eqref{eq:boundVGradTheta} to get
\begin{align}
\label{c6}
|I_{6}|
&= \bigg| \int_{0}^{t} \iint \langle v \rangle^{2k} \eta_R\, g_{s}^{N} \dive_{x} \big( (v\vartheta^{N}) \ast  F_{s} \big) \dd x\dd v \dd s \bigg| \nonumber\\
&\leq \frac{C}{N^{\alpha-\beta}} \int_{0}^{t} \iint \langle v \rangle^{2k}  \eta_R\,  |g_{s}^{N}| (\vartheta^{N} \ast  F_{s} )\dd x \dd v \dd s \nonumber\\
&\leq \int_{0}^{t}  \iint \langle v \rangle^{2k}  \eta_R\, |g_{s}^{N}|^{2} \dd x\dd v \dd s
+ \frac{C}{N^{2\alpha-2\beta}} \sup_{s\in [0,T]} \big\| \langle v \rangle^{k} \overline{F}^N_s \big\|_{L^{2}_{x,v}}^{2}.
\end{align}

By using Lemma~\ref{lem:regf}, and from \eqref{soma}, \eqref{eq:I2-App}, \eqref{c3}, \eqref{eq:App-I5},  \eqref{c6}, choosing $ \delta > 0 $ small enough,
and taking the limit as $R\rightarrow\infty$ we obtain
\begin{align*}
 \iint \langle v \rangle^{2k}  |g_{t}^{N}|^{2} \dd x\dd v
 &\leq \iint \langle v \rangle^{2k}   |g_{0}^{N}|^{2} \dd x\dd v
 +  C \int_{0}^{t} \iint  \langle v \rangle^{2k} |g_{s}^{N}|^{2} \dd x\dd v \dd s \\
&\quad  +  C\left(\frac{1}{N^{2\beta(\gamma-\frac{d}{p})}} + \frac{1}{N^{2\alpha}} + \frac{1}{N^{2\alpha-2\beta}}\right).
\end{align*}
The conclusion follows with Gr\"onwall's Lemma.
\end{proof}

\begin{remark}
\label{rk:ratetilde-sigma>0}
If $\sigma_{N} \to \sigma> 0$, then following the same proof, one gets
\begin{equation*}
\ratetilde = \sup_{t\in[0,T]} \big\| \langle v \rangle^{k} (F^{(N)}_{t}\ast \vartheta^{N} - F^{(N)}_{t}) \big\|_{L^{2}_{x,v}} \leq C \Big(  \big\| \langle v \rangle^{k} (F_{0}\ast \vartheta^{N}-F_{0})  \big\|_{L^{2}_{x,v}}  + \frac{1}{N^{\beta(\gamma-\frac{d}{p})}} + \frac{1}{N^{\alpha-\beta}} \Big).
\end{equation*}
\end{remark}

We now state two lemmas that will permit to control $\ratetwo^2 + \sup_{t\in [0,T]}\big\| u_t^{(N)} - u_t \big\|_{ \gamma,p}^2$.

\begin{lemma}
\label{lem:energyFN-F}
Let Assumptions~\ref{assump-PDE} and \ref{assump-particles} hold, with $\sigma\in [0,+\infty)$.
Then there exists a constant $C=C(T,k,A,\gamma,d,p)>0$ such that for any $N\in \N^*$ and any $t\in [0,T]$,
\begin{align*}
\big\| \langle v \rangle^{k} &(F^{(N)}_{t} - F_{t}) \big\|_{L^{2}_{x,v}}^2 + \frac{\sigma_{N}^2}{2} \int_{0}^t \big\| \langle v \rangle^{k} \nabla_{v}(F^{(N)}_{s} - F_{s}) \big\|_{L^2_{x,v}}^2 \dd s\\
&\leq C \, |\sigma_{N}-\sigma| \int_{0}^t \big\| \langle v \rangle^{k} \nabla_{v}F_{s} \big\|_{L^2_{x,v}}^2 \dd s + C \int_{0}^t \Big\| \langle v \rangle^{k} (u_s^{(N)} - u_s) F^{(N)}_{s} \Big\|_{L^2_{x,v}}^2 \dd s.
\end{align*}
\end{lemma}

\begin{proof}
Up to regularising $F^{(N)}$ and $F$, and applying eventually Fatou's Lemma, we can work directly as if they were strong solutions. Mimicking the proof of Lemma~\ref{lem:momentsF2}, one gets,
\begin{align*}
\frac{\dd}{\dd t} \big\| \langle v \rangle^{k} (F^{(N)}_{t} - F_{t}) \big\|_{L^{2}_{x,v}}^2
&= -2 \iint \langle v \rangle^{2k} (F^{(N)}_{t} - F_{t})\, \nabla_{v} \cdot \big( (u^{(N)}_{t}-v) F^{(N)}_{t} - (u_{t}-v) F_{t} \big) \dd x \dd v\\
&\quad -\sigma_{N}^2 \iint \langle v \rangle^{2k} \big| \nabla_{v} (F^{(N)}_{t} - F_{t}) \big|^2 \dd x \dd v\\
&\quad - 2k\sigma_{N}^2 \iint \langle v \rangle^{2k-2} v (F^{(N)}_{t} - F_{t}) \cdot \nabla_{v} (F^{(N)}_{t} - F_{t}) \dd x \dd v\\
&\quad - (\sigma_{N}^2 - \sigma^2)  \iint \langle v \rangle^{2k} (F^{(N)}_{t} - F_{t}) \Delta_{v} F_{t} \dd x \dd v.
\end{align*}
After integration-by-parts, using Young's inequality and upon rearranging the terms, one gets the expected inequality.
\end{proof}

\begin{lemma}
\label{lem:u(N)-u}
There exists $C \equiv C(T,k,A,\lVert u \rVert_{L^\infty_{x}}, \gamma, d,p)>0$ such that for any $N\in \N^*$,
\begin{align*}
\big\| u_t^{(N)} - u_t \big\|_{ \gamma,p}
\leq C \int_{0}^t \frac{1}{(t-s)^{\tilde{\gamma}}} \Big( \big\| u_s^{(N)} - u_s \big\|_{ \gamma,p} +  \big\| \langle v \rangle^{k} (F^{(N)}_{s}-F_{s}) \big\|_{L^{2}_{x,v}} \Big) \dd s ,
\end{align*}
where $\tilde{\gamma}\in (0,1)$ is defined in \eqref{eq:deftildegamma}.
\end{lemma}

\begin{proof}
We only sketch the proof of this inequality, as similar computations will be developed in the proof of Theorem~\ref{th:convBessel} at the beginning of Section~\ref{subsec:ProofThBessel}.
Consider the  difference of the mild solutions for $u^N $ and $ u$, which upon using~\eqref{eq:umild} reads
\begin{align*}
u_{t}^{(N)}-u_{t}
&= - \int_{0}^{t} e^{\left(t-s\right)  \Delta} P \big[ (u^{(N)}_s \cdot \nabla) \bchi(u^{(N)}_s) - ( u_{s} \cdot \nabla) \bchi(u_{s}) \big] \dd s \\
&\quad-\int_{0}^{t} e^{\left(t-s\right) \Delta} P \Big[ \big(\bchi(u^{(N)}_s) m_{0}(F^{(N)}_{s})- m_{1}(F^{(N)}_{s})\big) - \big(\bchi(u_s) m_{0}(F_{s})- m_{1}(F_{s})\big) \Big] \dd s.
\end{align*}
It follows that
\begin{align}
\label{estima-0}
\| u_{t}^{(N)}-u_{t} \|_{\gamma, p}
&\leq \int_{0}^{t}  \big\| e^{(t-s)\Delta} P \big[ (u^{(N)}_s \cdot \nabla) \bchi(u^{(N)}_s) - ( u_{s} \cdot \nabla) \bchi(u_{s}) \big] \big\|_{\gamma, p}  \dd s \nonumber\\
& \quad + \int_{0}^{t}
  \Big\|e^{\left( t-s\right) \Delta}  P\Big[\bchi(u^{(N)}_s) m_{0}(F^{(N)}_{s}) - \bchi(u_s) m_{0}(F_{s})\Big] \Big\|_{\gamma, p} \dd s \\
& \quad + \int_{0}^{t}  \big\| e^{(t-s) \Delta} P\big[m_{1}(F_{s}^{(N)})- m_{1}(F_{s})\big]\big\|_{\gamma, p} \dd s.\nonumber
\end{align}
For the first term in the right-hand side of \eqref{estima-0}, by using Lemma~\ref{lem:divfreecommute} and the fact that $\bchi$ is Lipschitz continuous, we arrive at
\begin{align}
\label{eq:conv-Bessel-2-0}
\int_{0}^{t}  \big\| e^{(t-s) \Delta}
& P \nabla \cdot ( u^{(N)}_s \otimes \bchi(u^{(N)}_s) -  u_{s} \otimes \bchi(u_{s})) \big\|_{\gamma, p}  \dd s \nonumber\\
&\leq C \int_{0}^{t} \frac{1}{(t-s)^{\frac{\gamma+1}{2}}} \Big( \big\| (u_{s}^{(N)}-u_{s}) \otimes \bchi(u_{s}^{(N)}) \big\|_{L^p_{x}} + \big\| u_{s} \otimes (\bchi(u_{s}^{(N)})-\bchi(u_{s}) ) \big\|_{L^p_{x}} \Big) \dd s \nonumber\\
&\leq C \int_{0}^{t}  (1+A + \|u_{s}\|_{L^\infty_{x}}) \frac{1}{(t-s)^{\frac{\gamma+1}{2}}}   \| u_{s}^{(N)}-u_{s} \|_{L^p_{x}}  \dd s .
\end{align}

For the third term in the right-hand side of \eqref{estima-0}, recalling the definitions of the moments in \eqref{eq:defMoments}, using the semigroup property~\eqref{eq:Bessel-heat-estimate} with $r=2\leq p$ and Lemma~\ref{lem:japanese}, one deduces that for some $ k \geq 3$,
\begin{align}
\label{eq:conv-Bessel-3-0}
 \int_{0}^{t} \big\|e^{\left(t-s\right) \Delta} P\big[m_{1}(F_{s}^{(N)})- m_{1}(F_{s})\big] \big\|_{\gamma, p} \dd s
& \leq C \int_{0}^{t}  \frac{1}{(t-s)^{\frac{1}{2}(\gamma + d(\frac{1}{2}-\frac{1}{p}))}}  \big\| \langle v \rangle^{k} (F_{s}^{(N)}- F_{s})\big\|_{L^{2}_{x,v}} \dd s .
\end{align}

For the second term, we have
\begin{equation}
\label{eq:3rdterm-0}
\begin{split}
 \int_{0}^{t} &\Big\|e^{\left(t-s\right) \Delta} P \Big[ \bchi(u^{(N)}_s) m_{0}(F^{(N)}_{s}) - \bchi(u_s) m_{0}(F_{s}) \Big]\Big\|_{\gamma, p} \dd s \\
& \leq  \int_{0}^{t} \Big\|e^{\left(t-s\right)\Delta} P\Big[ \bchi(u_{s}^{(N)})\, \big( m_{0}(F_{s}^{(N)}) - m_{0}(F_{s})\big) \Big] \Big\|_{\gamma, p} \dd s \\
&\quad + \int_{0}^{t}  \Big\|e^{\left(t-s\right)\Delta} P\Big[ \big(\bchi(u_{s}^{(N)}) - \bchi(u_s) \big)\, m_{0}(F_{s}) \Big] \Big\|_{\gamma, p} \dd s
 \eqqcolon J_{1}+J_{2}.
\end{split}
\end{equation}
 For $J_{1}$, by the semigroup property~\eqref{eq:Bessel-heat-estimate}, the boundedness of $\bchi$ and Lemma~\ref{lem:japanese}, it comes
\begin{align*}
J_{1}
 &\leq C \int_{0}^{t}  \frac{1}{(t-s)^{\frac{1}{2}(\gamma + d(\frac{1}{2}-\frac{1}{p}))}} \big\| \langle v \rangle^{k} (F_{s}^{(N)}- F_{s}) \big\|_{L^{2}_{x,v}}  \dd s ,
\end{align*}
for some $k \geq 3$.
Similarly for $J_{2}$,
\begin{align*}
J_2
&\leq C\int_{0}^{t}  \frac{1 }{(t-s)^{\frac{1}{2}(\gamma + d(\frac{1}{2}-\frac{1}{p}))}} \|  u_{s}^{(N)}-u_{s} \|_{\gamma,p} \| m_{0}(F_{s})\|_{L^2_{x}}  \dd s ,
\end{align*}
where we used that $\bchi$ is Lipschitz and that $\|u_{s}^N-u_{s} \|_{L^\infty_{x}} \lesssim \|  u_{s}^N-u_{s} \|_{\gamma,p}$. Hence combining the bounds on $J_{1}$ and $J_{2}$ in \eqref{eq:3rdterm-0} yields
\begin{equation}\label{eq:conv-Bessel-4-0}
\begin{split}
&  \int_{0}^{t} \Big\|e^{\left(t-s\right) \Delta} P \Big[ \bchi(u^{(N)}_s) m_{0}(F^{(N)}_{s}) - \bchi(u_s) m_{0}(F_{s}) \Big]\Big\|_{\gamma, p} \dd s\\
 &\quad \lesssim \int_{0}^{t}  \frac{1 }{(t-s)^{\frac{1}{2}(\gamma + d(\frac{1}{2}-\frac{1}{p}))}}
 \left(\big\| \langle v \rangle^{k} (F_{s}^{(N)}- F_{s}) \big\|_{L^{2}_{x,v}}
 + \|  u_{s}^{(N)}-u_{s} \|_{\gamma,p} \| m_{0}(F_{s})\|_{L^2_{x}} \right) \dd s.
\end{split}
\end{equation}

Plugging the estimates \eqref{eq:conv-Bessel-2-0}, \eqref{eq:conv-Bessel-3-0} and \eqref{eq:conv-Bessel-4-0} into \eqref{estima-0} yields the result.
\end{proof}

We are now ready to formulate the main result of this subsection, which gives a bound on $\rate$ when $\sigma > 0$.
\begin{proposition}
\label{prop:rate-rhoN}
Let Assumptions~\ref{assump-PDE} and \ref{assump-particles} hold, with $\sigma > 0$.
Then there exists a constant $C=C(T,k,A,\gamma,d,p)>0$ such that for any $N\in \N^*$,
\begin{align*}
\rate
&\leq C \Big(  \big\| \langle v \rangle^{k} (F_{0}\ast \vartheta^{N}-F_{0})  \big\|_{L^{2}_{x,v}}  + \frac{1}{N^{\beta(\gamma-\frac{d}{p})}} + \frac{1}{N^{\alpha-\beta}} +|\sigma_{N}-\sigma|^\frac{1}{2} \Big).
\end{align*}
\end{proposition}

\begin{proof}
First, observe that when $\sigma>0$, Lemma~\ref{lem:momentsF2} gives that $\int_{0}^T \big\| \langle v \rangle^{k} \nabla_{v}F_{s} \big\|_{L^2_{x,v}}^2 \dd s <\infty$. Besides, in view of the bound from Lemma~\ref{lem:boundsuNFN}, it comes from Lemma~\ref{lem:energyFN-F} that for any $t\in [0,T]$,
\begin{align}
\label{eq:rateN-bound1}
\big\| \langle v \rangle^{k} (F^{(N)}_{t} - F_{t}) \big\|_{L^{2}_{x,v}}^2 + \frac{\sigma_{N}^2}{2} \int_{0}^t \big\| \langle v \rangle^{k} \nabla_{v}(F^{(N)}_{s} - F_{s}) \big\|_{L^2_{x,v}}^2 \dd s
\lesssim  \, |\sigma_{N}-\sigma| + \int_{0}^t \big\|u_s^{(N)} - u_s \big\|_{L^\infty_{x}}^2 \dd s.
\end{align}
Now, observe that for a nonnegative and nondecreasing function $f:[0,T]\to \R$, one has for any $t\leq T$,
\begin{align}
\label{eq:trick-kernel}
\sup_{s\in [0,t]} \int_{0}^s (s-r)^{-\tilde\gamma} f(r)\dd r
&\leq \sup_{s\in [0,t]} \int_{0}^s u^{-\tilde\gamma} f(s-u)\dd u\nonumber\\
&\leq \int_{0}^t u^{-\tilde\gamma} f(t-u)\dd u = \int_{0}^t (t-r)^{-\tilde\gamma} f(r)\dd r .
\end{align}
Hence in view of the embedding $H^\gamma_{p} \hookrightarrow L^\infty_{x}$ and \eqref{eq:trick-kernel}, Lemma~\ref{lem:u(N)-u} yields
\begin{align*}
\sup_{s\in [0,t]} \big\| u_s^{(N)} - u_s \big\|_{ \gamma,p}
\lesssim \int_{0}^t \frac{1}{(t-r)^{\tilde{\gamma}}} \Big(\sup_{s\in [0,r]} \big\| u_s^{(N)} - u_s \big\|_{ \gamma,p} + \sup_{s\in [0,r]} \big\| \langle v \rangle^{k} (F^{(N)}_{s}-F_{s}) \big\|_{L^{2}_{x,v}} \Big) \dd r ,
\end{align*}
and by Jensen's inequality,
\begin{align}
\label{eq:rateN-bound2}
\sup_{s\in [0,t]} \big\| u_s^{(N)} - u_s \big\|_{ \gamma,p}^2
\lesssim t^{1-\tilde{\gamma}} \int_{0}^t \frac{1}{(t-r)^{\tilde{\gamma}}} \Big(\sup_{s\in [0,r]} \big\| u_s^{(N)} - u_s \big\|_{ \gamma,p}^2 + \sup_{s\in [0,r]} \big\| \langle v \rangle^{k} (F^{(N)}_{s}-F_{s}) \big\|_{L^{2}_{x,v}}^2 \Big) \dd r .
\end{align}
Combining \eqref{eq:rateN-bound1} and \eqref{eq:rateN-bound2} gives, for any $t\in [0,T]$,
\begin{align*}
\sup_{s\in [0,t]} &\big\| u_s^{(N)} - u_s \big\|_{ \gamma,p}^2 + \sup_{s\in [0,t]} \big\| \langle v \rangle^{k} (F^{(N)}_{s} - F_{s}) \big\|_{L^{2}_{x,v}}^2 \\
& \lesssim |\sigma_{N}-\sigma| + \int_{0}^t \Big(1+ \frac{1}{(t-r)^{\tilde{\gamma}}}\Big) \Big( \sup_{s\in [0,r]} \big\| u_s^{(N)} - u_s \big\|_{ \gamma,p}^2 + \sup_{s\in [0,r]} \big\| \langle v \rangle^{k} (F^{(N)}_{s}-F_{s}) \big\|_{L^{2}_{x,v}}^2 \Big) \dd r.
\end{align*}
Employing the convolutional-type Gr\"onwall lemma from \cite[Lemma 7.1.1]{Henry} yields
\begin{align*}
\sup_{s\in [0,t]} \big\| u_s^{(N)} - u_s \big\|_{ \gamma,p}^2 + \sup_{s\in [0,t]} \big\| \langle v \rangle^{k} (F^{(N)}_{s} - F_{s}) \big\|_{L^{2}_{x,v}}^2
 \lesssim |\sigma_{N}-\sigma| .
\end{align*}

Finally, recall from the beginning of this section that $\rate \leq \sup_{t\in[0,T]} \big\| u_t^{(N)} - u_t \big\|_{ \gamma,p}^2 + 2\sup_{t\in[0,T]} \big\| \langle v \rangle^{k} (F_{t}\ast \vartheta^{N} - F_{t}) \big\|_{L^{2}_{x,v}}^2 + 2 \sup_{t\in[0,T]} \big\| \langle v \rangle^{k} (F^{(N)}_{t} - F_{t})\ast \vartheta^{N} \big\|_{L^{2}_{x,v}}^2$. Thus by Young's convolution inequality and the previous bound, combined with Proposition~\ref{prop:rho1N}, the result follows.
\end{proof}

\subsection{Deterministic approximation of $(u,F)$: the case $\sigma=0$}
\label{subsec:discuss-rateregul-0}

In this section, we proceed similarly to the previous section, relying on the bound~\eqref{eq:decomprate} on $\rate^2$.
But unlike the previous case, when $\sigma=0$, we cannot rely on the regularisation effect of $\sigma \Delta$, hence we need some regularity on $F^\circ$ to be transported over time so that $\rate \to 0$.

A first technical tool to control $\rateone$ is the following lemma. Then we will deduce in Proposition~\ref{prop:rate-sigma0} an explicit rate of convergence for $\rateone$, assuming that $F$ has some Bessel regularity.

\begin{lemma}
\label{lem:rate-approx-L2}
Let $\gamma>0$ and let $\vartheta^N$ be the mollifier on $\T^d\times \R^d$ defined in \eqref{eq:defThetaN}. There exists $C=C(d,\gamma,\vartheta^1,\vartheta^2)$ such that for any $f \in H^{\gamma}_{2}(\T^d\times \R^d)$ and $N\in \N^*$,
\begin{align*}
\big\lVert \langle v\rangle^k (f -f\ast \vartheta^N) \big\rVert_{L^2_{x,v}} \leq C \, \lVert \langle v\rangle^{2k} f \rVert_{L^2_{x,v}}^{\frac{1}{2}} \,  \lVert f \rVert_{\gamma,2}^{\frac{1}{2}}\, N^{-\frac{1}{2}(\alpha \wedge \beta)(\gamma\wedge 1)} .
\end{align*}
\end{lemma}

\begin{proof}
First, by the Cauchy-Schwarz inequality,
\begin{align*}
\big\lVert \langle v\rangle^k (f -f\ast \vartheta^N) \big\rVert_{L^2_{x,v}}^2
&\leq \big\lVert \langle v\rangle^{2k} (f -f\ast \vartheta^N) \big\rVert_{L^2_{x,v}} \, \lVert f -f\ast \vartheta^N \rVert_{L^2_{x,v}}\\
&\leq 2 \lVert \langle v\rangle^{2k} f \rVert_{L^2_{x,v}} \, \lVert f -f\ast \vartheta^N \rVert_{L^2_{x,v}} ,
\end{align*}
using Lemma~\ref{lem:regf} in the last inequality. The approximation of $f$ in $L^2(\T^d\times \R^d)$ can be carried out as follows: by the Fourier isometry,
\begin{align*}
 \lVert f -f\ast \vartheta^N \rVert_{L^2_{x,v}}^2 =  \sum_{k\in \Z^d} \int_{\R^d} |\widehat{\mathcal{F}f}(\xi,k) |^2  \big|1-(2\pi)^d \widehat{\vartheta^{1,N}}(k) \mathcal{F}\vartheta^{2,N}(\xi) \big|^2 \, \dd \xi ,
\end{align*}
where $\mathcal{F}$ denotes the Fourier transform on $\R^d$ and $\widehat{\cdot}$ denotes the Fourier coefficients of a periodic function. Hence it comes, using the definitions \eqref{eq:defTheta0} and \eqref{eq:defTheta2N} of respectively $\vartheta^{1,N}$ and $\vartheta^{2,N}$, and the fact that $\vartheta^1$ and $\vartheta^2$ are densities, that
\begin{align*}
 \lVert f -f\ast \vartheta^N \rVert_{L^2_{x,v}}^2
&=  \sum_{k\in \Z^d} \int_{\R^d} |\widehat{\mathcal{F}f}(\xi,k) |^2  \big|1-(2\pi)^d \widehat{\vartheta^{1,N}}(k) \mathcal{F}\vartheta^{2,N}(\xi) \big|^2 \, \dd \xi \\
&= (2\pi)^{2d} \sum_{k\in \Z^d} \int_{\R^d} |\widehat{\mathcal{F}f}(\xi,k) |^2  \big|\widehat{\vartheta^1}(0)\mathcal{F}\vartheta^2(0) - \widehat{\vartheta^{1,N}}(k) \mathcal{F}\vartheta^{2,N}(\xi) \big|^2 \, \dd \xi \\
&= (2\pi)^{2d} \sum_{k\in \Z^d} \int_{\R^d} |\widehat{\mathcal{F}f}(\xi,k) |^2  \big|\widehat{\vartheta^1}(0)\mathcal{F}\vartheta^2(0) - \widehat{\vartheta^{1}}(N^{-\beta} k) \mathcal{F}\vartheta^{2}(N^{-\alpha}\xi) \big|^2 \, \dd \xi .
\end{align*}
Now by interpolating between the bounds
\begin{align*}
\big|\widehat{\vartheta^1}(0)\mathcal{F}\vartheta^2(0) - \widehat{\vartheta^{1}}(N^{-\beta} k) \mathcal{F}\vartheta^{2}(N^{-\alpha}\xi) \big| \leq 2 \lVert \widehat{\vartheta^1}\mathcal{F}\vartheta^2 \rVert_{\infty}
\end{align*}
and
\begin{align*}
\big|\widehat{\vartheta^1}(0)\mathcal{F}\vartheta^2(0) - \widehat{\vartheta^{1}}(N^{-\beta} k) \mathcal{F}\vartheta^{2}(N^{-\alpha}\xi) \big| \leq |(N^{-\beta}\xi, N^{-\alpha} k)|\, \big\lVert \nabla (\widehat{\vartheta^1}\mathcal{F}\vartheta^2)\big\rVert_{\infty} ,
\end{align*}
one gets, for $\gamma\in (0,1]$,
\begin{align*}
 \lVert f -f\ast \vartheta^N \rVert_{L^2_{x,v}}^2
 &\leq C \sum_{k\in \Z^d} \int_{\R^d} |\widehat{\mathcal{F}f}(\xi,k) |^2  |(N^{-\beta}\xi, N^{-\alpha} k)|^{2\gamma} \dd \xi \\
&\leq C N^{-2(\alpha\wedge \beta)\gamma} \sum_{k\in \Z^d} \int_{\R^d} |\widehat{\mathcal{F}f}(\xi,k) |^2  \big(|\xi|^{2\gamma} + |k|^{2\gamma} \big) \dd \xi\\
&\leq C N^{-2(\alpha\wedge \beta)\gamma} \lVert f\rVert_{\gamma,2}^2 ,
\end{align*}
where $C$ now depends on $\gamma$ and the $W^{1,\infty}$-norm of $\widehat{\vartheta^1}\mathcal{F}\vartheta^2$.
\end{proof}

Our goal now is to exploit Lemma~\ref{lem:rate-approx-L2} to deduce a bound on $\rateone$. This approach naturally requires to have some information on the regularity and integrability of $F$. There are certainly many possible assumptions on $(u^\circ,F^\circ)$ leading to $F$ having some Bessel regularity uniformly on $[0,T]$: we discuss briefly in Remark~\ref{rk:reg(u,F)-sigma=0} some of these possibilities. To  formulate a general statement in the next proposition, we assume that $F$ has some given regularity rather than making specific assumptions on $(u^\circ,F^\circ)$.

\begin{proposition}\label{prop:rate-sigma0}
Let Assumptions~\ref{assump-PDE} and \ref{assump-particles} hold, with $\sigma=0$. Let us further assume that $F \in L^\infty([0,T]; H^\gamma_{2}(\T^d\times \R^d))$, for some $\gamma>0$. Then there exists $C>0$ depending on $d,\gamma,\vartheta^1,\vartheta^2$ as well as on $\lVert \langle v\rangle^{2k} F^\circ \rVert_{L^2_{x,v}}$, $\lVert F \rVert_{L^\infty_{t} H^\gamma_{2}}$ and $ \lVert u \rVert_{L^\infty_{t}L^\infty_{x}}$, such that for any $N\in \N^*$,
\begin{equation*}
\rateone=  \sup_{t\in[0,T]} \big\| \langle v \rangle^{k} (F_{t}- F_{t}\ast \vartheta^{N}) \big\|_{L^{2}_{x,v}} \leq C N^{-\frac{1}{2} \beta(\gamma\wedge 1)} .
\end{equation*}
\end{proposition}
\begin{proof}
Applying Lemma~\ref{lem:rate-approx-L2}, we get
\begin{align*}
\rate \lesssim
\lVert \langle v\rangle^{2k} F_{t} \rVert_{L^2_{x,v}}^{\frac{1}{2}} \, \lVert F_{t} \rVert_{\gamma,2}^{\frac{1}{2}}\, N^{-\frac{1}{2}(\alpha \wedge \beta)(\gamma\wedge 1)} .
\end{align*}
The result follows upon using that $\alpha\geq \beta$ and applying Lemma~\ref{lem:momentsF2} to bound $\lVert \langle v\rangle^{2k} F_{t} \rVert_{L^\infty_{t} L^2_{x,v}}$ by a factor of $\lVert \langle v\rangle^{2k} F^\circ \rVert_{L^2_{x,v}}$.
\end{proof}

\begin{remark}
\label{rk:reg(u,F)-sigma=0}
We briefly mention here some sufficient conditions that ensure that the assumption $F \in L^\infty([0,T]; H^\gamma_{2}(\T^d\times \R^d))$ from the previous proposition is fulfilled.

First, in \cite{Choi_2015}, the authors prove the strong well-posedness of the Vlasov--Navier--Stokes equation in dimension $3$ with the following assumptions and regularity: if $u^\circ\in H^2_{2}(\T^3)^3$, $F^\circ\in H^2_{2}(\T^3\times \R^3)$ and $F^\circ$ has compact support, then $F\in \mathcal{C}([0,T]; H^{2}_{2}(\T^3\times \R^3))$.

Alternatively, adapting the proof of \cite[Theorem 3.19]{BCD}, one could prove that if $u\in L^1([0,T]; H^\gamma_{2}(\T^d)^d)$ and $F^\circ \in H^\gamma_{2}(\T^d\times \R^d)$ with $\gamma>1+\frac{d}{2}$, then $F\in \mathcal{C}([0,T];H^\gamma_{2}(\T^d\times \R^d))$. However proving rigorously this result would go beyond the scope of this work.
\end{remark}

\begin{remark}
\label{rk:ratetilde-sigma=0}
Under the assumptions of Proposition~\ref{prop:rate-sigma0}, the same scheme of proof yields
\begin{equation*}
\ratetilde  \leq C N^{-\frac{1}{2} \beta(\gamma\wedge 1)} .
\end{equation*}
\end{remark}

Now we can state the main result of this section, which provides a rate of convergence on $\rate$.
\begin{proposition}
\label{prop:rate-rhoN-sigma0}
Let Assumptions~\ref{assump-PDE} and \ref{assump-particles} hold, with $\sigma=0$. Let us further assume that $F \in L^\infty([0,T]; H^\gamma_{2}(\T^d\times \R^d))$ for some $\gamma>0$, and that $\langle v\rangle^k \nabla_{v} F \in L^2([0,T];L^2(\T^d\times\R^d))$.
Then there exists $C>0$ depending on $T,k,A,\gamma,d,p,\vartheta^1,\vartheta^2$ as well as on $\lVert \langle v\rangle^{2k} F^\circ \rVert_{L^2_{x,v}}$, $\lVert F \rVert_{L^\infty_{t} H^\gamma_{2}}$, $ \lVert u \rVert_{L^\infty_{t}L^\infty_{x}}$ and $ \lVert \langle v\rangle^k \nabla_{v} F \rVert_{L^2_{t} L^2_{x,v}}$, such that for any $N\in \N^*$,
\begin{equation*}
\rate \leq C \sigma_{N}^\frac{1}{2} .
\end{equation*}
\end{proposition}

\begin{proof}
In view of Lemma~\ref{lem:energyFN-F} and Lemma~\ref{lem:u(N)-u}, we obtain as in  the proof of Proposition~\ref{prop:rate-rhoN} that
\begin{align*}
\sup_{s\in [0,t]} \big\| u_s^{(N)} - u_s \big\|_{ \gamma,p}^2 + \sup_{s\in [0,t]} \big\| \langle v \rangle^{k} (F^{(N)}_{s} - F_{s}) \big\|_{L^{2}_{x,v}}^2
 \leq C \sigma_{N} \, \lVert \langle v\rangle^k \nabla_{v} F \rVert_{L^2_{t} L^2_{x,v}}^2.
\end{align*}
Hence, it follows from~\eqref{eq:decomprate} and Proposition~\ref{prop:rate-sigma0} that
\begin{align*}
\rate^2 \leq C N^{-\beta(\gamma\wedge 1)} + C \sigma_{N} .
\end{align*}
From Assumption~\ref{assump-particles}\ref{hyp:sigma}, we deduce that $N^{-\beta(\gamma\wedge 1)}$ is negligible compared to $\sigma_{N}$, and the result follows.
\end{proof}

\section{Proofs of the main results}
\label{sec:proofs}

\subsection{Convergence in Bessel norm: Proof of Theorem~\ref{th:convBessel}}
\label{subsec:ProofThBessel}

By the triangle inequality,
\begin{equation}
\label{eq:split-u-uN}
\begin{split}
\Big(\E &\sup_{t\in[0,T]} \big\| u^{N}_t- u_t \big\|_{ \gamma,p}^{2q}\Big)^{\frac{1}{q}}
+  \Big(\E \sup_{t\in[0,T]} \big\| \langle v \rangle^{k}(F_{t}^{N}-F_{t}) \big\|_{L^{2}_{x,v}}^{2q} \Big)^{\frac{1}{q}} \\
& \leq 2 \Big(\E \sup_{t\in[0,T]} \big\| u^{N}_t- u^{(N)}_t \big\|_{ \gamma,p}^{2q}\Big)^{\frac{1}{q}}
+  \Big(\E \sup_{t\in[0,T]} \big\| \langle v \rangle^{k}(F_{t}^{N}-F^{(N)}_{t}) \big\|_{L^{2}_{x,v}}^{2q} \Big)^{\frac{1}{q}} + 2 \rate^2 ,
\end{split}
\end{equation}
where we recall that $(u^{(N)},F^{(N)})$ is defined in~\eqref{eq:intermediatePDE} and that $\rate$ is defined in~\eqref{eq:def-rate}. Bounding $\rate$ was the purpose of Sections~\ref{subsec:discuss-rateregul} and~\ref{subsec:discuss-rateregul-0}, and the results are summarised in Remark~\ref{rk:rhoN}. Thus we focus now on bounding $u^N-u^{(N)}$ and $F^N-F^{(N)}$.

\paragraph{Step $1$ -- Bound on $\| u_{t}^{N}-u^{(N)}_{t} \|_{\gamma, p} $.}
We consider the mild equation fulfilled by $u^{(N)}$, which is identical to~\eqref{eq:umild} up to the presence of cut-offs, and upon using~\eqref{eq:mildeq}, the  difference of $u^N$ and $u^{(N)}$ reads
\begin{align*}
u_{t}^{N}-u^{(N)}_{t}
&= e^{t  \Delta}(u_{0}^{N}-u_{0}) -
\int_{0}^{t} e^{\left(t-s\right) \Delta} P \big[ (u^{N}_s \cdot \nabla) \bchi(u^N_s) - ( u^{(N)}_{s} \cdot \nabla) \bchi(u^{(N)}_{s}) \big] \dd s \\
&\quad-\int_{0}^{t} e^{\left(t-s\right) \Delta}  P \Big[ \frac{1}{N} \sum_{i=1}^{N} ( \bchi\big(u_{s}^{N}(X_{s}^{i,N})\big)-V_{s}^{i,N}) \delta_{X^{i,N}_s}^{N}- (\bchi(u^{(N)}_s) m_{0}(F^{(N)}_{s})- m_{1}(F^{(N)}_{s})) \Big] \dd s.
\end{align*}
It follows that
\begin{align}
\label{estima}
\| u_{t}^{N}-u^{(N)}_{t} \|_{\gamma, p}
&\leq \| e^{t \Delta} (u_{0}^{N}-u_{0})\|_{\gamma, p}
+ \int_{0}^{t}  \big\| e^{(t-s)\Delta} P \big[ (u^{N}_s \cdot \nabla) \bchi(u^N_s) - ( u^{(N)}_{s} \cdot \nabla) \bchi(u^{(N)}_{s}) \big] \big\|_{\gamma, p}  \dd s \nonumber\\
& \quad + \int_{0}^{t}
  \Big\|e^{\left(  t-s\right)  \Delta}  P\Big[\frac{1}{N} \sum_{i=1}^{N}  \bchi\big(u_{s}^{N}(X_{s}^{i,N})\big)\delta_{X^{i,N}}^{N}- \bchi(u^{(N)}_s) m_{0}(F^{(N)}_{s}) \Big] \Big\|_{\gamma, p} \dd s \\
& \quad + \int_{0}^{t} \big\| e^{(t-s) \Delta} P\big[m_{1}(F_{s}^{N})- m_{1}(F^{(N)}_{s})\big]\big\|_{\gamma, p} \dd s.\nonumber
\end{align}

We first  notice that
\begin{equation}
\label{eq:conv-Bessel-1}
\| e^{t \Delta }( u_{0}^{N}-u_{0})\|_{\gamma, p}\leq  \|  u_{0}^{N}-u_{0}\|_{\gamma, p}.
\end{equation}

For the second term in the right-hand side of \eqref{estima}, by using Lemma~\ref{lem:divfreecommute} and the fact that $\bchi$ is Lipschitz continuous, we arrive at
\begin{align}
\label{eq:conv-Bessel-2}
\int_{0}^{t}  \big\| e^{(t-s) \Delta}
& P \nabla \cdot ( u^{N}_s \otimes \bchi(u^N_s) -  u^{(N)}_{s} \otimes \bchi(u^{(N)}_{s})) \big\|_{\gamma, p}  \dd s \nonumber\\
&\leq C \int_{0}^{t} \frac{1}{(t-s)^{\frac{\gamma+1}{2}}} \| u^{N}_s \otimes \bchi(u_{s}^N) -  u^{(N)}_{s} \otimes \bchi(u^{(N)}_{s}) \|_{L^p_{x}}  \dd s \nonumber\\
&\leq C \int_{0}^{t} \frac{1}{(t-s)^{\frac{\gamma+1}{2}}} \Big( \big\| (u_{s}^N-u^{(N)}_{s}) \otimes \bchi(u_{s}^N) \big\|_{L^p_{x}} + \big\| u^{(N)}_{s} \otimes (\bchi(u_{s}^N)-\bchi(u^{(N)}_{s}) ) \big\|_{L^p_{x}} \Big) \dd s \nonumber\\
&\leq C \int_{0}^{t}  (1+A + \|u^{(N)}_{s}\|_{L^\infty_{x}}) \frac{1}{(t-s)^{\frac{\gamma+1}{2}}}   \|  u_{s}^N-u^{(N)}_{s} \|_{L^p_{x}}  \dd s .
\end{align}
From Lemma~\ref{lem:boundsuNFN}, recall that $\sup_{N} \|u^{(N)}_{s}\|_{L^\infty_{x}} <\infty$, which will be used at the end of this proof.

For the fourth term in the right-hand side of \eqref{estima}, recalling the definitions of the moments in \eqref{eq:defMoments}, use again \eqref{eq:Bessel-heat-estimate} with $r=2\leq p$ to deduce, for $ k \geq 3$,
\begin{align}
\label{eq:conv-Bessel-3}
 \int_{0}^{t} \big\|e^{\left( t-s\right)\Delta} P\big[m_{1}(F_{s}^{N})- m_{1}(F^{(N)}_{s})\big] \big\|_{\gamma, p} \dd s
 &\leq C \int_{0}^{t} \frac{1}{(t-s)^{\frac{1}{2}(\gamma + d(\frac{1}{2}-\frac{1}{p}))}}  \big\|m_{1}(F_{s}^{N}) - m_{1}(F^{(N)}_{s})\big\|_{L^2_{x}} \dd s \nonumber\\
& \leq C \int_{0}^{t}  \frac{1}{(t-s)^{\frac{1}{2}(\gamma + d(\frac{1}{2}-\frac{1}{p}))}}  \big\| \langle v \rangle^{k} (F_{s}^{N}- F^{(N)}_{s})\big\|_{L^{2}_{x,v}} \dd s .
\end{align}

For the third term, we introduce the following decomposition:
\begin{equation}
\label{eq:3rdterm}
\begin{split}
 \int_{0}^{t} \Big\|e^{\left(t-s\right) \Delta}& P \Big[ \frac{1}{N}\sum_{i=1}^{N}  \bchi\big(u_{s}^{N}(X_{s}^{i,N})\big)\delta_{X^{i,N}_s}^{N}- \bchi(u^{(N)}_s) m_{0}(F^{(N)}_{s}) \Big]\Big\|_{\gamma, p} \dd s \\
&\leq \int_{0}^{t}  \Big\|e^{\left(t-s\right)  \Delta} P \Big[ \frac{1}{N}\sum_{i=1}^{N}  \big(\bchi\big(u_{s}^{N}(X_{s}^{i,N})\big)- \bchi\big(u_{s}^{N}(\cdot)\big)\big) \delta_{X^{i,N}_s}^{N} \Big]\Big\|_{\gamma, p} \dd s\\
&\quad + \int_{0}^{t} \Big\|e^{\left(t-s\right)\Delta} P\Big[ \bchi(u_{s}^{N})\, \big( m_{0}(F_{s}^{N}) - m_{0}(F^{(N)}_{s})\big) \Big] \Big\|_{\gamma, p} \dd s\\
&\quad + \int_{0}^{t}  \Big\|e^{\left(t-s\right)\Delta} P\Big[ \big(\bchi(u_{s}^{N}) - \bchi(u^{(N)}_s) \big)\, m_{0}(F^{(N)}_{s}) \Big] \Big\|_{\gamma, p} \dd s  \eqqcolon J_{1} + J_{2}+J_3.
\end{split}
\end{equation}
By the property of the semigroup \eqref{eq:Bessel-heat-estimate} and the fact that $P$ is continuous in $L^p$,
\begin{align*}
J_{1}
 &\leq \int_{0}^{t} \Big\|e^{\left(t-s\right)\Delta} P\Big[\frac{1}{N} \sum_{i=1}^{N}  \big(\bchi\big(u_{s}^{N}(X_{s}^{i,N})\big) - \bchi\big(u_{s}^{N}(\cdot)\big)\big) \delta_{X^{i,N}}^{N} \Big] \Big\|_{\gamma, p} \dd s \\
 &\leq \int_{0}^{t} \frac{C}{(t-s)^{\frac{1}{2}(\gamma + d(\frac{1}{2}-\frac{1}{p}))}}  \Big\| \frac{1}{N} \sum_{i=1}^{N}  \big(\bchi\big(u_{s}^{N}(X_{s}^{i,N})\big) - \bchi\big(u_{s}^{N}(\cdot)\big)\big) \delta_{X^{i,N}}^{N} \Big\|_{L^2_{x}} \dd s .
\end{align*}
Note that when $d=2$, we have $\frac{1}{2}(\gamma + d(\frac{1}{2}-\frac{1}{p})) < 1$, using that $\gamma$ is assumed to be smaller than $1$. When $d=3$, we recall that it was further assumed that $\frac{1}{2}(\gamma + 3(\frac{1}{2}-\frac{1}{p}))<1$.
Using that $u^N$ is H\"older continuous, we get from Lemma~\ref{lem:rateHolderReg} with $\mu \equiv \frac{1}{N}\sum_{i=1}^{N}  \delta_{X^{i,N}}$ that
\begin{align*}
J_{1}
  &\leq  \frac{C}{N^{\beta(\gamma-\frac{d}{p})}} \int_{0}^{t} \frac{\| u_{s}^{N} \|_{\gamma, p} }{(t-s)^{\frac{1}{2}(\gamma + d(\frac{1}{2}-\frac{1}{p}))}}
\| m_{0}(F_{s}^{N})\|_{L_{x}^2} \dd s \\
& \leq  \frac{C}{N^{\beta(\gamma-\frac{d}{p})}} \int_{0}^{t} \frac{\| u_{s}^{N} \|_{\gamma, p} }{(t-s)^{\frac{1}{2}(\gamma + d(\frac{1}{2}-\frac{1}{p}))}}
\|  \langle v \rangle^{k} F_{s}^{N}\|_{L^2_{x,v}} \dd s ,
\end{align*}
for $k \geq 3$, using Lemma~\ref{lem:japanese} in the last inequality. For $J_{2}$, again by the property \eqref{eq:Bessel-heat-estimate} and the boundedness of $\bchi$, it comes
\begin{align*}
J_{2}
 &\leq C(1+A)\int_{0}^{t}  \frac{1 }{(t-s)^{\frac{1}{2}(\gamma + d(\frac{1}{2}-\frac{1}{p}))}}  \| m_{0}(F_{s}^{N})-m_{0}(F^{(N)}_{s}) \|_{L^2_{x}} \dd s \\
 &\leq C \int_{0}^{t}  \frac{1 }{(t-s)^{\frac{1}{2}(\gamma + d(\frac{1}{2}-\frac{1}{p}))}} \big\| \langle v \rangle^{k} (F_{s}^{N}- F^{(N)}_{s}) \big\|_{L^{2}_{x,v}}  \dd s ,
\end{align*}
for $k \geq 3$, using Lemma~\ref{lem:japanese} in the second inequality.
Similarly for $J_{3}$,
\begin{align*}
J_3 &\leq C\int_{0}^{t}  \frac{1 }{(t-s)^{\frac{1}{2}(\gamma + d(\frac{1}{2}-\frac{1}{p}))}} \big\|\big(\bchi(u_{s}^{N}) - \bchi(u^{(N)}_s) \big)\, m_{0}(F^{(N)}_{s}) \big\|_{L^2_{x}} \dd s \\
&\leq C\int_{0}^{t}  \frac{1 }{(t-s)^{\frac{1}{2}(\gamma + d(\frac{1}{2}-\frac{1}{p}))}} \|  u_{s}^N-u^{(N)}_{s} \|_{\gamma,p} \| m_{0}(F^{(N)}_{s})\|_{L^2_{x}}  \dd s ,
\end{align*}
where we used that $\bchi$ is Lipschitz and that $\|u_{s}^N-u^{(N)}_{s} \|_{L^\infty_{x}} \lesssim \|  u_{s}^N-u^{(N)}_{s} \|_{\gamma,p}$ in the last inequality. Hence combining the bounds on $J_{1}$, $J_{2}$ and $J_{3}$ in \eqref{eq:3rdterm} yields
\begin{equation}\label{eq:conv-Bessel-4}
\begin{split}
& \int_{0}^{t} \Big\|e^{\left(t-s\right) \Delta} P \Big[ \frac{1}{N}\sum_{i=1}^{N}  \bchi\big(u_{s}^{N}(X_{s}^{i,N})\big)\delta_{X^{i,N}_s}^{N}- \bchi(u^{(N)}_s) m_{0}(F^{(N)}_{s}) \Big]\Big\|_{\gamma, p} \dd s\\
 &\quad \lesssim \int_{0}^{t}  \frac{1}{(t-s)^{\frac{1}{2}\big(\gamma + d(\frac{1}{2}-\frac{1}{p})\big)}}
 \bigg(\frac{\| u_{s}^{N} \|_{\gamma, p}}{N^{\beta(\gamma-\frac{d}{p})}}  \| \langle v \rangle^{k} F_{s}^{N}\|_{L^2_{x,v}} \\
&\hspace{5cm} +\big\| \langle v \rangle^{k} (F_{s}^{N}- F^{(N)}_{s}) \big\|_{L^{2}_{x,v}}
 + \|  u_{s}^N-u^{(N)}_{s} \|_{\gamma,p} \| m_{0}(F^{(N)}_{s})\|_{L^2_{x}} \bigg) \dd s.
\end{split}
\end{equation}

Plugging the previous estimates \eqref{eq:conv-Bessel-1}, \eqref{eq:conv-Bessel-2}, \eqref{eq:conv-Bessel-3} and \eqref{eq:conv-Bessel-4} into \eqref{estima} reads, for some constant $C \equiv C(A,T, \sup_{N} \|u^{(N)}_{s}\|_{L^\infty_{x}}, \gamma, d,p,k)$,
\begin{equation}
\label{eq:bounduN-u}
\begin{split}
\| u_{t}^{N}-u^{(N)}_{t} \|_{\gamma, p}
&\leq \| u_{0}^{N}-u_{0} \|_{\gamma, p} +  C \int_{0}^{t} \frac{1}{(t-s)^{\frac{\gamma+1}{2}}} \|  u_{s}^N-u^{(N)}_{s} \|_{L^p_{x}}  \dd s\\
&\quad + C\int_{0}^{t}  \frac{1}{(t-s)^{\frac{1}{2}(\gamma + d(\frac{1}{2}-\frac{1}{p}))}}  \big\| \langle v \rangle^{k} (F_{s}^{N}- F^{(N)}_{s})\big\|_{L^{2}_{x,v}} \dd s\\
&\quad +  C\, N^{-\beta(\gamma-\frac{d}{p})} \sup_{s\in [0,T]}\| u_{s}^{N} \|_{\gamma, p}  \sup_{s\in [0,T]}\| \langle v \rangle^{k} F_{s}^{N}\|_{L^2_{x,v}} \\
&\quad +C\, \sup_{s\in [0,T]}\| m_{0}(F^{(N)}_{s})\|_{L^2_{x}} \int_{0}^{t}  \frac{1}{(t-s)^{\frac{1}{2}(\gamma + d(\frac{1}{2}-\frac{1}{p}))}} \|  u_{s}^N-u^{(N)}_{s} \|_{\gamma,p} \dd s.
\end{split}
\end{equation}

\paragraph{Step $2$ -- Bound on $\big(\E  \sup_{s\in [0,t]} \| u_{s}^{N}-u^{(N)}_{s} \|_{\gamma, p}^{2q}\big)^{\frac{1}{q}}$.}

Recall the definition of $\ratetilde$ from~\eqref{eq:def-ratetilde} and of
$\tilde\gamma$ from \eqref{eq:deftildegamma}. Then, using Lemma~\ref{lem:boundsuNFN} and the embedding $\|u_{s}^N-u^{(N)}_{s} \|_{L^p_{x}}\lesssim \|  u_{s}^N-u^{(N)}_{s} \|_{\gamma,p}$, we deduce from \eqref{eq:bounduN-u} that
\begin{align*}
\| u_{t}^{N}-u^{(N)}_{t} \|_{\gamma, p}
&\lesssim \| u_{0}^{N}-u_{0} \|_{\gamma, p}+ \ratetilde + N^{-\beta(\gamma-\frac{d}{p})} \sup_{s\in [0,T]}\| u_{s}^{N} \|_{\gamma, p}  \sup_{s\in [0,T]}\| \langle v \rangle^{k} F_{s}^{N}\|_{L^2_{x,v}} \\
&\quad +  \int_{0}^{t} \frac{1 }{(t-s)^{\tilde\gamma}} \Big( \| u_{s}^N-u^{(N)}_{s} \|_{\gamma,p} + \big\| \langle v \rangle^{k} (F_{s}^{N}-  \tilde{F}_{s}^{N})\big\|_{L^{2}_{x,v}} \Big)\dd s.
\end{align*}
It follows from the same trick used in \eqref{eq:trick-kernel} that
\begin{align*}
\sup_{s\in [0,t]} \| u_{s}^{N}-u^{(N)}_{s} \|_{\gamma, p}
&\lesssim \| u_{0}^{N}-u_{0} \|_{\gamma, p}+ \ratetilde + N^{-\beta(\gamma-\frac{d}{p})} \sup_{s\in [0,T]}\| u_{s}^{N} \|_{\gamma, p} \sup_{s\in [0,T]}\| \langle v \rangle^{k} F_{s}^{N}\|_{L^2_{x,v}} \\
&\quad +  \int_{0}^{t} \frac{1}{(t-s)^{\tilde\gamma}} \Big(\sup_{r\in [0,s]} \|  u_{r}^N-u^{(N)}_{r} \|_{\gamma,p} + \sup_{r\in [0,s]} \big\| \langle v \rangle^{k} (F_{r}^{N}-  \tilde{F}_{r}^{N})\big\|_{L^{2}_{x,v}} \Big)\dd s .
\end{align*}
Hence by Jensen's inequality,
\begin{align*}
\sup_{s\in [0,t]} \| u_{s}^{N}-u^{(N)}_{s} \|_{\gamma, p}^2
&\lesssim \| u_{0}^{N}-u_{0} \|_{\gamma, p}^2+ \ratetilde^2 + N^{-2\beta(\gamma-\frac{d}{p})} \sup_{s\in [0,T]}\| u_{s}^{N} \|_{\gamma, p}^2  \sup_{s\in [0,T]}\| \langle v \rangle^{k} F_{s}^{N}\|_{L^2_{x,v}}^2 \\
&\quad +  t^{1-\tilde\gamma} \int_{0}^{t} \frac{1}{(t-s)^{\tilde\gamma}} \Big(\sup_{r\in [0,s]} \|  u_{r}^N-u^{(N)}_{r} \|_{\gamma,p}^2 + \sup_{r\in [0,s]} \big\| \langle v \rangle^{k} (F_{r}^{N} - \tilde{F}_{r}^{N})\big\|_{L^{2}_{x,v}}^2 \Big)\dd s ,
\end{align*}
and taking the $L^q(\Omega)$ norm,
\begin{align*}
\Big(\E &\sup_{s\in [0,t]} \| u_{s}^{N}-u^{(N)}_{s} \|_{\gamma, p}^{2q}\Big)^{\frac{1}{q}} \\
&\lesssim \big(\E \| u_{0}^{N}-u_{0} \|_{\gamma, p}^{2q}\big)^{\frac{1}{q}} + \ratetilde^2 + N^{-2\beta(\gamma-\frac{d}{p})} \big( \E\sup_{s\in [0,T]}\| u_{s}^{N} \|_{\gamma, p}^{4q}\big)^{\frac{1}{2q}}  \big( \E \sup_{s\in [0,T]}\| \langle v \rangle^{k} F_{s}^{N}\|_{L^2_{x,v}}^{4q}\big)^{\frac{1}{2q}} \\
&\quad +  \int_{0}^{t} \frac{1}{(t-s)^{\tilde\gamma}} \bigg( \Big(\E\sup_{r\in [0,s]} \big\| u_{r}^N-u^{(N)}_{r} \big\|_{\gamma,p}^{2q}\Big)^{\frac{1}{q}}  + \Big(\E \sup_{r\in [0,s]} \big\| \langle v \rangle^{k} (F_{r}^{N}-  \tilde{F}_{r}^{N})\big\|_{L^{2}_{x,v}}^{2q}\Big)^{\frac{1}{q}} \bigg) \dd s .
\end{align*}
By Propositions~\ref{moment} and \ref{propu}, it follows that
\begin{align*}
\big(\E & \sup_{s\in [0,t]} \big\| u_{s}^{N}-u^{(N)}_{s} \big\|_{\gamma, p}^{2q}\big)^{\frac{1}{q}} \\
&\lesssim \big(\E  \big\| u_{0}^{N}-u_{0} \big\|_{\gamma, p}^{2q}\big)^{\frac{1}{q}} + \ratetilde^2 + N^{-2\beta(\gamma-\frac{d}{p})} e^{CT\sigma_{N}^{-2}}\\
&\quad +  \int_{0}^{t} \frac{1}{(t-s)^{\tilde\gamma}} \bigg( \Big(\E\sup_{r\in [0,s]} \big\| u_{r}^N-u^{(N)}_{r} \big\|_{\gamma,p}^{2q}\Big)^{\frac{1}{q}}  + \Big(\E \sup_{r\in [0,s]} \big\| \langle v \rangle^{k} (F_{r}^{N} - \tilde{F}_{r}^{N})\big\|_{L^{2}_{x,v}}^{2q}\Big)^{\frac{1}{q}} \bigg) \dd s .
\end{align*}

\paragraph{Step $3$ -- Conclusion.}

Now we aim to bound $\big(\E  \sup_{s\in [0,t]} \big\| u_{s}^{N}-u^{(N)}_{s} \big\|_{\gamma, p}^{2q}\big)^{\frac{1}{q}} + \big(\E \sup_{s\in[0,t]} \big\| \langle v \rangle^{k} (F_{s}^{N} - \tilde{F}_{s}^{N}) \big\|_{L^{2}_{x,v}}^{2q} \big)^{\frac{1}{q}}$ using Proposition~\ref{prop:momentsFN-tildeFN} and the last bound from the previous step of the proof. It comes
\begin{align*}
\big(\E  &\sup_{s\in [0,t]} \big\| u_{s}^{N}-u^{(N)}_{s} \big\|_{\gamma, p}^{2q}\big)^{\frac{1}{q}} + \big(\E \sup_{s\in[0,t]} \big\| \langle v \rangle^{k} (F_{s}^{N} - \tilde{F}_{s}^{N}) \big\|_{L^{2}_{x,v}}^{2q} \big)^{\frac{1}{q}} \\
&\lesssim \big(\E \big\| u_{0}^{N}-u_{0} \big\|_{\gamma, p}^{2q}\big)^{\frac{1}{q}} + \big(\E \big\| \langle v \rangle^{k} (F_{0}^{N}-\tilde{F}_{0}^N) \big\|_{L^{2}_{x,v}}^{2q}\big)^{\frac{1}{q}} \\
&\quad  + N^{-2(\alpha-\beta)} + \mathfrak{C}_{\varepsilon} N^{-2\beta(\gamma-\frac{d}{p})+\varepsilon} + N^{-(\frac{1}{2} - d\beta -(d+1)\alpha)} + \ratetilde^2 \\
&\quad +  \int_{0}^{t} \frac{1}{(t-s)^{\tilde\gamma}} \Big( \big(\E\sup_{r\in [0,s]} \big\|  u_{r}^N-u^{(N)}_{r} \big\|_{\gamma,p}^{2q}\big)^{\frac{1}{q}}  + \big(\E \sup_{r\in [0,s]} \big\| \langle v \rangle^{k} (F_{r}^{N}-  \tilde{F}_{r}^{N})\big\|_{L^{2}_{x,v}}^{2q}\big)^{\frac{1}{q}} \Big)\dd s\\
 &\quad + (1+\sigma_{N}^{-2}) \int_{0}^t \Big(\E \sup_{r\in[0,s]} \big\| \langle v \rangle^{k} (F_{r}^{N} - \tilde{F}_{r}^{N}) \big\|_{L^{2}_{x,v} }^{2q}\Big)^{\frac{1}{q}} + \Big(\E \sup_{r\in [0,s]} \big\|u^N_r-u^{(N)}_r \big\|_{L^\infty_{x}}^{2q}\Big)^{\frac{1}{q}} \dd s .
 \end{align*}
By the embedding $H^\gamma_{p} \hookrightarrow L^\infty$ and the convolutional-type Gr\"onwall's Lemma \cite[Lemma 7.1.1]{Henry}, we obtain
\begin{align*}%
\Big(\E &\sup_{s\in [0,t]} \big\| u_{s}^{N}-u^{(N)}_{s} \big\|_{\gamma, p}^{2q}\Big)^{\frac{1}{q}} + \Big(\E \sup_{s\in[0,t]} \big\| \langle v \rangle^{k} (F_{s}^{N} - \tilde{F}_{s}^{N}) \big\|_{L^{2}_{x,v}}^{2q} \Big)^{\frac{1}{q}}\\
&\lesssim \bigg( \Big(\E \big\| u_{0}^{N}-u_{0} \big\|_{\gamma, p}^{2q}\Big)^{\frac{1}{q}} +  \Big(\E \big\| \langle v \rangle^{k} (F_{0}^{N}-\vartheta^N\ast F_{0}) \big\|_{L^{2}_{x,v}}^{2q}\Big)^{\frac{1}{q}}\\
&\qquad + N^{-2(\alpha-\beta)} + \mathfrak{C}_{\varepsilon} N^{-2\beta(\gamma-\frac{d}{p})+\varepsilon} + N^{-(\frac{1}{2} - d\beta -(d+1)\alpha)}+ \ratetilde^2 \bigg) \, \frac{1}{1-\tilde\gamma} e^{CT\big(1+\sigma_{N}^{-2} + \sigma_{N}^{-\frac{1}{1-\tilde\gamma}}\big)}.
\end{align*}
To conclude the proof of this theorem:
\begin{itemize}
\item  in the case $\sigma > 0$, it remains to use that $(\sigma_{N})_{N}$ is bounded away from $0$, choose $\varepsilon=0$ (as allowed from the definition of $\mathfrak{C}_{\varepsilon}$ in Proposition~\ref{prop:momentsFN-tildeFN}) and to plug the previous inequality into~\eqref{eq:split-u-uN}.

\item In the case $\sigma=0$, we simply plug the previous inequality into~\eqref{eq:split-u-uN}. \hfill $\square$
\end{itemize}

\subsection{Convergence in energy norm: Proof of Theorem~\ref{th:convenergy}}

We denote $U^{N}\coloneqq u-u^{N}$. Since $A\geq \| u\|_{L^{\infty}([0,T]\times \T^{d})}$, we can replace $u$ by $\bchi(u)$. Thus  $U^{N}$ reads
\begin{align*}
\partial_t U^{N}_t(x)
&= \Delta U^{N}_t(x) + (u^{N}_t(x) \cdot\nabla)[\bchi(u^{N}_t)](x) - (u_t(x)\cdot \nabla)[u_t](x)-\nabla (p_t-p^{N}_t)(x)\\
&\quad -\int_{\R^{d}} \big(\bchi(u_t(x))-v\big)  F_t(x,v) \dd v +\frac{1}{N}\sum_{i=1}^{N} \big(\bchi(u_{t}^{N}(X_{t}^{i,N}))-V_{t}^{i,N}\big) \delta_{X^{i,N}_t}^{N}(x).
\end{align*}
Now, inject the null term $(u_t^N(x) \cdot \nabla)[u_t](x) - (u_t^N(x) \cdot \nabla) [\bchi(u_t)](x)$ in the right-hand-side, to obtain
\begin{align*}
\partial_t U^{N}_t(x)
&= \Delta U^{N}_t(x) + (u^{N}_t(x) \cdot\nabla)[ \bchi(u^{N}_t)-\bchi(u_t)](x) - (U^{N}_t\cdot \nabla)[u_t](x)-\nabla (p_t-p^{N}_t)(x)\\
&-\int_{\R^{d}} \big(\bchi(u_t(x))-v\big)  F_t(x,v) \dd v +\frac{1}{N}\sum_{i=1}^{N} \big(\bchi(u_{t}^{N}(X_{t}^{i,N}))-V_{t}^{i,N}\big) \delta_{X^{i,N}_t}^{N}(x).
\end{align*}
Multiplying by $U^{N}$, integrating in time and space and applying integration-by-parts, we obtain
\begin{align*}
\frac{1}{2} \lVert U^{N}_t \rVert_{L^2_x}^2 &+ \int_{0}^{t} \lVert \nabla U^{N}_s \rVert_{L^2_x}^2 \dd s =
\frac{1}{2}\lVert U_{0}^{N} \rVert_{L^2_x}^2-\int_{0}^{t} \int_{\T^{d}} U^{N}_s(x) (U^{N}_s \cdot \nabla)[u_s](x) \dd x \dd s\\
&+\int_{0}^{t} \int_{\T^{d}} U^{N}_s(x) (u^N_s \cdot \nabla)[ \bchi(u^{N}_s)-\bchi(u_s)](x) \dd x \dd s\\
&-\int_{0}^{t} \int_{\T^{d}} U^{N}_s(x) \int_{\R^{d}} \big(\bchi(u_s(x))-v \big)  F_s(x,v) \dd v  \dd x \dd s\\
&+\int_{0}^{t} \int_{\T^{d}} U^{N}_s(x) \frac{1}{N} \sum_{i=1}^{N} \big(\bchi(u_{s}^{N}(X_{s}^{i,N}))-V_{s}^{i,N} \big) \delta_{X^{i,N}_s}^{N}(x) \dd x \dd s.
\end{align*}
Thus, by using Lemma 5.3 in \cite{Flandoli2}, where in particular we use the assumption~\eqref{hyp:theta2} about the symmetry in variable $v$ of the mollifier $\vartheta^2$, we obtain
\begin{align*}
\frac{1}{2} \lVert U^{N}_t \rVert_{L^2_x}^2 &+ \int_{0}^{t} \lVert \nabla U^{N}_s \rVert_{L^2_x}^2 \dd s  =
\frac{1}{2}\lVert U_{0}^{N} \rVert_{L^2_x}^2 -\int_{0}^{t} \int_{\T^{d}} U^{N}_s(x) (U^{N}_s\cdot \nabla)[ u_s](x) \dd x \dd s\\
&\quad+\int_{0}^{t} \int_{\T^{d}} U^{N}_s(x) (u^N_s \cdot \nabla)[\bchi(u^{N}_s)-\bchi(u_s)](x) \dd x \dd s\\
&\quad+\int_{0}^{t} \int_{\T^{d}} U^{N}_s(x) \left( \frac{1}{N}\sum_{i=1}^{N} \bchi(u_{s}^{N}(X_{s}^{i,N})) \delta_{X^{i,N}_s}^{N}(x)  - \int_{\R^{d}} \bchi(u_s(x)) F_s(x,v) \dd v \right)\dd x \dd s\\
&\quad+\int_{0}^{t} \int_{\T^{d}} U^{N}_s(x) \, m_{1}(F_s-F_{s}^{N})(x) \dd x \dd s\\
&\eqqcolon E_{0} + E_{1} + E_{2}+ E_{3} + E_{4}.
\end{align*}
\paragraph{Bound on $E_{1}$.}
By the properties of the non-linear form and Young's inequality,
\begin{equation}\label{eq:boundE1}
\begin{split}
|E_{1}|
&=\bigg|\int_{0}^{t} \int_{\T^{d}} U^{N}_s(x) (U^{N}_s \cdot \nabla) [u_s](x) \dd x \dd s\bigg|
\leq \int_{0}^{t} \int_{\T^{d}} |U^{N}_s(x)| |\nabla U^{N}_s(x)| |u_s(x)| \dd x \dd s\\
&\leq \int_{0}^{t} \| u_s \|_{L^\infty_x} \int_{\T^{d}} |U^{N}_s(x)| |\nabla U^{N}_s(x)| \dd x \dd s\\
&\leq \frac{1}{4} \int_{0}^{t} \lVert \nabla U^{N}_s \rVert_{L^2_x}^{2} \dd s + \sup_{t\in[0,T]}\| u_t \|_{L^\infty_x}^2\int_{0}^{t}  \lVert U^{N}_s \rVert_{L^2_x}^{2} \dd s.
\end{split}
\end{equation}
\paragraph{Bound on $E_{2}$.}
Using that $\dive u^N=0$ and after an integration-by-parts, we obtain
\[
|E_{2}|
\leq \int_{0}^{t} \int_{\T^{d}} |\nabla U^{N}_s(x)| |u^{N}_s(x)| |\bchi(u^{N}_s(x))-\bchi(u_s(x))| \dd x \dd s.
\]
Finally, using Young's inequality, the triangle inequality, using that $\bchi \equiv \bchi_A$ is bounded by $1+A$ and is Lipschitz continuous, we obtain
\begin{equation}\label{eq:boundE2}
\begin{split}
|E_{2}|
&\leq \frac{1}{4} \int_{0}^{t} \int_{\T^{d}} |\nabla U^{N}_s(x)|^{2} \dd x \dd s + \int_{0}^{t} \int_{\T^{d}}  (|U^{N}_s(x)|+|u_s(x)|)^{2} |\bchi(u^{N}_s(x))-\bchi(u_s(x))|^{2} \dd x \dd s\\
& \leq \frac{1}{4} \int_{0}^{t} \lVert\nabla U^{N}_s\rVert^{2}_{L^2_x} \dd s + 2(2(1+A))^{2} \int_{0}^{t} \lVert U^{N}_s\rVert^{2}_{L^2_x} \dd s + 2 \sup_{t\in[0,T]}\| u_t \|_{L^\infty_x}^2 \int_{0}^{t} \lVert U^{N}_s\rVert^{2}_{L^2_x} \dd s.
\end{split}
\end{equation}
\paragraph{Bound on $E_{3}$.}
We first rewrite,
\begin{align*}
E_{3}&=\int_{0}^{t} \int_{\T^{d}} U^{N}_s(x) \left( \frac{1}{N}\sum_{i=1}^{N}  \bchi\big(u_{s}^{N}(X_{s}^{i,N})\big)\, \delta_{X^{i,N}_s}^{N}(x)  - \int_{\R^{d}} \bchi(u_s(x)) F_s(x,v) \dd v \right)\dd x \dd s\\
&=\int_{0}^{t} \int_{\T^{d}} U^{N}_s(x) \left( \frac{1}{N} \sum_{i=1}^{N} \big(\bchi(u_{s}^{N}(X_{s}^{i,N}))- \bchi(u_{s}^{N}(x))\big)\, \delta_{X^{i,N}_s}^{N}(x) \right) \dd x \dd s\\
&\quad+\int_{0}^{t} \int_{\T^{d}} U^{N}_s(x) \bchi(u_{s}^{N}(x))\, m_{0}(F_{s}^{N}- F_{s})(x) \dd x \dd s\\
&\quad +\int_{0}^{t} \int_{\T^{d}} U^{N}_s(x) \big(\bchi(u^{N}_s(x))-\bchi(u_s(x))\big) \, m_{0}(F_{s})(x) \dd x \dd s\\
&\eqqcolon E_{3,1} +E_{3,2}+E_{3,3}.
\end{align*}
For the first term $E_{3,1}$, using Lemma \ref{lem:rateHolderReg}, Young's inequality and Lemma \ref{lem:japanese}, we have the bound
\begin{equation}\label{eq:boundE3-1}
\begin{split}
|E_{3,1}|
&\leq \int_{0}^{t}  \lVert U^{N}_s \rVert_{L^2_x}^{2} \dd s + C^2 \frac{\sup_{t\in[0,T]}\lVert u^{N}_t \rVert_{\gamma, p}^{2}}{N^{2\beta(\gamma-\frac{d}{p})}} \int_{0}^{t} \lVert m_{0}(F_{s}^{N})\rVert_{L^2_x}^{2} \dd s\\
&\leq \int_{0}^{t}  \lVert U^{N}_s \rVert_{L^2_x}^{2} \dd s + C^3 \frac{\sup_{t\in[0,T]}\lVert u^{N}_t \rVert_{\gamma, p}^{2}}{N^{2\beta(\gamma-\frac{d}{p})}} \int_{0}^{t} \big\lVert \langle v \rangle^k F_{s}^{N}\big\rVert_{L^2_{x,v}}^{2} \dd s,
\end{split}
\end{equation}
for any index $k\geq 3$.
For the term $E_{3,2}$, using that $\bchi \equiv \bchi_A$ is bounded by $1+A$, Young's inequality and Lemma \ref{lem:japanese},
\begin{equation}\label{eq:boundE3-2}
\begin{split}
|E_{3,2}|
&\leq\bigg|\int_{0}^{t} \int_{\T^{d}} U^{N}_s(x) \bchi(u^{N}_s(x))\, m_{0}(F_{s}^{N}- F_{s})(x) \dd x \dd s\bigg|\\
 &\leq (1+A) \int_{0}^{t} \int_{\T^{d}} |U^{N}_s(x)| |m_{0}(F_{s}^{N}- F_{s})(x)| \dd x \dd s\\
&\leq (1+A)^2 \int_{0}^{t} \lVert U^{N}_s \rVert_{L^2_x}^2 + \frac{1}{4} \lVert m_{0}(F_{s}^{N}- F_{s}) \rVert_{L^2_x}^2 \dd s\\
& \leq (1+A)^2\int_{0}^{t} \lVert U^{N}_s \rVert_{L^2_x}^{2} \dd s + \frac{C}{4} \int_{0}^{t} \big\| \langle v \rangle^{k} (F_{s}^{N}- F_{s})\big\|_{L^{2}_{x,v}}^{2} \dd s,
\end{split}
\end{equation}
for any index $k \geq 3$.
Finally, for the term $E_{3,3}$, using that $\bchi$ is Lipschitz continuous, we obtain
\begin{equation*}%
\begin{split}
|E_{3,3}|
&\leq\bigg|\int_{0}^{t} \int_{\T^{d}} U^{N}_s(x) \big(\bchi(u_s(x))-\bchi(u^{N}_s(x))\big)\, m_{0}(F_{s})(x) \dd x \dd s \bigg|\\
&\leq \int_{0}^{t} \int_{\T^{d}} |U^{N}_s(x)|^2\, m_{0} (F_{s})(x) \dd x \dd s\\
&\leq  \int_{0}^{t} \sup_{x\in \T^2} m_{0} (F_{s})(x)\, \lVert U^{N}_s \rVert_{L^2_x}^{2} \dd s.
\end{split}
\end{equation*}
Using the supplementary assumption $\sup_{s\in[0,T]} \lVert m_{0} (F_{s})\rVert_{L^\infty_{x}} <\infty$, it comes
\begin{equation}\label{eq:boundE3-3}
|E_{3,3}|
\leq \Big(\sup_{s\in[0,T]} \lVert m_{0} (F_{s})\rVert_{L^\infty_{x}} \Big)\int_{0}^{t} \lVert U^{N}_s \rVert_{L^2_x}^{2} \dd s.
\end{equation}
\paragraph{Bound on $E_{4}$.}
By Young's inequality and Lemma~\ref{lem:japanese}, we obtain for any index $k\geq 3$ that
\begin{equation}\label{eq:boundE4}
\begin{split}
|E_4|
&\leq  \int_{0}^{t}  \lVert U^{N}_s \rVert_{L^2_x}^{2} \dd s +\frac{1}{4} \int_{0}^{t} \big\lVert m_{1}(F_{s}^{N}- F_{s}) \big\rVert_{L^2_x}^{2} \dd s\\
&\leq \int_{0}^{t} \lVert U^{N}_s \rVert_{L^2_x}^{2} \dd s + \frac{C}{4} \int_0^t \big\lVert \langle v \rangle^{k} (F_{s}^{N}- F_{s}) \big\rVert_{L^{2}_{x,v}}^{2} \dd s.
\end{split}
\end{equation}
From the estimates \eqref{eq:boundE1}, \eqref{eq:boundE2}, \eqref{eq:boundE3-1}, \eqref{eq:boundE3-2}, \eqref{eq:boundE3-3}, \eqref{eq:boundE4},
and taking the supremum on the time interval $[0,t]$, we obtain, for any $t\in [0,T]$,
\begin{align*}
\sup_{s\in[0,t]} \lVert U^{N}_s \rVert_{L^2_x}^2 &+ \int_{0}^{t} \lVert \nabla U^{N}_s \rVert_{L^2_x}^2 \dd s \\
&\leq \lVert U_{0}^{N} \rVert_{L^2_x}^2+ C \int_{0}^{t} \sup_{r\in[0,s]}\lVert U^{N}_r \rVert_{L^2_x}^{2} \dd s
+ \frac{CT}{2}\sup_{t\in[0,T]} \big\lVert \langle v \rangle^{k} (F_{s}^{N}- F_{s}) \big\rVert_{L^{2}_{x,v}}^{2} \\
&\quad + \frac{CT}{N^{2\beta(\gamma-\frac{d}{p})}} \sup_{t\in[0,T]} \lVert u^{N}_t \rVert_{\gamma, p}^{2}
\sup_{t\in[0,T]} \big\lVert \langle v \rangle^{k} F_{t}^{N} \big\rVert_{L^{2}_{x,v}}^{2}.
\end{align*}
Applying the $L^q(\Omega)$ norm in the previous inequality, and using Proposition~\ref{moment} and Proposition~\ref{propu} on the last term, it comes for some $C=C(T,A,k,\gamma, p, d)>0$ that
\begin{align*}
&\Big(\E \sup_{s\in[0,t]} \lVert U^{N}_s \rVert_{L^2_x}^{2q}\Big)^\frac{1}{q} + \bigg(\E \bigg| \int_{0}^{t} \lVert \nabla U^{N}_s \rVert_{L^2_x}^2 \dd s \bigg|^q \bigg)^\frac{1}{q} \\
&\leq
2 \big(\E\lVert U_{0}^{N} \rVert_{L^2_x}^{2q}\big)^\frac{1}{q}
+ C \int_{0}^{t} \Big(\E \sup_{r\in[0,s]} \lVert U^{N}_r \rVert_{L^2_x}^{2q}\Big)^\frac{1}{q} \dd s
+C \Big(\E \sup_{t\in[0,T]} \big\lVert \langle v \rangle^{k} (F_{s}^{N}- F_{s}) \big\rVert_{L^{2}_{x,v}}^{2q}\Big)^\frac{1}{q}
+ \frac{C}{N^{2\beta(\gamma-\frac{d}{p})}} e^{CT\sigma_N^{-2}}.
\end{align*}
In view of Theorem~\ref{th:convBessel}, we get
\begin{align}
\label{eq:boundE}
\Big(\E \sup_{s\in[0,t]}& \lVert U^{N}_s \rVert_{L^2_x}^{2q}\Big)^\frac{1}{q} + \bigg(\E \bigg| \int_{0}^{t} \lVert \nabla U^{N}_s \rVert_{L^2_x}^2 \dd s \bigg|^q \bigg)^\frac{1}{q} \nonumber\\
&\leq 2 \big(\E\lVert U_{0}^{N} \rVert_{L^2_x}^{2q}\big)^\frac{1}{q}
+ C \int_{0}^{t} \Big(\E \sup_{r\in[0,s]} \lVert U^{N}_r \rVert_{L^2_x}^{2q}\Big)^\frac{1}{q} \dd s \nonumber\\
&\quad +\bigg( \big(\E \| u_{0}^{N}-u_{0} \|_{\gamma, p}^{2q}\big)^{\frac{1}{q}} +   \big(\E \big\| \langle v \rangle^{k} (F_{0}^{N}-\vartheta^N\ast F_{0}) \big\|_{L^{2}_{x,v}}^{2q}\big)^{\frac{1}{q}}\\
&\quad + N^{-2(\alpha-\beta)} + \mathfrak{C}_{\varepsilon} N^{-2\beta(\gamma-\frac{d}{p})+\varepsilon} + N^{-(\frac{1}{2} - d\beta -(d+1)\alpha)}+ \ratetilde^2 \bigg) \, \frac{1}{1-\tilde\gamma} e^{CT\big(1+\sigma_{N}^{-2} + \sigma_{N}^{-\frac{1}{1-\tilde\gamma}}\big)} + 2\rate^2 .\nonumber
\end{align}
By Gr\"onwall's Lemma and using that $\big(\E\lVert U_{0}^{N} \rVert_{L^2_x}^{2q}\big)^\frac{1}{q} \lesssim \big(\E \| u_{0}^{N}-u_{0} \|_{\gamma, p}^{2q}\big)^{\frac{1}{q}}$, which comes from the embedding $H^\gamma_{p} \hookrightarrow L^2_{x}$ on the torus (recall that $\gamma>0$ and $p>d\geq 2$), we conclude that
\begin{equation}
\label{eq:boundEU}
\begin{split}
\Big(&\E \sup_{s\in[0,t]} \lVert U^{N}_s \rVert_{L^2_x}^{2q}\Big)^\frac{1}{q}
\leq C \bigg( \big(\E \| u_{0}^{N}-u_{0} \|_{\gamma, p}^{2q}\big)^{\frac{1}{q}} +   \big(\E \big\| \langle v \rangle^{k} (F_{0}^{N}-\vartheta^N\ast F_{0}) \big\|_{L^{2}_{x,v}}^{2q}\big)^{\frac{1}{q}}\\
& + N^{-2(\alpha-\beta)} + \mathfrak{C}_{\varepsilon} N^{-2\beta(\gamma-\frac{d}{p})+\varepsilon} + N^{-(\frac{1}{2} - d\beta -(d+1)\alpha)}+ \ratetilde^2 \bigg) \, \frac{1}{1-\tilde\gamma} e^{CT\big(1+\sigma_{N}^{-2} + \sigma_{N}^{-\frac{1}{1-\tilde\gamma}}\big)} + C \rate^2.
\end{split}
\end{equation}
Finally, injecting the bound \eqref{eq:boundEU} into the right-hand side of~\eqref{eq:boundE} finishes the proof.
\hfill $\square$

\subsection{Strong propagation of chaos: Proof of Theorem~\ref{th:strongPoC}}
\label{subsec:prooflastcorollary}

Recall that for $A\geq \lVert u \rVert_{L^\infty_{t,x}}$ we can write $u_{t}(x) = \bchi(u_{t}(x))$. By the triangle inequality, Jensen's inequality and the Lipschitz continuity of $\bchi$, we get
\begin{align*}
\big|X_{t}^{i,N} -\overline{X}_{t}^i\big|^q + \big|V_{t}^{i,N} - \overline{V}_{t}^i\big|^q
& \lesssim |\sigma-\sigma_{N}|^q |B_{t}|^q + \int_{0}^t \big|V_{s}^{i,N} - \overline{V}_{s}^i\big|^q \dd s + \int_{0}^t \big|\bchi\big(u_{s}^N(X^{i,N}_{s})\big) - \bchi\big(u_{s}(\overline{X}^{i}_{s})\big) \big|^q \dd s\\
& \lesssim |\sigma-\sigma_{N}|^q |B_{t}|^q + \int_{0}^t \big|V_{s}^{i,N} - \overline{V}_{s}^i\big|^q \dd s \\
&\quad + \int_{0}^t \big|\big(u_{s}^N-u_{s}\big)(X^{i,N}_{s}) \big|^q \dd s + \int_{0}^t \big|u_{s}(X^{i,N}_{s}) - u_{s}(\overline{X}^{i}_{s}) \big|^q \dd s .
\end{align*}
For the penultimate integral, use the embedding $H^\gamma_{p} \hookrightarrow L^\infty$ when $\gamma>d/p$; and for the last integral, use the Lipschitz regularity of $u$ given by Lemma~\ref{prop:CoBessel-u}\ref{item:reg-nablau} to deduce
\begin{align*}
\sup_{s\in [0,t]} \big|X_{s}^{i,N} -\overline{X}_{s}^i\big|^q + \big|V_{s}^{i,N} - \overline{V}_{s}^i\big|^q
& \lesssim |\sigma-\sigma_{N}|^q \sup_{s\in [0,t]}|B_{s}|^q + \int_{0}^t \big|V_{s}^{i,N} - \overline{V}_{s}^i\big|^q \dd s \\
&\quad + \sup_{s\in [0,T]} \lVert u_{s}^N-u_{s} \rVert_{\gamma,p}^q + \int_{0}^t s^{-\frac{1}{2}} \big|X^{i,N}_{s} - \overline{X}^{i}_{s} \big|^q \dd s .
\end{align*}
Taking the expectation in the previous inequality and applying Gr\"onwall's Lemma yields the result.

\appendix

\section{Well-posedness of the fluid-particle system: Proof of Proposition~\ref{prop:existence-IPS}}
\label{app:proof-prop-IPS}

In this section, we prove Proposition~\ref{prop:existence-IPS}: by a compactness argument developed in Section~\ref{app-subsec:existence}, we will obtain, for fixed $\omega\in \Omega$, the existence of a solution $(u^N(\omega), X^{1,N}(\omega), V^{1,N}(\omega),\dots, X^{N,N}(\omega), V^{N,N}(\omega))$, in a deterministic fashion close to what is done in~\cite{Boudin2009}. However it is essential in this work to obtain measurability in $\omega$ of the solution. This is achieved by a uniqueness argument proved in Section~\ref{app-subsec:uniqueness}; hence we first prove uniqueness of an adequate deterministic system, see~\eqref{eq:mildPDE-ODE} with generic drivers $b^1,\dots,b^N$ in place of the Brownian motions.

\subsection{Uniqueness}
\label{app-subsec:uniqueness}

We show uniqueness for a coupled mild PDE-ODE system described as follows. For fixed $N\in \N^*$ and given continuous trajectories $b^1,\dots, b^N$ in $\mathcal{C}([0,T];\R^d)$, consider the system of equations:
\begin{equation}
\label{eq:mildPDE-ODE}
\begin{cases}
&\mathtt{u}_{t}^{N} = e^{t  \Delta}u_{0}^{N} - \int_{0}^{t} e^{\left(t-s\right) \Delta} P\big[ (\mathtt{u}^N_{s} \cdot \nabla)\bchi(\mathtt{u}^N_s)\big] \dd s  - \int_{0}^{t} e^{\left(t-s\right)  \Delta} P\big[ \frac{1}{N}\sum_{i=1}^{N} ( \bchi(\mathtt{u}_{s}^{N}(x_{s}^{i,N}))-v_{s}^{i,N}) \, \delta_{x^{i,N}_s}^{N} \big] \dd s \\
& x^{i,N}_t= x^i_{0} + \int_{0}^t v^{i,N}_s\dd s,\\
& v^{i,N}_t= v^i_{0} +
\int_{0}^t \big(\bchi_A\big(\mathtt{u}_{s}^N(x_{s}^{i,N})\big)-v_{s}^{i,N}\big) \dd s + \sigma_{N} b_{t}^{i}\, .
\end{cases}
\end{equation}
Assume that there are two solutions $(\mathtt{u}_{t}^{N}, x^{1,N},v^{1,N},\dots, x^{N,N},v^{N,N})$ and $(\widetilde{\mathtt{u}}_{t}^{N}, \widetilde{x}^{1,N}, \widetilde{v}^{1,N},\dots, \widetilde{x}^{N,N}, \widetilde{v}^{N,N})$ in $L^\infty([0,T]; H^\gamma_p(\T^d)^d) \times \mathcal{C}([0,T];\R^{2d})^N$.
Doing calculations similar to Proposition~\ref{prop:CoBessel-u}\ref{item:reg-nablau},  we have
 $\sup_{t\in [0,T]} \| \mathtt{u}^{N}_{t} \|_{\gamma, p} < \infty$ and $\sup_{t\in [0,T]} t^{1/2} \| \nabla \mathtt{u}^{N}_{t} \|_{\gamma, p} < \infty$.
 A direct comparison and simple computations yield
 \begin{align*}
 \| \mathtt{u}^N_{t} - \widetilde{\mathtt{u}}^N_{t} \|_{\gamma, p}
 &\leq C \int_{0}^{t} \frac{1}{(t-s)^{(1+\gamma) /2}}  (C_{A} + \|\mathtt{u}_{s}^{N}\|_{\gamma, p} ) \| \mathtt{u}_{s}^{N}-\widetilde{\mathtt{u}}_{s}^{N} \|_{\gamma, p} \dd s \\
 &\quad + C \int_{0}^{t} \frac{1}{(t-s)^{\gamma/2}}  (C_{A, N} + \| \nabla \mathtt{u}_{s}^{N}\|_{\gamma, p} +  \max_{i} |v_{s}^{i,N}|)\\
 &\hspace{2cm} \times \Big( \| \mathtt{u}_{s}^{N}- \widetilde{\mathtt{u}}_{s}^{N} \|_{\gamma, p} + \frac{1}{N} \sum_{i=1}^N |v_{s}^{i,N}-\widetilde{v}_{s}^{i,N}| + |x_{s}^{i,N}-\widetilde{x}_{s}^{i,N}| \Big) \dd s ,
 \end{align*}
and for all $i\in \{1,\dots, N\}$,
 \begin{align*}
 &| x_{t}^{i,N}-\widetilde{x}_{t}^{i,N} | \leq \int_{0}^{t}  |v_{s}^{i,N}-\widetilde{v}_{s}^{i,N}| \dd s,  \\
 & | v_{t}^{i,N}- \widetilde{v}_{t}^{i,N} | \leq \int_{0}^{t} \| \nabla \mathtt{u}_{s}^{N} \|_{\gamma, p}\, |x^{i,N}_{s} - \widetilde{x}^{i,N}_{s} | + \| \mathtt{u}_{s}^{N}-\widetilde{\mathtt{u}}_{s}^{N} \|_{\gamma, p} + |v_{s}^{i,N}-\widetilde{v}_{s}^{i,N}|  \dd s.
 \end{align*}
Since $\sup_{t\in [0,T]} \max_{i} |v_{s}^{i,N}| <\infty$, we deduce from the previous inequalities and from Gr\"onwall's Lemma that there is uniqueness in $L^\infty([0,T]; H^\gamma_p(\T^d)^d) \times \mathcal{C}([0,T];\R^{2d})^N$.

\subsection{Existence}
\label{app-subsec:existence}

Let $N\in \N^*$ and let $\Omega_{N}\in \mathcal{F}$ be such that $\PP(\Omega_{N})=1$ and for all $\omega \in \Omega_{N}$, $(B^1(\omega),\dots, B^N(\omega))$ is in $\mathcal{C}^{1/2-\eta}([0,T];\R^d)^N$ for some $\eta\in (0,1/2)$. Until the last step of this proof, we work with some fixed (arbitrary) $\omega\in \Omega_{N}$.
For this $\omega$, we now show the existence of a solution $(u^{N}, X^{1,N}, V^{1,N},\dots, X^{N,N}, V^{N,N})$ to \eqref{eq:mildPDE-ODE} in $L^\infty([0,T]; H^\gamma_p(\T^d)^d) \times \mathcal{C}([0,T]; \T^d\times \R^{d} )^{N}$. The proof follows the approach used by~\cite{Boudin2009} in the analysis of the Vlasov--Navier--Stokes system.
To this end, for each $ \epsilon >0$, we introduce a regularised system
\begin{equation}\label{PartEbi}
\begin{cases}
& \displaystyle \partial_t u_{t}^{\epsilon}(x) - \Delta u_{t}^{\epsilon}(x) + ((u_{t}^{\epsilon}\ast \rho^\epsilon)    \cdot \nabla)  \left[\bchi(u_{t}^{\epsilon})\right](x) + \nabla p^\epsilon_{t}(x)  \\
&  \hspace{2cm}\displaystyle +\frac{1}{N} \sum_{i=1}^{N} \big(\bchi\big(u_{t}^{\epsilon} \ast \rho^\epsilon(X_{t}^{i,\epsilon})\big)-V_{t}^{i,\epsilon}\big)\, \delta_{X^{i,\epsilon}_t}^{N}(x)=0, \quad t>0,~ x\in \T^{d},\\
& \dive u_t^{\epsilon}(x) =0 , \quad t>0,~ x\in \T^{d}, \\
& u^{\epsilon}_{0}(x) = u^{\circ,N}(x), \quad x\in\T^{d},\\
& \dd X^{i,\epsilon}_t= V^{i,\epsilon}_t \dd t, \quad t>0,~ i\in \{1,\dots, N\}, \\
& \dd V^{i,\epsilon}_t=   \big(\bchi\big(u_{t}^{\epsilon} \ast \rho^\epsilon(X_{t}^{i,\epsilon})\big)-V_{t}^{i,\epsilon}\big) \dd t + \sigma_N \dd B_{t}^{i}, \quad t>0,~ i\in \{1,\dots, N\},
\end{cases}
\end{equation}
where $\rho^\epsilon$ are standard mollifiers on $\T^d$.
By the same argument that led to uniqueness for \eqref{eq:mildPDE-ODE} in the previous paragraph, we know that there is uniqueness for the mild PDE-ODE system~\eqref{PartEbi} in $L^\infty([0,T]; H^\gamma_p(\T^d)^d) \times \mathcal{C}([0,T];\T^d\times \R^{d} )^{N}$.
Now we prove the existence of a solution to the regularised system using an iterative scheme. Subsequently, we establish uniform estimates independent of $\epsilon$ that allow us to pass to the limit and thereby conclude the existence of solutions to the original problem.
\\

\textbf{Iterative scheme.}
Recall that $N$ is fixed and set $\epsilon>0$. Without loss of generality, we take $\sigma_N=1$ in the rest of this proof. For any $k\in \N^*$, we will construct
$(u^{N,\epsilon,k}, X^{1,N,\epsilon,k}, V^{1,N,\epsilon,k},\dots, X^{N,N,\epsilon,k}, V^{N,N,\epsilon,k})$, which for simplicity we rather denote $(\bar{u}^{k}, \bar{X}^{1,k}, \bar{V}^{1,k},\dots, \bar{X}^{N,k}, \bar{V}^{N,k})$, that is solution of the following uncoupled linear system:
\begin{equation}\label{PartE}
\begin{cases}
& \displaystyle \partial_t \bar{u}_{t}^k(x) - \Delta \bar{u}_{t}^k(x) + \big((\bar{u}_{t}^{k-1}\ast \rho^\epsilon) \cdot \nabla\big) \big[\bchi(\bar{u}_{t}^{k-1})\big](x) + \nabla p^{k}_{t}(x)   \\
& \hspace{2cm}\displaystyle
+ \frac{1}{N} \sum_{i=1}^{N} \big(\bchi\big(\bar{u}_{t}^{k-1} \ast \rho^\epsilon(\bar{X}_{t}^{i,k-1})\big) -\bar{V}_{t}^{i,k-1}\big)\, \delta_{\bar{X}^{i,k-1}_t}^{N}(x)=0 , \quad t>0,~ x\in \T^{d},\\
& \dive \bar{u}_t^{k}(x) =0 , \quad t>0,~ x\in \T^{d}, \\
& \bar{u}^{k}_{0}(x) = u^{\circ,N}(x), \quad x\in\T^{d},\\
& \dd \bar{X}^{i,k}_t= \bar{V}^{i,k-1}_t \dd t, \quad t>0,~ i\in \{1,\dots, N\},\\
& \dd \bar{V}^{i,k}_t= \big(\bchi\big(\bar{u}_{t}^{k-1} \ast \rho^\epsilon(\bar{X}_{t}^{i,k-1})\big)-\bar{V}_{t}^{i,k-1}\big) \dd t +  \dd B_{t}^{i}, \quad t>0,~ i\in \{1,\dots, N\},
\end{cases}
\end{equation}
in the space  $\mathcal{Y}= \big(L^{\infty}([0,T]; L^{2}(\T^d)^d) \cap  L^{2}([0,T]; H^{1}(\T^d)^d)\big) \times \mathcal{C}([0,T];\T^d\times \R^{d} )^N$,
with initial positions and velocities $(X^{i,N}_{0},V^{i,N}_{0})$.

The first equation is a Navier--Stokes equation with a forcing term belonging to $ L^{2}([0,T]; L^{2}(\T^d)^d)$ and, as the drift term is Lipschitz continuous, the iterative scheme is well-defined from $\mathcal{Y}$ into $\mathcal{Y}$.

First,  we obtain uniform bounds in $k$ for  the sequence $(\bar{u}^{k}, \bar{X}^{1,k}, \bar{V}^{1,k},\dots, \bar{X}^{N,k}, \bar{V}^{N,k})$. We have
\begin{align}
\label{bouinParti}
\frac{1}{2}\| \bar{u}^{k}_t\|_{L^2_x}^{2} +  \int_{0}^{t} \| \nabla \bar{u}^{k}_s\|_{L^2_x}^{2}  \dd s
& =  \frac{1}{2}  \| u^{\circ,N}\|_{L^2_x}^{2}
+ \int_{0}^{t}  \int_{\T^d}  \nabla \bar{u}^{k}_{s} \,  (\bar{u}^{k-1}_{s}\ast \rho^\epsilon)   \bchi(\bar{u}^{k-1}_{s}) \dd x \dd s \nonumber\\
& + \int_{0}^{t}  \int_{\T^d} \bar{u}^{k}_{s}  \,
\frac{1}{N} \sum_{i=1}^{N} \big(\bchi\big(\bar{u}^{k-1}_{s} \ast \rho^\epsilon(\bar{X}^{i,k-1}_{s})\big)-\bar{V}^{i,k-1}_{s}\big) \delta_{\bar{X}^{i,k-1}_s}^{N}(x) \dd x \dd s \nonumber\\
& \eqqcolon \frac{1}{2}\| u^{\circ,N}\|_{L^2_x}^{2} + I_{1} + I_{2}.
\end{align}
For the first term, by using the H\"older and Young inequalities, it follows that
\begin{equation}\label{I1P}
|I_{1}| \leq \frac{1}{2} \int_{0}^{t} \| \nabla \bar{u}^{k}_s\|_{L^2_x}^{2} \dd s + C_{A} \int_{0}^{t} \| \bar{u}^{k-1}_{s}\|_{L^2_x}^{2} \dd s.
\end{equation}
Similarly,
\[
|I_{2}| \leq  \int_{0}^{t} \| \bar{u}^{k}_s\|_{L^2_x}^{2} \dd s +
C \int_{0}^{t}
\frac{1}{N} \sum_{i=1}^{N} \big|\bchi\big(\bar{u}^{k-1}_{s} \ast  \rho^\epsilon(\bar{X}^{i,k-1}_{s})\big)-\bar{V}^{i,k-1}_{s} \big|^{2}  \,\| \delta_{\bar{X}^{i,k-1}_s}^{N}\|_{L^2_x}^{2}  \dd s.
\]
We observe that  $  \| \delta_{\bar{X}^{i,k-1}_s}^{N}\|_{L^2_x}^{2}$
is a constant which depends only on $N$. Thus, we have
\begin{equation}\label{I2P}
|I_{2}| \leq  \int_{0}^{t} \| \bar{u}^{k}_s\|_{L^2_x}^{2} \  \dd s +
C_{N} \int_{0}^{t} ( C_{A}  +  \sup_{r\in [0,s]} |\bar{V}^{i,k-1}_{r}|^{2} )   \  \dd s.
\end{equation}
From \eqref{bouinParti}, \eqref{I1P} and \eqref{I2P},
we deduce that
\begin{align*}
\frac{1}{2}\| \bar{u}^{k}_t\|_{L^2_x}^{2}  & +  \frac{1}{2} \int_{0}^{t} \| \nabla \bar{u}^{k}_s\|_{L^2_x}^{2}  \dd s \\
& \leq    \frac{1}{2} \| u^{\circ,N}\|_{L^2_x}^{2} + C_{A} \int_{0}^{t} \| \bar{u}^{k-1}_{s}\|_{L^2_x}^{2} \dd s +
 \int_{0}^{t} \| \bar{u}^{k}_s\|_{L^2_x}^{2} \dd s +
C_{N} \int_{0}^{t}  (C_{A}  +  \sup_{r\in [0,s]}|\bar{V}^{i,k-1}_{r}|^{2})  \dd s.
\end{align*}
On the other hand,
\begin{align*}
 &\sup_{r\in [0,t]} |\bar{X}^{i,k}_{r}|^{2} +  \sup_{r\in [0,t]}\  |\bar{V}^{i,k}_{r}|^{2} \\
 &\quad \leq  C_T \Big( |X_0^{i,N}|^2 +  |V_0^{i,N}|^2 +
 \int_{0}^{t} |\bar{V}^{i,k-1}_{s}|^{2} \dd s +
 \int_{0}^{t}  (C_{A} + \sup_{r\in [0,s]}| \bar{V}^{i,k-1}_{r}|^{2}) \dd s + \sup_{r\in [0,t]} |B_{r}^{i}|^2 \Big) .
\end{align*}
Let
\[a_{k}(t)= \| \bar{u}^{k}_t\|_{L^2_x}^{2} %
+  \sup_{r\in [0,t]} \max_{i\in \{1,\dots, N\}} |\bar{X}^{i,k}_{r}|^{2} +  \sup_{r\in [0,t]} \max_{i\in \{1,\dots, N\}}  |\bar{V}^{i,k}_{r}|^{2}.\]
In view of the previous inequalities, by Gr\"onwall's Lemma for sequences, see \emph{e.g.}~\cite[Lemma 3]{Boudin2009}, we deduce that there exists $K>0$ (which depends on $\omega$) such that
\[
a_{k}(t) \leq K \exp(Kt)
\]
for all $k$ and $0\leq t\leq T$. Thus, $(\bar{u}^{k}, \bar{X}^{1,k}, \bar{V}^{1,k},\dots, \bar{X}^{N,k}, \bar{V}^{N,k})_{k\in \N^*}$ is uniformly bounded in $\mathcal{Y}$.

Next,  we will show that the sequence
$(\bar{u}^{k}, \bar{X}^{1,k}, \bar{V}^{1,k},\dots, \bar{X}^{N,k}, \bar{V}^{N,k})_{k\in \N^*}$ is  a Cauchy sequence in $\mathcal{Y}$. Let  $U^{k}_t=\bar{u}^{k+1}_{t}-\bar{u}^{k}_{t}$, then we have
\begin{align}
\label{J}
\frac{1}{2}\| U^{k}_t\|_{L^2_x}^{2} +  \int_{0}^{t} \| \nabla U^{k}_s\|_{L^2_x}^{2}  \dd s
& = \bigg( \int_{0}^{t}  \int_{\T^d}  \nabla U_{s}^{k} \, (\bar{u}^{k}_{s}\ast \rho^\epsilon)\, \bchi(\bar{u}^{k}_{s}) \dd x \dd s \nonumber \\
&\quad - \int_{0}^{t}  \int_{\T^d}  \nabla U_{s}^{k} \, (\bar{u}^{k-1}_{s}\ast \rho^\epsilon)\,\bchi(\bar{u}^{k-1}_{s}) \dd x \dd s \bigg) \nonumber\\
&\quad  + \bigg(\int_{0}^{t}  \int_{\T^d} U_{s}^{k}  \,
\frac{1}{N} \sum_{i=1}^{N} \big(\bchi\big(\bar{u}^{k}_{s} \ast \rho^\epsilon(\bar{X}^{i,k}_{s})\big)-\bar{V}^{i,k}_{s}\big) \delta_{\bar{X}^{i,k}_{s}}^{N}(x) \dd x \dd s \nonumber\\
&\quad - \int_{0}^{t}  \int_{\T^d} U_{s}^{k}  \,
\frac{1}{N} \sum_{i=1}^{N} \big(\bchi\big(\bar{u}^{k-1}_{s} \ast \rho^\epsilon(\bar{X}^{i,k-1}_{s})\big)-\bar{V}^{i,k-1}_{s} \big) \delta_{\bar{X}^{i,k-1}_s}^{N}(x) \dd x \dd s \bigg)\nonumber\\
&\eqqcolon J^{1} + J^{2}.
\end{align}
Adding and subtracting the term
$ \int_{0}^{t}  \int_{\T^d}  \nabla U_{s}^{k} \, (\bar{u}^{k}_{s}\ast \rho^\epsilon)\, \bchi(\bar{u}^{k-1}_{s}) \dd x \dd s$, we can rewrite
\begin{align}
\label{J1}
J^{1}
&=\int_{0}^{t}  \int_{\T^d}  \nabla U_{s}^{k} \, (\bar{u}^{k}_{s}\ast \rho^\epsilon)\, \big( \bchi(\bar{u}^{k}_{s})-\bchi(\bar{u}^{k-1}_{s}) \big) \dd x \dd s - \int_{0}^{t}  \int_{\T^d}  \nabla U_{s}^{k} \, (U_{s}^{k-1}\ast \rho^\epsilon)\, \bchi(\bar{u}^{k-1}_{s})  \dd x \dd s \nonumber\\
&\eqqcolon J^{1,1}  + J^{1,2}.
\end{align}
By using the H\"older and Young inequalities, it comes
\begin{equation}\label{J12}
J^{1,2} \leq \frac{1}{4} \int_{0}^{t} \| \nabla U^{k}_s\|_{L^2_x}^{2} \dd s + C_{A} \int_{0}^{t} \| U^{k-1}_s\|_{L^2_x}^{2} \dd s.
\end{equation}
Similarly, and using that $\bchi$ is Lipschitz continuous,
\begin{align}
\label{J11}
J^{1,1}
&\leq \frac{1}{4} \int_{0}^{t} \| \nabla U^{k}_s\|_{L^2_x}^{2} \dd s + C_{A} \int_{0}^{t} \big\|(\bar{u}^{k}_{s}\ast \rho^\epsilon) U^{k-1}_s \big\|_{L^2_x}^{2} \dd s \nonumber\\
& \leq \frac{1}{4} \int_{0}^{t} \| \nabla U^{k}_s\|_{L^2_x}^{2} \dd s + C_{A} \int_{0}^{t} \| \bar{u}^{k}_{s}\ast \rho^\epsilon \|_{L^\infty_{x}}^{2}  \|U^{k-1}_s\|_{L^2_x}^{2} \dd s \nonumber\\
&\leq \frac{1}{4} \int_{0}^{t} \| \nabla U^{k}_s\|_{L^2_x}^{2} \dd s + C_{A, \epsilon} \int_{0}^{t} \|\bar{u}^{k}_{s}\|_{L^2_x}^{2} \|U^{k-1}_s\|_{L^2_x}^{2} \dd s.
 \end{align}
Now for $J_{2}$, adding and subtracting the term $\int_{0}^{t} \int_{\T^d} U_{s}^{k}\,
\frac{1}{N} \sum_{i=1}^{N} \big(\bchi\big(\bar{u}^{k}_{s} \ast \rho^\epsilon(\bar{X}^{i,k}_{s})\big)-\bar{V}^{i,k}_{s}\big) \delta_{\bar{X}^{i,k-1}_s}^{N}(x) \dd x \dd s$, we rewrite
\begin{align}
\label{J2}
J^{2}
&=  \int_{0}^{t}  \int_{\T^d} U_{s}^{k}\, \frac{1}{N} \sum_{i=1}^{N} \big(\bchi\big(\bar{u}^{k}_{s} \ast \rho^\epsilon(\bar{X}^{i,k}_{s})\big)-\bar{V}^{i,k}_{s}\big) (\delta_{\bar{X}^{i,k}_{s}}^{N}(x)- \delta_{\bar{X}^{i,k-1}_s}^{N}(x)) \dd x \dd s \nonumber\\
&\quad - \int_{0}^{t}  \int_{\T^d} U_{s}^{k}\, \frac{1}{N} \sum_{i=1}^{N} \big(\bchi\big(\bar{u}^{k-1}_{s} \ast \rho^\epsilon(\bar{X}^{i,k-1}_{s})\big)-\bar{V}^{i,k-1}_{s}
-\bchi\big(\bar{u}^{k}_{s} \ast \rho^\epsilon(\bar{X}^{i,k}_{s})\big)+\bar{V}^{i,k}_{s}\big) \delta_{\bar{X}^{i,k-1}_s}^{N}(x) \dd x \dd s \nonumber\\
&\eqqcolon J^{2,1} + J^{2,2}.
\end{align}
Proceeding as before and using that $\delta^N $ is Lipschitz continuous,
\begin{equation}\label{J21}
|J^{2,1}|  \leq \int_{0}^{t} \| U^{k}_s\|_{L^2_x}^{2} \dd s
+ C_{N} \int_{0}^{t}  (C_{A} +  \sup_{r\in [0,s]} | \bar{V}^{i,k}_{r}|^{2})  \sup_{r\in [0,s]} |\bar{X}_{r}^{i,k}- \bar{X}_{r}^{i,k-1}|^{2} \dd s.
\end{equation}
Adding and subtracting the term  $\bchi\big(\bar{u}^{k-1}_{s} \ast \rho^\epsilon(\bar{X}^{i,k}_{s})\big)$ and using that $\bchi\big(\bar{u}^{k-1}_{s} \ast \rho^\epsilon\big)$ and $\bchi$ are Lipschitz continuous, we deduce that
\begin{equation}
\label{J22}
\begin{split}
|J^{2,2}|
 &\leq \int_{0}^{t} \| U^{k}_s\|_{L^2_x}^{2} \dd s
+ C_{N} \int_{0}^{t}  | \bar{V}^{i,k}_{s}- \bar{V}^{i,k-1}_{s} |^{2} \dd s\\
&\quad + C_{N, A, \epsilon} \int_{0}^{t}  | \bar{X}^{i,k}_{s}- \bar{X}^{i,k-1}_{s} |^{2} \dd s
+ C_{N, A,  \epsilon} \int_{0}^{t}  \| U^{k-1}_s \|_{L^2_x}^{2} \dd x \dd s.
\end{split}
\end{equation}
From \eqref{J}, \eqref{J1}, \eqref{J12}, \eqref{J11}, \eqref{J2}, \eqref{J21}, \eqref{J22}, and using the uniform bound previously obtained, we arrive at
\begin{equation}
\label{U}
\begin{split}
\frac{1}{2}\| U^{k}_t\|_{L^2_x}^{2} + \frac{1}{2} \int_{0}^{t} \| \nabla U^{k}_s\|_{L^2_x}^{2}  \dd s
&\leq C_{N, A, \epsilon} \int_{0}^{t}  \| U^{k-1}_s \|_{L^2_x}^{2}  \dd s + 2  \int_{0}^{t} \| U^{k}_s\|_{L^2_x}^{2} \dd s \\
&\quad+ C_{N, A, \epsilon} \int_{0}^{t}(  | \bar{X}^{i,k}_{s}- \bar{X}^{i,k-1}_{s} |^{2} + | \bar{V}^{i,k}_{s}- \bar{V}^{i,k-1}_{s} |^{2} )\dd s.
\end{split}
\end{equation}
We also observe that
\begin{align}
\label{ODE}
&\sup_{r\in [0,t]} |\bar{X}^{i,k+1}_r-\bar{X}^{i,k}_{r}|^{2} +  \sup_{r\in [0,t]} |\bar{V}^{i,k+1}_{r}-\bar{V}^{i,k}_{r}|^{2} \nonumber\\
&\quad \leq C_{T,A, \epsilon}  \bigg( \int_{0}^{t} \|U_{s}^{k-1}\|_{L^2_{x}}^{2} \dd s
 +   \int_{0}^{t} \sup_{r\in [0,s]} |\bar{X}_{r}^{i,k}- \bar{X}_{r}^{i,k-1}|^{2} \dd s + \int_{0}^{t} \sup_{r\in [0,s]} |\bar{V}_{r}^{i,k}- \bar{V}^{i,k-1}_{r}|^{2}  \dd s\bigg).
\end{align}

 From \eqref{U}, \eqref{ODE}  and Gr\"onwall's Lemma for sequences (see \cite[Lemma 3]{Boudin2009}), we conclude that $(\bar{u}^{k}, \bar{X}^{1,k}, \bar{V}^{1,k},\dots, \bar{X}^{N,k}, \bar{V}^{N,k})_{k\in \N^*}$ is  a Cauchy sequence in $\mathcal{Y}$ and the limit, that we denote by $(u^{\epsilon}, X^{1,\epsilon}, V^{1,\epsilon},\dots, X^{N,\epsilon}, V^{N,\epsilon})$, verifies~\eqref{PartEbi} in a weak and mild sense. Observe that this limit holds almost surely, since it holds for all $\omega \in \Omega_{N}$. In particular, the limit $(u^{\epsilon}, X^{1,\epsilon}, V^{1,\epsilon},\dots, X^{N,\epsilon}, V^{N,\epsilon})$ is measurable and $\mathcal{F}_{t}$-adapted.

\medskip

\textbf{Uniform estimates in $\epsilon$.}
We first bound
$(u^{\epsilon}, X^{1,\epsilon}, V^{1,\epsilon},\dots, X^{N,\epsilon}, V^{N,\epsilon})_{\epsilon>0}$ in $\mathcal{Y}$. We have
\begin{align}
\label{bouinPartibis}
\frac{1}{2}\| u^{\epsilon}_t\|_{L^2_x}^{2} +  \int_{0}^{t} \| \nabla u^{\epsilon}_s\|_{L^2_x}^{2}  \dd s
&= \frac{1}{2} \| u^{\circ,N}\|_{L^2_x}^{2}
+ \int_{0}^{t}  \int_{\T^d}  \nabla u_{s}^{\epsilon} \, (u_{s}^{\epsilon}\ast \rho^\epsilon)\,\bchi(u_{s}^{\epsilon}) \dd x \dd s \nonumber\\
&\quad + \int_{0}^{t}  \int_{\T^d} u_{s}^{\epsilon}\, \frac{1}{N} \sum_{i=1}^{N} \big(\bchi\big(u_{s}^{\epsilon} \ast \rho^\epsilon(X_{s}^{i,\epsilon})\big)-V_{s}^{i,\epsilon}\big) \delta_{X^{i,\epsilon}_s}^{N}(x) \dd x \dd s \nonumber\\
&\eqqcolon \frac{1}{2}\| u^{\circ,N}\|_{L^2_x}^{2} + I_{1} + I_{2} .
\end{align}
By using the H\"older and Young inequalities, it follows that
\begin{equation}\label{I1Pbis}
|I_{1}| \leq \frac{1}{2} \int_{0}^{t} \| \nabla u^{\epsilon}_s\|_{L^2_x}^{2} \dd s + C_{A} \int_{0}^{t} \| u^{\epsilon}_s\|_{L^2_x}^{2} \dd s.
\end{equation}
Similarly,
\[
|I_{2}| \leq  \int_{0}^{t} \| u^{\epsilon}_s\|_{L^2_x}^{2} \dd s +
C \int_{0}^{t}
\frac{1}{N} \sum_{i=1}^{N} \big| \bchi\big(u_{s}^{\epsilon} \ast  \rho^\epsilon (X_{s}^{i,\epsilon})\big)-V_{s}^{i,\epsilon} \big|^{2} \, \| \delta_{X^{i,\epsilon}_s}^{N}\|_{L^2_x}^{2} \dd s.
\]
Since $\|\delta_{X^{i,\epsilon}_s}^{N}\|_{L^2_x}^{2}$
is a constant which depends only on $N$, one has
\begin{equation}\label{I2Pbis}
|I_{2}| \leq  \int_{0}^{t} \| u^{\epsilon}_s\|_{L^2_x}^{2} \dd s +
C_{N} \int_{0}^{t} ( C_{A}  +  \sup_{r\in [0,s]} \max_{i\in \{1,\dots, N\}}|V_{r}^{i,\epsilon}|^{2} ) \dd s.
\end{equation}
From \eqref{bouinPartibis}, \eqref{I1Pbis} and \eqref{I2Pbis}, it follows that
\begin{align}\label{comp}
\frac{1}{2}\| u^{\epsilon}_t\|_{L^2_x}^{2} &+ \frac{1}{2} \int_{0}^{t} \| \nabla u^{\epsilon}_s\|_{L^2_x}^{2}  \dd s \nonumber\\
&\leq \frac{1}{2} \| u^{\circ,N}\|_{L^2_x}^{2}
+ C_{A} \int_{0}^{t} \| u^{\epsilon}_s\|_{L^2_x}^{2} \dd s +
C_{N} \int_{0}^{t} ( C_{A}  +  \sup_{r\in [0,s]} \max_{i\in \{1,\dots, N\}} |V_{r}^{i,\epsilon}|^{2} ) \dd s.
\end{align}
Now, by Gr\"onwall's Lemma, it is not difficult to see that
\begin{equation}\label{comp2}
 \sup_{r\in [0,t]} |X^{i,\epsilon}_r| +  \sup_{r\in [0,t]}\  |V^{i,\epsilon}_{r}| \leq  C_T\Big(|X_0^{i,N}|+  |V_0^{i,N}| + C_{A} +   \sup_{r\in [0,t]} |B_{r}^{i}| \Big).
\end{equation}
Plugging this estimate into \eqref{comp} and using again Gr\"onwall's Lemma, we infer that $u^\epsilon_t$ is bounded in $L^{\infty}([0,T]; L^{2}(\T^d)^d) \cap  L^{2}([0,T]; H^{1}(\T^d)^d)$.

Next, we observe that by using the bound \eqref{comp2} and the H\"older continuity of Brownian motion,
\[
|X^{i,\epsilon}_t- X^{i,\epsilon}_s| +
|V^{i,\epsilon}_t- V^{i,\epsilon}_s|
\leq C_{T,A} |t-s|  +  C_{T,\omega} |t-s|^{1/2 - \eta},
\]
for $ \eta < 1/2$. Hence, we can apply the Arzel\`a-Ascoli Theorem and conclude that there exists a subsequence of $(X^{1,\epsilon}, V^{1,\epsilon},\dots, X^{N,\epsilon}, V^{N,\epsilon})$, relabeled the same, that converges uniformly towards $(X^{1,N}, V^{1,N},\dots,X^{N,N}, V^{N,N})$. Recall here that this is for fixed $\omega\in \Omega_{N}$.

By writing the mild formulation for $u^\epsilon_t$, we can bound,  for $ p > d $ and $\gamma \in (\frac{d}{p},1)$,
\begin{align*}
\| u^{\epsilon}_t\|_{\gamma, p}
&\leq C\| u^{\circ,N}\|_{\gamma, p} + C_{A}\int_{0}^{t} \frac{1}{(t-s)^{(1+\gamma)/2}}
\| u_{s}^{\epsilon}\|_{L^p_x}  \dd s + \int_{0}^{t} \max_{i\in \{1,\dots, N\}} \big(C_{A} +   \sup_{r\in [0,s]} |V^{i,\epsilon}_{r}| \big)
\|\delta_{X^{i,\epsilon}_s}^{N}\|_{\gamma, p} \dd s .
\end{align*}
Since  $\|\delta_{X^{i,\epsilon}_s}^{N}\|_{\gamma, p} $ is bounded by a constant which depend only on $N$, it follows from Gr\"onwall's Lemma that
\begin{equation}\label{comp3}
 \sup_{[0,T]}\| u_{t}^{\epsilon}\|_{\gamma, p} \leq C_{T, A, N}.
\end{equation}

Finally we  estimate $ \partial_{t} u^{\epsilon}_t $. Let  $V = \{ v \in H^1(\T^d): \dive v =0\} $ and $V'$ denote its dual.
From~\eqref{PartEbi}, we have that
\begin{align*} \| \partial_{t} u^{\epsilon}_t \|_{V'}& \leq
\| \nabla u^{\epsilon}_t\|_{L^2_x} +  \|   u^{\epsilon}_t \bchi(u^{\epsilon}_t) \|_{L^2_x} + \max_{i\in \{1,\dots, N\}} (C_{A} + \sup_{r\in [0,t]}\  |V^{i,\epsilon}_{r}|) \|\delta_{X^{i,\epsilon}_t}^{N} \|_{L^2_x}\\
 & \leq \| \nabla u^{\epsilon}_t\|_{L^2_x} + C_A \|   u^{\epsilon}_t     \|_{L^2_x} +
 C_N(C_{A} + \max_{i\in \{1,\dots, N\}} \sup_{r\in [0,t]} |V^{i,\epsilon}_{r}|) .
\end{align*}
The previous estimates give us that
$ \partial_{t} u^{\epsilon}_t $ is uniformly bounded in $L^2([0,T];V')$.

As $H^\gamma_p(\T^d)$ is compactly embedded into $H^{\gamma-\eta}_p(\T^d)$ for some small $\eta< \gamma $ such that $ \gamma - \eta \in (\frac{d}{p},1)$, the Aubin-Lions Lemma allows to conclude that
there exists a subsequence of $u^\epsilon$, relabeled the same, that converges strongly to $u^N$ in $\mathcal{C}([0,T];H^{\gamma-\eta}_p(\T^d)^d)
\cap L^{2}([0,T];H^{1-\eta}(\T^d)^d)$; moreover $u^{\epsilon}$ converges weakly in  $L^2( [0,T]; H^1(\T^d)^d)$ and weakly-$\ast$ in $L^\infty([0,T];H^{\gamma}_p(\T^d)^d)$.
By Sobolev embedding, we also have that $u^\epsilon$ converges strongly in $\mathcal{C}([0,T]; \mathcal{C}_{b}^{\gamma-\eta-d/p}(\T^d)^d)$.

To pass to the limit in \eqref{PartEbi} as $ \epsilon$ goes to zero, we only need to check the convergence of the terms involving the convolution. By writing
\[u_{t}^{\epsilon} \ast \rho^\epsilon - u^N_t =
(u_{t}^{\epsilon} -u^N_t)\ast \rho^\epsilon + \rho^\epsilon \ast u_t^N - u_t^N,
\]
we see that, by the H\"older regularity of $u_t^N$, it follows that $\rho^\epsilon \ast u_t^N - u_t^N$ goes to $0$ in $\mathcal{C}([0,T];\mathcal{C}(\T^d)^d)$, and by Young's inequality for convolution, $(u_{t}^{\epsilon} -u^N_t)\ast \rho^\epsilon $ also converges to $0$ in $\mathcal{C}([0,T];\mathcal{C}(\T^d)^d)$.

Now, we  have
\begin{align*}
\bchi\big(u_{t}^{\epsilon} \ast \rho^\epsilon(X^{i,\epsilon}_t)\big)- \bchi\big(u_{t}^{N}(X^{i,N}_t)\big)
&= \Big( \bchi\big(u_{t}^{\epsilon} \ast \rho^\epsilon(X^{i,\epsilon}_t)\big)
 - \bchi\big(u_{t}^{\epsilon} \ast \rho^\epsilon(X^{i,N}_t)\big) \Big) \\
&\quad + \Big( \bchi\big(u_{t}^{\epsilon} \ast \rho^\epsilon(X^{i,N}_t)\big) - \bchi\big(u_{t}^{N}(X^{i,N}_t)\big) \Big) \\
 &\eqqcolon K_{1} + K_{2}.
\end{align*}
Since $\bchi$ is Lipschitz continuous, we deduce that
 \[
 |K_{2}|\leq C_{A} |(u_{t}^{\epsilon} \ast \rho^\epsilon)-u_{t}^{N})(X^{i,N}_t) | \rightarrow 0 .
 \]
Finally, by observing that the H\"older norm of $u_{t}^{\epsilon} \ast \rho^\epsilon$
 remains uniformly bounded, we are led to
 \begin{align*}
 |K_{1}| & \leq C_{A} |(u_{t}^{\epsilon} \ast \rho^\epsilon)(X^{i,\epsilon}_t)
 -( u_{t}^{\epsilon} \ast \rho^\epsilon)(X^{i,N}_t) |
 \\
&\leq C_{A} |X^{i,\epsilon}_t- X^{i,N}_t|^{\gamma-\eta- d/p}
 \end{align*}
which goes to zero.  By uniqueness in the system~\eqref{Part}, we conclude that $(u^{\epsilon}, X^{1,\epsilon}, V^{1,\epsilon},\dots, X^{N,\epsilon}, V^{N,\epsilon})$ converges to $(u^N, X^{1,N}, V^{1,N},\dots, X^{N,N}, V^{N,N})$ as $\epsilon \to 0$ (without passing through any subsequence). This convergence holds almost surely,
 thus $(u^N, X^{1,N}, V^{1,N},\dots, X^{N,N}, V^{N,N})$ remains $\mathcal{F}_{t}$-adapted.

%

\providecommand{\bysame}{\leavevmode\hbox to3em{\hrulefill}\thinspace}

\end{document}